\documentclass{article}
\usepackage{PRIMEarxiv}

\usepackage{amsmath,amssymb,amsfonts,amsthm,bbm,mathrsfs,verbatim,soul} 
\usepackage[misc]{ifsym}
\usepackage{bbm}
\usepackage{graphics}      
\usepackage{graphicx}
\usepackage{subcaption}
\usepackage{float}
\usepackage{dsfont}
\usepackage{array}
\usepackage{hyperref} 
\usepackage{fancyhdr}       
\usepackage{mathtools}
\usepackage{blkarray}
\usepackage{makecell}
\usepackage{color}                    
\usepackage{hyperref}                 
\usepackage{makeidx}
\usepackage{caption}
\usepackage{booktabs} 
\usepackage{siunitx}  
\usepackage{bbm}
\usepackage{algorithm}
\usepackage{algpseudocode}
\usepackage{authblk}
\usepackage{amsmath} 
\usepackage{booktabs}
\usepackage{multirow} 
\DeclarePairedDelimiter\abs{\lvert}{\rvert}
\DeclarePairedDelimiter\norm{\lVert}{\rVert}

\DeclareUnicodeCharacter{02B9}{'}

\newtheorem{theorem}{Theorem}[section]
\newtheorem{proposition}[theorem]{Proposition}

\newtheorem{lemma}[theorem]{Lemma}
\theoremstyle{definition}
\newtheorem{remark}[theorem]{Remark}
\newtheorem{definition}[theorem]{Definition}

\usepackage{neuralnetwork}

\usepackage{tikz}
\usepackage{listofitems} 
\usetikzlibrary{arrows.meta} 
\usepackage[outline]{contour} 

\usetikzlibrary{positioning, calc, arrows.meta}

\tikzset{>=latex} 
\usepackage{xcolor}
\colorlet{myred}{red!80!black}
\colorlet{myblue}{blue!80!black}
\colorlet{mygreen}{green!60!black}
\colorlet{myorange}{orange!70!red!60!black}
\colorlet{mydarkred}{red!30!black}
\colorlet{mydarkblue}{blue!40!black}
\colorlet{mydarkgreen}{green!30!black}
\tikzstyle{node}=[thick,circle,draw=myblue,minimum size=22,inner sep=0.5,outer sep=0.6]
\tikzstyle{node in}=[node,green!20!black,draw=mygreen!30!black,fill=mygreen!25]
\tikzstyle{node hidden}=[node,blue!20!black,draw=myblue!30!black,fill=myblue!20]
\tikzstyle{node convol}=[node,orange!20!black,draw=myorange!30!black,fill=myorange!20]
\tikzstyle{node out}=[node,red!20!black,draw=myred!30!black,fill=myred!20]
\tikzstyle{connect}=[thick,mydarkblue] 
\tikzstyle{connect arrow}=[-{Latex[length=4,width=3.5]},thick, mydarkblue,shorten <=0.5,shorten >=1]
\tikzset{ 
	node 1/.style={node in},
	node 2/.style={node hidden},
	node 3/.style={node out},
}

\def\R{{\mathbb R}}

\newcommand\grad[1]{\nabla \, #1}

\begin{document}
\graphicspath{{Figures/}}
\title{Fredholm Neural Networks for inverse problems in elliptic PDEs}

\author[1, \thanks{Work done while at:
Department of Electrical Engineering and Information Technologies, \textit{University of Naples “Federico II”}, Naples, Italy}]{Kyriakos C. Georgiou}
\author[2,\thanks{{Corresponding author: \texttt{constantinos.siettos@unina.it}}}]{Constantinos Siettos}
\author[3,\thanks{{Corresponding author: \texttt{ayannaco@aueb.gr}}} 
]{Athanasios N.\ Yannacopoulos}

\affil[1]{Division of Applied Mathematics, \textit{Brown University}, Providence, USA}
\affil[2]{Dipartimento di Matematica e Applicazioni  ``Renato Caccioppoli'', \textit{University of Naples “Federico II”}, Naples, Italy }
\affil[3]{Department of Statistics and Stochastic Modelling and Applications Laboratory, \textit{Athens University of Economics and Business}, Athens, Greece }


\date{}
\maketitle 
\begin{abstract} 
Building on our previous work on Fredholm Neural Networks (Fredholm NNs/ FNNs) for solving integral equations, we extend the framework to inverse problems for linear and nonlinear elliptic partial differential equations. The proposed scheme consists of a custom-designed deep neural network (DNN) in which the number of layers, weights, biases and hyperparameters are computed in an explainable manner based on a fixed-point scheme, and we therefore refer to this as the Potential Fredholm Neural Network (PFNN). We first build the PFNN as a method for solving the forward problem, showing that this approach ensures both a  high accuracy and explainability, achieving small errors in the interior of the domain, and near machine-precision on the boundary. We then use this approach to solve inverse problems for elliptic PDEs, and provide a rigorous proof for the consistency of the scheme and error bounds for both the interior and boundary of the domain, tied directly to the architecture of the PFNN. In particular, we show that these error bounds depend on the approximation of the boundary function and the integral discretization scheme, both of which directly correspond to components of the Fredholm NN architecture. In this way, we construct an explainable scheme that provides accurate solutions to the inverse problems, whilst still explicitly respecting the boundary conditions, due to the architecture of the PFNN. 
We assess the performance of the proposed scheme for  linear and semi-linear elliptic PDEs in two and three dimensions.  
\end{abstract}	


\section{Introduction}

In recent years, both theoretical and technological advancements have laid the ground for major advancements in the field of scientific machine learning. Both forward and inverse problems in dynamical systems have been extensively studied, across multiple applications and domains. In particular, the introduction of PINNs \cite{raissi2019physics} and neural operators such as DeepONet \cite{lu2021learning}, as well as Fourier neural networks \cite{li2020Fourier}, graph-based neural operators \cite{li2020multipole,kovachki2023neural}, random feature models \cite{nelsen2021random}, and RandONets \cite{fabiani2024randonet} provide an arsenal of such scientific machine learning methods (for a review see, e.g., \cite{karniadakis2021physics, fabiani2024randonet}). 
Furthermore, other techniques that bridge methods from numerical analysis  with machine learning, aiming towards more computationaly efficient and explainable schemes have also been developed, such as Legendre neural networks (\cite{mall2016application}),  simulation-based techniques (\cite{frey2022deep}), Deep Ritz-Finite element methods (\cite{grekas2025deep}) and recently Kolmogorov-Arnold neural networks \cite{liu2024kan, toscano2025kkans}. These approaches have been applied to the solution of the forward problem for Ordinary/Algebraic differential equations (ODEs, DAEs) \cite{lu2021deepxde,ji2021stiff,kim2021stiff,de2022physics,fabiani2023parsimonious,koenig2024kan}, Partial Differential Equations (PDEs) \cite{han2018solving,raissi2019physics,pang2019neural,lee2020coarse,lu2021deepxde,chen2021solving,dong2021local,calabro2021extreme,fabiani2021numerical,dong2021local,dong2021modified,galaris2022numerical,laghi2023physics,lee2023learning,jacob2024spikans}, as well as to Partial Integro-differential Equations (PIDEs) \cite{zappala2022neural,yuan2022pinn, georgiou2024deep,fabiani2024task}.  
Similarly, seminal work has been done when considering Integral Equations (IEs) such as the Fredholm and Volterra IEs  with applications in various domains (see e.g., \cite{mikhlin2014integral}). Furthermore, IEs can be connected to ODEs and PDEs. In particular, the connection to PDEs is established via the Boundary Integral Equation (BIE) method (see e.g., \cite{kellogg2012foundations,hsiao2008boundary}). This connection has been considered in the context of neural networks in \cite{lin2021binet}. 
More recently, in \cite{meng2024solving} the Boundary Integral Type Deep Operator Network (BI-DeepONet) and Boundary Integral Trigonometric Deep Operator Neural Network (BI-TDONet) are developed, combining the DeepONet framework with the BIE framework to solve PDEs with an optimized training process. In  \cite{effati2012neural, guan2022solving, jafarian2013utilizing, qu2023boundary}, the authors use loss functions based on the residual between the model prediction and integral approximations to estimate the solution to FIEs using the standard learning approach for DNNs. In \cite{zappala2022neural} the authors introduce Neural Integral Equations, a machine learning framework for estimating integral operators. In \cite{solodskikh2023integral}, the authors presented Integral Neural Networks, which were applied to computer vision problems, and were based on the observation that a fully connected neural network layer essentially performs a numerical integration.\par 
Despite the multitude of schemes and applications, the complexity of DNNs is known to give rise to the so called ''black-box'' effect training, meaning that the  rigorous (not just computational) quantification of errors and the sensitivity of the output to the hyperparameters  is a difficult task. In addition, the standard  training approach via loss function optimization (which may even fail in some cases, as shown in \cite{wang2022and}), leads to weights and biases that have practically no explainability. Therefore, very recently, researchers construct neural network architectures that learn by emulating  iterative numerical analysis schemes. For example, based on the concept of Runge-Kutta NNs presented back in the 90s' \cite{rico1992discrete}, in  \cite{guo2022personalized,akrivis2024runge}, the authors construct deep neural networks, based on the Runge-Kutta family of integrators, to learn the solution of ODEs. In \cite{doncevic2024recursively}, the authors propose a recursively recurrent neural network (R2N2) superstructure that can learn numerical fixed point algorithms such as Krylov, Newton-Krylov and Runge-Kutta algorithms. In this line of research, in \cite{kopanivcakova2024deeponet}, the authors have introduced a hybrid scheme combining iterative numerical solvers in Krylov subspaces with preconditioning of DeepONets for the solution of parametrized steady-state linear PDEs. Finally, in \cite{zhang2024blending,kahana2023geometry}, a hybrid approach called HINTS (Hybrid Iterative Numerical Transferable Solver) for the solution of PDEs has been proposed. This approach integrates relaxation methods such as Jacobi, Gauss-Seidel, and multigrid techniques, with DeepONets, to enhance the convergence behavior across eigenmodes.
Of course, robust error analysis is of critical importance and an open challenging problem, in all aforementioned schemes. One of the most prevalent issues for the solution of PDEs, is the fact that machine learning methods, such as PINNs impose boundary conditions as constraints in the optimization problem. However, this can exhibit large errors near and on the boundaries (e.g., \cite{georgiou2024deep} studies such errors when considering Kolmogorov backward equations in credit risk). More generally, this issue exists in problems with more complex dynamics, as displayed in \cite{krishnapriyan2021characterizing}. Recently, important work has been done regarding error analysis for PINNs (and related) models. For example, in \cite{chantada2024exact} exact and approximate error bounds are given for PINNs solutions for non-linear first-order ODEs. In \cite{hillebrecht2022certified, hillebrecht2023rigorous} the authors provide upper bounds for the error incurred by PINNs that depend on the dynamics of the underlying PDE.  Such analysis has been recently done for operator learning (\cite{de2022generic}) and also combined with Runge-Kutta methods (\cite{stiasny2024error}). Additional mathematically rigorous treatments of errors for some types of PINNs can be found e.g., in \cite{de2022error, doumeche2023convergence}.  We also note that recent work provides rigorous definitions and analysis of stability and convergence for PINNs (\cite{gazoulis2023stability}) for PDEs. However, model architecture and hyperparameter tuning still relies mainly on ad hoc/heuristic approaches (see for example \cite{yu2020hyper,dong2021modified,galaris2022numerical,fabiani2023parsimonious,bolager2024sampling}).

Very recently in \cite{georgiou2024fredholm}, we introduced Fredholm NNs, integrating DNNs with the theory of fixed point schemes in the Banach space, for the solution of Fredholm IEs. Here, generalizing this explainable framework, we utilize the potential theory and the Boundary Integral Equation (BIE) method to construct explainable DNNs for the solution of forward and inverse problems for elliptic PDEs, which occur in a vast range of problems and applications, such as in fluid mechanics \cite{gilbarg1977elliptic}, optimal control \cite{casas2014optimal}, thermodynamics and quantum mechanics \cite{poppenberg2002existence}. We refer to this construction as the Potential Fredholm Neural Network (PFNN). In this work, we will focus on the use of PFNNs to solve inverse problems in elliptic PDEs. To do this, we provide details regarding the construction of the PFNN and rigorous convergence and error analysis results for PFNNs in the interior of the PDE domain as well as on the boundary. This analysis relies on the approximation of the boundary function and the integral discretization scheme, which directly correspond to specific components of the PFNN architecture. Using these results, we then apply PFNNs to inverse source problems for elliptic PDEs, by considering models for the unknown source terms which are ``fed'' into the PFNN.

The structure of the paper is as follows: in section \ref{preliminary}, we recall the main results pertaining to Potential Theory for PDEs and Fredholm NNs. In section \ref{main-section} we develop and present the proposed methodology. Specifically, we present the inverse source problem, we construct the Potential Fredholm Neural Networks for elliptic PDEs, and we prove a Universal Approximation Theorem, based on explicit error bounds deduced for the PFNNs. 
In section \ref{numerics-section}, we provide detailed numerical examples for 2D and 3D PDes and assess the performance of the proposed scheme, with detailing the errors in the interior and the on the boundary of the PDE domains. We conclude in section \ref{conclusion}. For completeness, we also present examples for forward problems in Appendix \ref{appendixx}.

\section{Mathematical Preliminaries}\label{preliminary}
We first present the methodology by detailing a rigorous construction of PFNNs using the Poisson and Helmholtz PDEs as representative paradigms. This will construct the basis for presenting the main focus of our work, which is the use of the PFNN for the solution of the inverse problem for elliptic PDEs. For clarity, we also provide a brief reminder regarding Potential Theory and Fredholm Neural Networks, as introduced in \cite{georgiou2024fredholm}. 
In what follows, we will be considering the Banach space $\mathcal{X}:= \mathcal{C}\big(\Omega, \mathbb{R} \big)$, endowed with the norm $\norm{f}:=\norm{f}_{C(\Omega; {\mathbb R})} = \sup_{x \in \Omega} |f(x)|$, with $x \in \mathbb{R}^d$ and $\Omega \subset \mathbb{R}^d$ a bounded domain, and the Hilbert space ${\mathcal H}:=L^{2}(\Omega ; \mathbb{R})$ endowed with the inner product $\langle f, g \rangle =\int_{\Omega} f(x) g(x) dx$. We remark that the results are generalizable to other spaces, e.g., $\mathcal{C}(\Omega, \mathbb{R}^n)$ or Lebesgue spaces.

\subsection{Fredholm Neural Networks}
In \cite{georgiou2024fredholm}, we used the above formulation to construct Fredholm NNs for the solution of the Fredholm Integral Equation which produces the boundary function $\beta(x), x \in \partial \Omega$. 
For completeness, here we present the Fredholm Neural Network that is used for the solution of the forward problems for (\ref{BIE}). 

Consider an integral operator $\mathcal{T}: \mathcal{X} \rightarrow \mathcal{X}$,  acting on $f$ as:
\begin{eqnarray}\label{OPERATOR}
    \mathcal{T}f(x)  =  \int_{\mathcal{A}}K(x,y)f(y)dy, \quad \forall \, x \in \Omega,
\end{eqnarray} 
for $\mathcal{A} \subset \R^p, p \leq d$, which is non-expansive, i.e.,:
\begin{eqnarray}
    \left\|\mathcal{T}f_1 - \mathcal{T}f_2\right\|_{\mathcal{A}}\leq q \left\|f_1 - f_2\right\|_{\mathcal{A}},
\end{eqnarray}
for some $q \leq 1$, and the corresponding Fredholm Integral Equation:
\begin{equation}\label{ie}
    f(x) = g(x) + \mathcal{T}f(x).
\end{equation}

\begin{definition}[Discretized Krasnosel'skii-Mann operator]
Consider a function $f$ satisfying the FIE \eqref{ie}. Consider a $y-$grid $Y = \{y_1, \dots y_N \} \subset {\cal A}$, $|Z| = N$, with quadrature volume per grid-point $\Delta y = \frac{\abs{A}}{N}$ and such that $x \in Y$. Then, we define the following discrete operator $\mathcal{P}_N$ representing a single iteration of the KM-algorithm using the discretized version of the integral operator 
$\mathcal{T}$
(with fixed $g :\mathbb{R} \rightarrow \mathbb{R}$), by:
\begin{flalign}\label{disc-op}
	(\mathcal{P}_{N, \kappa} f) (x) := (1-\kappa)f(x) &+ \kappa \big(g(x) + \sum_{i =1}^N K(x,y_j)f(y_j) \Delta y \big) \notag \\ &= \kappa g(x) + \sum_{j =1}^N f(y_j)\Big( K(x,y_{j})\kappa \Delta y + (1-\kappa) \mathbbm{1}_{\{y_j = x\}}\Big).
\end{flalign}
\end{definition}

Hence, we can obtain a the {\it $M$-layer} approximation of the solution to the FIE as the following representation arising from applying the discretized KM operator $\mathcal{P}_{N, \kappa}$ repeatedly, for the given choice of $g(x)$, $M$ times:
\begin{align} \label{m-operator}
f^{(M)}(x) = (\mathcal{P}_N^{(M)} g)(x) := \mathcal{P}_{N, \kappa_{M}} \circ (\mathcal{P}_{N, \kappa_{M-1}} \circ (\dots \circ (\mathcal{P}_{N, \kappa_{1}}g)(x))) \end{align}

\begin{proposition}
[Fredholm Neural Network construction \textcolor{teal}{\cite{georgiou2024fredholm}}]\label{DNN_construction}
Consider the integral operator $\mathcal{T}: \mathcal{X} \rightarrow \mathcal{X}$, as given in \eqref{OPERATOR} and the discretization grid $Y$, as well as and a constant $\kappa_n = \kappa$, $\kappa \in (0,1)$, for all $n \in {\mathbb N}$. Then, the solution to the linear FIE \eqref{ie} can be approximated by a deep neural network with input $x \in {\cal A}$, 
$M$ hidden layers, a linear activation function and a single output node, where the weights and biases are given by:
\begin{flalign}
    W_1 = \big(\begin{array}{ccc}
		\kappa g(y_1), \dots, \kappa g(y_{N})
	\end{array}\big)^{\top}, \,\,\,\,\    b_1 = \left(\begin{array}{ccc}
		0, 0, \dots, 0
	\end{array}\right)^{\top}
 \end{flalign} 
for the first hidden layer,  
\begin{eqnarray}	\label{inner-weight}
W_m=
\left(\begin{array}{cccc}
	K_D\left(y_1\right) & {K}\left(y_1, y_2\right)\kappa \Delta y & \cdots & {K}\left(y_1, y_{N}\right)\kappa \Delta y \\
 {K}\left(y_2, y_1\right)\kappa \Delta y & K_D\left(y_2\right) & \cdots & {K}\left(y_2, y_{N}\right)\kappa \Delta y \\
	\vdots & \vdots & \ddots & \vdots \\
	\vdots & \vdots & \vdots & \vdots \\
	{K}\left(y_{N}, y_1\right)\kappa \Delta y & {K}\left(y_{N}, y_2\right)\kappa \Delta y & \cdots & K_D\left(y_{N}\right) 
\end{array}\right),
\end{eqnarray}
and
\begin{eqnarray}
	b_m=\left(\begin{array}{ccc}
		\kappa g(y_1), \dots, \kappa g(y_{N})
	\end{array}\right)^{\top},
\end{eqnarray}
for hidden layers $m= 2, \dots, M-1$, where $K_D\left(y\right) := {K}\left(y, y\right)\kappa \Delta y + (1-\kappa)$, and:
\begin{flalign} \label{outer-weight}
	\begin{gathered}
		W_M=\big(\begin{array}{ccc}
			K(y_1, x)\kappa \Delta y, \dots, K(y_{i-1},x)\kappa \Delta y, K_D(x), K(y_{i+1}, x)\kappa \Delta y, \dots, K(y_{N}, x)\kappa \Delta y
		\end{array}\big)^{\top},
	\end{gathered}
\end{flalign}
$b_M =\big(\kappa g(x) \big)$, for the final layer, assuming $y_i = x$. 
\par Then, letting $\mathcal{N}_m$ be the map defined by $\mathcal{N}_m(x) = W_m x + b_m$, for $m=1,\dots, M$, the output of the neural network is given by:
\begin{equation}
    \hat{f}(x)
    = \Big(
        \mathcal{N}_M \circ \cdots \circ \mathcal{N}_1
      \Big)(x).
\end{equation}
\end{proposition}

This is the Fredholm Neural Network construction that implements the approximation \eqref{m-operator}. Based on the constructions above we can formulate the Universal Approximation Theorem for Fredholm Neural Networks (Fredholm NNs) below. 

\begin{theorem}[Fredholm NNs as universal approximators of linear FIE solutions \cite{georgiou2024fredholm}]\label{FUA}
Consider the linear integral operator, $\mathcal{T} : \mathcal{V} \rightarrow \mathcal{V}$ defined in \eqref{OPERATOR}, where ${\cal V}$ is either ${\cal X}$ or ${\cal H}$.
Consider the space $\mathcal{V'} \subset \mathcal{V}$ defined by $\mathcal{Z'}: = \{ f \in \mathcal{V}: g+ \mathcal{T}f = f, \text{ for some } g, K \}$. Then, for any $f \in {\mathcal Z}'$ and for any given approximation error $\epsilon$, there exists a fully connected Fredholm NN (denoted by $\mathcal{F}_M(x;g, K)$) with $M$ hidden layers as described in Proposition \ref{DNN_construction}, and activation function given by:
\begin{eqnarray}
 \sigma(x) =
\begin{cases}
   \kappa g(x), \text{ for layer }  m= 1, \\
    x, \text { for layer }  m \geq 2,   
\end{cases}
\end{eqnarray}
 such that $\|f(x) - \mathcal{F}_M(x;g, K)\|_{L^2(\Omega)} \leq \epsilon$. 
\end{theorem}
Note that hereinafter we will be using $\mathcal{F}_M(\cdot;g,K)$ to denote the Fredholm Neural Network, with the given choice for $g$ and $K$.



\color{black}

\subsection{Potential Theory}\label{sec:potential}
We begin by recalling the following theorem concerning the Poisson equation:


\begin{theorem}[\cite{folland2020introduction}]\label{BIM}
Consider $x\in \Omega \subset \mathbb{R}^d$ and the linear Poisson equation for $u(x)$:
    \begin{eqnarray}\label{poisson}
        \begin{cases}\Delta u(x) = \psi(x), & \text { for } x \in \Omega \\ u(x)= f(x) & \text { for } {x} \in \partial \Omega, \end{cases}
    \end{eqnarray}
with the smooth boundary $\partial \Omega$ is $C^2$. Its solution can be written via the double layer boundary integral given by:
\begin{eqnarray}\label{potential-integral}
   u(x) =  \int_{\partial \Omega} \beta(y) \frac{\partial \Phi}{\partial n_{y}}(x, y) d \sigma_{y} + \int_{\Omega} \Phi(x,y) \psi(y) d y , \,\, x \in \Omega,
\end{eqnarray}
where the fundamental solution is given by $\Phi(x,y) = \frac{1}{(d-2)\omega_d|x-y|^{d-2}}$,  for $d>2$ (and a function featuring a logarithmic singularity for $d=2$) with $\omega_d$  the volume of the $d-$dimensional unit sphere, $n_y$ is the outward pointing normal vector to $y$, $\sigma_y$ is the surface element at point $y\in \partial \Omega$, and $\frac{\partial \Phi}{\partial n_{y}} = n_y \cdot \grad_{ y}{\Phi}$. It can be shown that the following limit holds, as we approach the boundary: 
\begin{eqnarray}\label{BIE-limit}
\lim _{\substack{x \rightarrow x^{\star} \\ x \in \Omega}}   \int_{\partial \Omega} \beta(y) \frac{\partial \Phi}{\partial n_y}(x, y) d \sigma_{y} =u\left({x}^{\star}\right) - \frac{1}{2} \beta(x^{\star}), \quad x^{\star} \in \partial \Omega.
\end{eqnarray}
Hence, the function $\beta({x}^{\star})$, defined on the boundary, must satisfy the Boundary Integral Equation (BIE):
\begin{eqnarray}\label{BIE}
    \beta({x}^{\star}) = 2 \Big(f(x^{\star}) - \int_{\Omega
    } \Phi(x^*,y) \psi(y) dy \Big) - 2 \int_{\partial \Omega} \beta(y) \frac{\partial \Phi}{\partial n_{y}}(x^{\star}, y) d \sigma_{y}, \,\,\ x^{\star} \in \partial \Omega.
\end{eqnarray}
\end{theorem}

The above theorem can be extended for other linear elliptic equations, as well as other boundary conditions, in order to obtain integral representations of their solutions  involving the boundary data. This has been a very successful approach in the numerical approximation of linear boundary value problems, called the Boundary Integral Method (BIM). Clearly, the kernel $\Phi$ is related to the differential operator defining the PDE (or rather its inverse), and it should be noted that it is in general not possible to obtain the fundamental solution analytically for general elliptic equations.

Fredholm Neural Networks applied to elliptic PDEs are constructed using potential theory. Their explainability is achieved by connecting the parameters of the network to the double layer potential, as well as by accounting for the behavior of the double layer potential integral in the limit as we approach the boundary, which allows us to obtain a ''smooth'' version of the integral representation, i.e., one that does not exhibit the jump relation, as given in Theorem \ref{BIM}, during the numerical implementation of the scheme. This approach was first applied for the Laplace PDE in \cite{georgiou2024fredholm}.

\section{Methodology: The Potential Fredholm Neural Network for Inverse Problems} \label{main-section}
One of the most important applications of Fredholm NNs is in the numerical approximation of elliptic PDEs, by utilizing the well known fact that under certain conditions, the inverse of an elliptic differential operator is a Fredholm type integral operator (see classic references such as \cite{mclean2000strongly,kress2014linear,evans2010partial}), thus leading to an equivalent representation of the PDE as an integral equation of the Fredholm type. This observation is also valid for boundary value problems (BVPs), reducing the PDE to an appropriate integral formulation (see section \ref{sec:potential}) which importantly ``hard-wires'' the boundary conditions in the formulation. In our previous work \cite{georgiou2024fredholm}, through a simple example based on the Laplace equation,  we have simply scratched the surface of the possibility of combining the boundary integral (BI) formulation of the boundary value problem with the Fredholm NNs (introduced in  \cite{georgiou2024fredholm}) for the solution of elliptic BVPs, thus yielding a powerful explainable Scientific Machine Learning (SciML) tool for the approximation of the solutions, with excellent reproducibility properties for the boundary condition. Here, we return to this idea, by extending it to other types of elliptic linear and nonlinear PDEs and using the framework to solve the inverse problem (see section \ref{inv-sec}). We study the mathematical properties (e.g. convergence etc) of the proposed approximation scheme, as well as provide a-priori estimates for the hyperparameters of the Fredholm NN. Our results indicate the useful interplay between numerical analysis and ML, and how revisiting new ML techniques under the lens of traditional numerical analysis can provide important insights and explainability to ML methods in the field of scientific computing.

\par We will be considering the inverse source problem for the Poisson, Helmholtz and semi-linear PDEs. In this setting, the inverse problem consists of a known boundary function $f : \partial \Omega \to {\mathbb R}$, as well as a coarse set of data points $\{ u(x_i) \}_{i=1,\cdots, n}$, $x_i \in \Omega$, and looking to approximate the unknown source function $\psi: \Omega \times \mathbb{R} \to {\mathbb R}$, using for a suitable model, represented by $\psi_{\theta}: \Omega \times \mathbb{R} \to {\mathbb R}$ (or $\psi: {\Omega} \rightarrow \mathbb{R}$) with parameters $\theta$ (e.g., a shallow neural network), such that the data $\{{u}(x_i) \}$ and function $f(x)$ satisfy the corresponding PDE. Note that the inverse source problem is well known to be an ill-posed problem, not admitting a unique solution. Indeed, there are infinitely many source terms that can satisfy the PDE, given data from the solution. Therefore, with this perspective, our goal is to learn the source term such that when used in the PDE, the resulting solution closely matches our data, rather than learning the exact source term function per se. 
\par The proposed methodology is based on an extension of the Fredholm Neural Network, which we refer to as the Potential Fredholm Neural Network (PFNN) that approximates the solution of elliptic PDEs, using a neural network construction based on the double layer potential formulation of the solution, as described below. Then, our approach consists of a model for the unknown source term, which is used as an input to the PFNN, and trained to match the solution data. We will see that this forward pass through the PFNN provides a training approach the yields explainable and highly accurate results. 
\par We begin with the mathematical construction of the PFNNs for the classes of elliptic PDEs we will be considering.

\subsection{Construction of the Potential Fredholm Neural Networks}
\subsubsection{The Poisson PDE}\label{section-PFNN-Poisson}

 We begin by considering the Poisson PDE. Regarding this construction, recall that the Fredholm NN approximates the method of successive approximations for the solution of a Fredholm Integral Equation of the second kind. Accoding to the potential theorm for the Poisson PDE (see Theorem \ref{BIM}) the BIE (\ref{BIE}) is of this form and therefore, a Fredholm NN is the first component of the proposed framework, followed by an additional layer that simulates the calculation performed by the new representation of the double layer potential. This is outlined in the following main result:

\begin{proposition}[PFNN construction for the Poisson PDE]\label{prop-poisson}
The Poisson PDE (\ref{poisson}) can be solved using a Fredholm NN, with the $d-$dimensional input $x \in \Omega$, and with $M+1$ hidden layers, 
for which the final and output weights $W_{M+1} \in \mathbb{R}^{N \times N}, W_O \in \mathbb{R}^N$ are given by:
\begin{eqnarray}\label{weights-poisson}
    W_{M+1}= I_{N \times N},
    \,\,\,\,\
    W_{O}= \left(\begin{array}{cccc}
	\mathcal{D} \Phi({x}, {y}_1)\Delta \sigma_{y}, & \mathcal{D} \Phi({x}, {y}_2)\Delta\sigma_{y}, & \dots, & \mathcal{D} \Phi({x}, {y}_N) \Delta \sigma_{y}
\end{array}\right)^{\top},
\end{eqnarray}
where we define the simple operator $\mathcal{D} \Phi({x}, {y}):= \Big(\frac{\partial \Phi}{\partial n_{y}}({x}, {y})- \frac{\partial \Phi}{\partial n_{y}}({x}^{\star}, {y})\Big)$. The corresponding biases $b_{M+1} \in \mathbb{R}^{N}$ and $b_O \in \mathbb{R}$ are given by:
\begin{eqnarray}\label{biases-poisson}
   b_{M+1} = \left(\begin{array}{ccc}
		- \beta({x}^{\star}), \dots, - \beta({x}^{\star})
	\end{array}\right)^{\top}, \,\,\,\  
 b_O= \frac{1}{2} \beta({x}^{\star}) + \int_{\partial \Omega} \beta(y) \frac{\partial \Phi(x^*, y)}{\partial n_y} d\sigma_y + \int_{\Omega} \Phi({x},{y}) \psi(y) dy,
\end{eqnarray}
where $x^* = \text{argmin}_{y \in \partial \Omega} \|x - y\|$ is the unique point on the boundary, such that $x \rightarrow x^*$. 
\end{proposition}

\begin{proof}
In \cite{georgiou2024fredholm}, we have demonstrated that a Fredholm NN with $M$ hidden layers can be constructed to solve the integral equation (\ref{BIE}), producing the estimate $\beta_M(x)$.
We know that the fundamental solution of the Laplace operator satisfies
$\Delta_y \Phi(x,y)=\delta(x-y),$
and is given by:
\begin{equation}
  \Phi(x,y)=
\begin{cases}
\dfrac{1}{2\pi}\log |x-y|, & d=2,\\[1em]
-\dfrac{1}{(d-2)\omega_d |x-y|^{d-2}}, & d\geq 3,
\end{cases}  
\end{equation}
where \(\omega_d\) denotes the surface measure of the unit sphere in
$\mathbb R^d$. The double layer potential representation of the solution (see Theorem \ref{BIM}) is given by:
\begin{equation}
    u(x) = \int_{\partial \Omega} \beta(y) \frac{\partial \Phi(x, y)}{\partial n_y} d\sigma_y+ \int_{\Omega} \Phi(x, y) \psi(y) dy,
\end{equation}
where $\beta$ is a density defined on the boundary $\partial \Omega$, such that:
\begin{equation}
    f(x) = \int_{\partial \Omega} \beta(y) \frac{\partial \Phi(x, y)}{\partial n_y} d\sigma_y + 
    \int_{\Omega} \Phi(x, y) \psi(y) dy + \frac{1}{2} \beta(x), \quad \text{for} \ x \in \partial \Omega,
\end{equation}
which can be written as a Fredholm Integral Equation as:
\begin{flalign}
 \beta(x) = 2 \left(f(x) - \int_{\Omega} \Phi(x, y) \psi(y) dy \right) - 2 \int_{\partial \Omega} \beta(y) \frac{\partial \Phi(x, y)}{\partial n_y} d\sigma_y.
\end{flalign}
It can easily be shown that the kernel is non-expansive (see detailed calculations in section \ref{numerics-section}), as required under the Fredholm NN framework. Solving the BIE is the first step in the process. 


Next, we proceed with the calculation of the double layer potential in order to retrieve the solution of the PDE. Our approach relies on the underlying mathematical computations employing the dynamics of the double layer potential as we approach the boundary. Consider $x \in \Omega$ and $x^* \in \partial \Omega$, such that $x^* = \text{argmin}_{y \in \partial \Omega} \|x - y\|$, and we have:
\begin{equation}\label{res0}
    u(x) = \int_{\partial \Omega} \big( \beta(y) - \beta(x^*) \big) \frac{\partial \Phi(x, y)}{\partial n_y} d\sigma_y + \beta(x^*) + \int_{\Omega} \Phi(x, y) \psi(y) dy,
\end{equation}
where we have used:
\begin{equation}\notag 
    \int_{\partial \Omega} \frac{\partial \Phi(x, y)}{\partial n_y} d\sigma_y = 1, \quad \text{for} \ x \in \Omega.
    \end{equation}
Similarly, for $x^* \in \partial \Omega$:
\begin{equation}
    u(x^*) = f(x^*) = \int_{\partial \Omega} \big(\beta(y) - \beta(x^*) \big) \frac{\partial \Phi(x^*, y)}{\partial n_y} d\sigma_y + \frac{1}{2} \beta(x^*) + \int_{\Omega} \Phi(x^*, y) \psi(y) dy,
\end{equation}
where we have used that:
\begin{equation}\label{important-identity}
    \int_{\partial \Omega} \frac{\partial \Phi(x^*, y)}{\partial n_y} d\sigma_y = \frac{1}{2}, \quad \text{for} \ x^* \in \partial \Omega.
\end{equation}

Adding and subtracting $\frac{\partial \Phi(x^*,y)}{\partial n_y}$ within the first integral of \eqref{res0} and simplifying, we obtain:
\begin{flalign}\label{potential-form-1}
    u(x) = \int_{\partial \Omega} \big( \beta(y) -  \beta(x^*) \big)
    \Big(\frac{\partial \Phi(x, y)}{\partial n_y} & - \frac{\partial \Phi(x^*, y)}{\partial n_y} \Big) d\sigma_y + \frac{1}{2} \beta(x^*) \notag \\  & + \int_{\partial \Omega} \beta(y) \frac{\partial \Phi(x^*, y)}{\partial n_y} d\sigma_y + \int_{\Omega} \Phi(x, y) \psi(y) dy.
\end{flalign}
 The components of the expression can be used to construct the corresponding weights and biases of for the final hidden layer of the Potential Fredolm NN, as given in \eqref{weights-poisson} and \eqref{biases-poisson}. 
\end{proof}

\begin{remark}
    As a specific example, the above can be simplified to:
\begin{equation}\label{potential-form-poisson}
    u(x) = \int_{\partial \Omega} \big( \beta(y) - \beta(x^*) \big) 
    \left( \frac{\partial \Phi(x, y)}{\partial n_y} - \frac{1}{4 \pi} \right) d\sigma_y
    + \frac{1}{2} \beta(x^*) + \int_{\partial \Omega} \frac{1}{4\pi} \beta(y) d\sigma_y + \int_{\Omega} \Phi(x, y) \psi(y) dy,
\end{equation}
in the case of the Poisson PDE in $d=2$, where the fundamental solution is given by $\Phi(x,y) = \frac{1}{2\pi}\ln{|x-y|}$.
\end{remark}

\begin{remark}
  From a mathematical standpoint, this representation allows us to approach the boundary without the abrupt effects of the explicit jump component that the double layer potential exhibits. Instead, we are able to use a consistent representation that holds both in the domain and that identically satisfies the boundary conditions, as required.
\end{remark}

\par The result above is also illustrated in Fig. \ref{fig:LIC-diagram}. 

\begin{figure}[ht]
    \centering
    \includegraphics[ width=0.85\textwidth]{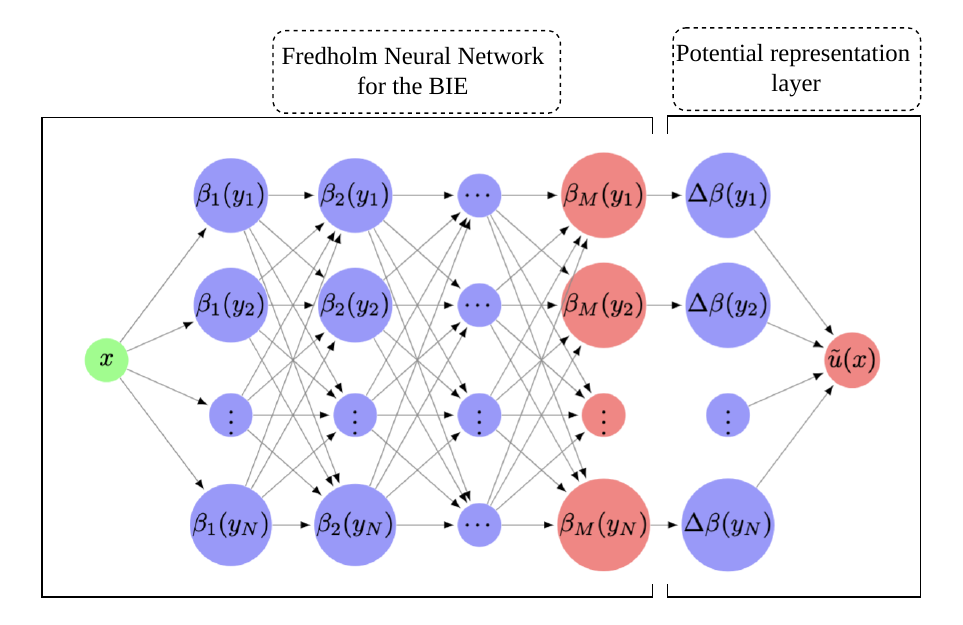}
    \caption{Schematic of the PFNN construction. The first component is a Fredholm Neural Network and the second encapsulates the representation of the double layer potential given in (\ref{potential-form-1}), decomposed into a the final hidden layer.} \label{fig:LIC-diagram}
\end{figure}

\subsubsection{The Helmholtz PDE}\label{semi-linear-pde-section}
We now consider the Helmholtz PDE, which is of the form:
\begin{eqnarray}\label{helmholtz-pde}
\begin{cases}
 \Delta u(x) - \lambda u(x) = \psi(x), \quad x \in \Omega \\ 
u(x) = f(x), \quad x \in \partial \Omega.   
\end{cases}
\end{eqnarray}

Theorem \ref{BIM} holds for the PDE (\ref{helmholtz-pde}), however the fundamental solution changes, and we have:
\begin{equation}\label{helm-fund}
   \Phi(x, y) = -\frac{1}{(2\pi)^{d/2}} \left( \frac{\sqrt{\lambda}}{|x - y|} \right)^{\frac{d-2}{2}} K_{\frac{d-2}{2}}(\sqrt{\lambda} |x - y|), 
\end{equation}
 where $K_{\nu}$ is the modified Bessel function of the second kind with index $\nu$. Even though many of the calculations for the PFNN construction remain the same as in the linear Poisson PDE, we provide the detailed calculations below, as the change in the fundamental solution leads to different handling of the double layer potential formulation, and subsequent changes to the neural network construction. We now prove the following proposition. 


\begin{proposition}\label{semi-linear-pot}
The Helmholtz PDE \eqref{helmholtz-pde} can be solved using a Fredholm NN with $M+1$ hidden layers, where the input is $x \in \Omega \subset \mathbb{R}^d$ and the first $M$ hidden layers solve the BIE (\ref{BIE}) on a discretized grid of the boundary, $y_1, \dots, y_N$. The final hidden and output layers are constructed according to the Fredholm NN representation of the double layer potential for PDE \eqref{helmholtz-pde}, with weights $W_{M+1} \in \mathbb{R}^{N \times N}, W_O \in \mathbb{R}^N$ given by:
\begin{eqnarray}
    W_{M+1}= I_{N \times N},
    \,\,\,\,\
    W_{O}= \left(\begin{array}{cccc}
	\mathcal{D} \Phi({x}, {y}_1)\Delta \sigma_{y}, & \mathcal{D} \Phi({x}, {y}_2)\Delta\sigma_{y}, & \dots, & \mathcal{D} \Phi({x}, {y}_N) \Delta \sigma_{y}
\end{array}\right)^{\top},
\end{eqnarray}
where, $\mathcal{D} \Phi({x}, {y}_i)$ as defined in Proposition \ref{prop-poisson}. The corresponding biases $b_{M+1} \in \mathbb{R}^{N}$ and $b_O \in \mathbb{R}$ are given by:
\begin{flalign}
   &b_{M+1} = \left(\begin{array}{ccc}
		- \beta({x}^{\star}), \dots, - \beta({x}^{\star})
	\end{array}\right)^{\top}, \\ 
 &b_O(x)= \Big(\frac{1}{2} + \int_{\Omega} \lambda \delta \Phi(x, y) dy \Big) \beta({x}^{\star}) + \int_{\partial \Omega} \beta(y) \frac{\partial \Phi(x^*, y)}{\partial n_y} d\sigma_y + \int_{\mathcal{D}} \Phi({x},{y}) f({y}) d {y},
\end{flalign}
where we define $\delta\Phi(x,y) := \Phi(x,y) - \Phi(x^*,y)$, and $x^* = \text{argmin}_{y \in \partial \Omega} \|x - y\|$, such that $x \rightarrow x^*$.
\end{proposition}

\begin{proof}
    
It is known that the fundamental solution of the linear operator $\mathcal{L} := \Delta - \lambda$ satisfies:
\begin{equation}
\mathcal{L} \Phi(x,y) = \delta(x - y).
\end{equation}
and is given by \eqref{helm-fund}. The corresponding BIE is:
\begin{equation}\label{bie-sl}
\beta(x) = 2\left( f(x) - \int_{\Omega} \Phi(x,y) \psi(y) dy \right) - 2 \int_{\partial \Omega} \beta(y) \frac{\partial \Phi(x,y)}{\partial n_y} d\sigma_y.
\end{equation}
It can be shown that the resulting kernel in the BIE (\ref{bie-sl}) is indeed non-expansive. Furthermore, it is known from the properties of the Bessel function that the weak singularity can be treated numerically (more details are given in section \ref{numerics-section}). Now, the change in the differential operator creates a difference in the discontinuity dynamics. Hence, we will need to modify the calculations as we approach the boundary. To this end, we consider the following steps:

\begin{itemize}
\item[{\it Step 1.}] Let $x \in \Omega$. Fix $\epsilon > 0$ and consider the ball centered at $x$ with radius $\epsilon$, such that $B(x,\epsilon) \subset \Omega$. In $\Omega - B(x,\epsilon)$, $\Phi(x,y)$ is smooth, and we have:
\begin{equation}
0 = \int_{\Omega - B(x,\epsilon)} \mathcal{L}\Phi(x,y) d\sigma_y = \int_{\Omega - B(x,\epsilon)} \Delta_y \Phi(x, y) dy - \lambda \int_{\Omega - B(x,\epsilon)} \Phi(x,y) d\sigma_y.
\end{equation}


By the Divergence Theorem, and using $n_y$ to be the outwaard unit normal to $\Omega - B(x,\epsilon)$, this becomes:
\begin{eqnarray}\label{div}
0 &=& \int_{\partial \Omega} \frac{\partial \Phi(x,y)}{\partial n_y} d\sigma_y + \int_{\partial B(x,\epsilon)} \frac{\partial \Phi(x,y)}{\partial n_y} d\sigma_y - \lambda \int_{\Omega - B(x,\epsilon)} \Phi(x,y)d\sigma_y
\end{eqnarray}


Let us now consider the second term, where the outer unit normal to $\Omega - B(x,\epsilon)$ is $n_y = \frac{x - y}{|x - y|}$.
Furthermore, note that, with $z = |x-y|$, $K_0(z) \sim - \ln(z)$ as $z \to 0$ (see e.g., \cite{kropinski2011fast}), and so we have:
\begin{flalign}
\frac{\partial \Phi(x,y)}{\partial n_y} = \nabla_y  \Phi(x, & y)  \cdot \frac{x - y}{|x - y|} \approx \frac{\sqrt{\lambda}}{2\pi} \big( \ln(\sqrt{\lambda} |x - y|) \big)' \big( (y_1 - x_1)^2 + (y_2 - x_2)^2 \big)  = -\frac{1}{2\pi |x - y|}.
\end{flalign}



Thus, we obtain
$\int_{\partial B(x,\epsilon)} \frac{\partial \Phi(x,y)}{\partial n_y} ds(y) \approx -1,$ and therefore from \eqref{div}, as $\epsilon \rightarrow 0$:
\begin{equation}\label{res1}
\int_{\partial \Omega} \frac{\partial \Phi(x,y)}{\partial n_y} d\sigma_y = 1 + \int_{\Omega} \lambda \Phi(x,y) dy.
\end{equation}

\item[{\it Step 2.}] Now, consider $x^* \in \partial \Omega$.
We know that $\frac{\partial \Phi(x,y)}{\partial n_y}$ is not defined for $x^* = y$. Fix $x^* \in \partial \Omega$ and consider $B(x^*,\epsilon)$ and $\Omega_\epsilon \equiv \Omega - \left( \Omega \cap B(x^*,\epsilon) \right)$. Furthermore, let $\Xi_\epsilon = \left\{ y \in \partial B(x^*,\epsilon) : n_y \cdot (x^* - y) < 0 \right\}$ and $\tilde{\Xi}_\epsilon = \partial \Omega_\epsilon \cap \Xi_\epsilon$. Then, since $x^* \notin \Omega_{\epsilon}$ and $\mathcal{L}\Phi(x^*,y)=0$ for $y\in \Omega_{\epsilon}$:
\begin{eqnarray}
0 &=& \int_{\partial \Omega_\epsilon} \mathcal{L}\Phi(x^*,y) dy = \int_{\Omega_\epsilon} \Delta_y \Phi(x^*,y) dy - \lambda \int_{ \Omega_\epsilon} \Phi(x^*,y) dy \notag \\
&=& \int_{\partial \Omega} \frac{\partial \Phi(x^*,y)}{\partial n_y} d\sigma_y + \int_{\tilde{\Xi}_\epsilon} \frac{\partial \Phi(x^*,y)}{\partial n_y} d\sigma_y -  \lambda \int_{ \Omega_\epsilon} \Phi(x^*,y) dy.
\end{eqnarray}

In the above, $n_y$ is the outer unit normal to $\Omega_\epsilon$. For $y \in \tilde{\Xi}_\epsilon$, $n_y = \frac{x^* - y}{|x^* - y|}$. Similar steps as before give:
\begin{equation}
\int_{\tilde{\Xi}_\epsilon} \frac{\partial \Phi(x^*,y)}{\partial n_y} d\sigma_y \approx \int_{\tilde{\Xi}_\epsilon} \frac{-1}{2\pi |x^* - y|} d\sigma_y \approx -\frac{1}{2\pi \epsilon} \left( \frac{1}{2} 2\pi \epsilon \right) = -\frac{1}{2}.
\end{equation}

Therefore, as as $\epsilon \rightarrow 0$:
\begin{equation}\label{res2}
\int_{\partial \Omega} \frac{\partial \Phi(x^*,y)}{\partial n_y} d\sigma_y =\frac{1}{2} + \int_{\Omega} \lambda \Phi(x^*,y) dy.
\end{equation}
\end{itemize}

With the above results at hand, we can add and subtract first $\frac{\partial\Phi(x^*,y)}{\partial n_y}$ and then $\beta(x^*)$, into the double layer potential \eqref{potential-integral}, and use \eqref{res1}-\eqref{res2} in order to obtain the following representation:
\begin{flalign}\label{potential-limit-expression}
u(x) = \int_{\partial \Omega} \big(\beta(y) &- \beta(x^*)\big) 
\Big( \frac{\partial \Phi(x,y)}{\partial n_y} - \frac{\partial \Phi(x^*, y)}{\partial n_y} \Big) d\sigma_y
 \nonumber \\
& + \beta(x^*)\Big(\frac{1}{2} + \int_{\Omega} \lambda \delta \Phi(x, y) dy \Big) + \int_{\Omega} \Phi(x, y) \psi(y) dy + \int_{\partial \Omega} \beta(y) \frac{\partial \Phi(x^*, y(s))}{\partial n_y} d\sigma_y,
\end{flalign}
with $\delta\Phi(x,y) := \Phi(x,y) - \Phi(x^*,y)$, completing the proof.
\end{proof}

As shown, in the case of the Helmholtz PDE, the difference in the fundamental solution creates an additional term in this new double layer potential representation. This, in turn, corresponds to a different bias term in the output of the constructed neural network. Indeed, as we will show in the error analysis section, the properties of the fundamental solution, and by extension, its numerical integration, play an important role in the approximation efficiency of the Fredholm NN. 

\subsubsection{Semi-linear elliptic PDEs}\label{non-linear-section}
Finally, we also consider the case of non-linear (semi-linear) elliptic PDEs of the form: 
\begin{eqnarray}\label{nl-pde}
\begin{cases}
\Delta u(x) = \psi(x, u(x)), \quad x \in \Omega \\
u(x) = f(x), \quad x \in \partial \Omega.
\end{cases} 
\end{eqnarray}
The dependency of the source term on the solution $u$ means that we can obtain an implicit  integral formulation of the semilinear PDE  (see representation \eqref{bie1} below) involving the double layer potential formulation, as detailed below. Hence, this cannot be implemented in the standard Potential Fredholm Neural Network framework, and we will employ an iterative scheme for the construction of the PFNN for the forward problem. We will explain how the implicit form, given in \eqref{bie1}, will be used for the inverse problem in as described below.

 Using the representations developed in the previous sections, we obtain a coupled, implicit integral formulation of the double layer potential and BIE. As we describe below, for the iterative approach it is beneficial to write the PDE using the Helmholtz differential operator. In order to use the same fundamental solution and structure, we adopt this approach for the implicit formulation, as well, even though it can easily be done also using the Poisson operator. To this end, we can consider the PDE re-written as:
    \begin{equation} 
    \begin{cases}
    \Delta u(x) - \lambda u(x) = -\lambda u(x) + \psi(x, u(x)), & x \in \Omega, \label{iteration}\\
    u(x) = f(x), & x \in \partial \Omega.
    \end{cases}
    \end{equation}
We can write $\tilde{\psi}(x,u(x)) := -\lambda u(x) + \psi(x, u(x))$, and obtain the corresponding double layer potential reformulation as:
\begin{flalign}\label{bie1}
u(x)=\int_{\partial\Omega}(\beta(y)-&\beta(x^*))\left(\frac{\partial\Phi(x,y)}{\partial n_y}- \frac{\partial\Phi(x^*,y)}{\partial n_y}\right)d\sigma_y  +\beta(x^*)\left(\frac12+\int_{\Omega}\lambda\delta\Phi(x,y)dy\right) \notag \\ & +\int_{\Omega}\Phi(x,y)\tilde{\psi}(y, u(y))dy+\int_{\partial\Omega}\beta(y)\frac{\partial\Phi(x^*,y(s))}{\partial n_y}d\sigma_y, \,\,\, x \in \Omega,
\end{flalign}
\begin{flalign}\label{bie2}
g(x) =
\int_{\partial\Omega}
\beta(y)
\frac{\partial \Phi(x,y)}{\partial n_y}
d\sigma_y+\int_{\Omega} \Phi(x,y) \tilde{\psi}(y,u(y))dy + \frac{1}{2}\,\beta(x)  \,\,\, x \in \partial\Omega,
\end{flalign}
where $x \in \Omega$ and $x^* \in \partial \Omega$ and recall that $\Phi$ is the fundamental solution of the Helmholtz PDE. 

Notice that the dependency on the solution $u(x)$ within the integral formulation does not create an issue when constructing the PFNN, for a given $\psi$. Indeed, this is exactly the fact that we take advantage of to solve the inverse problem, where we consider that we have data from the solution, which we can use to construct a forward PFNN whose input matches the given data.

\color{black}

\subsection{PFNNs for the inverse source problem }\label{inv-sec}
\par As described, the inverse source problem can now be studied within the framework of the PFNN. Conceptually, the inverse problem consists of, given the data ${u}$, $f$, modelling the unknown source function, such that, when using the approximated source function is passed to the forward PFNN, the resulting solution to the PDE, $\hat{u}$, is as close as possible to the given data ${u}$. Since the PFNN simulates the fixed-point scheme and the double layer representation, the learned source term is ''forced'' to be selected to satisfy the PDE, simply by passing through the PFNN during training. More specifically: select a set of parameters $\theta$ such that when constructing the estimated kernel ${\psi}_{\theta}(\cdot, \cdot)$ and then feeding this into the PFNN with $M$ hidden layers, denoted by $PFNN_M( \cdot ; {\psi}_{\theta})$, the output, $\hat{u}(x;{\psi}_{\theta})$ (which will also be denoted $\hat{u}(x)$ for brevity), is as ''close'' as possible (in terms of an appropriately chosen loss function) to the given data ${u}$. The learning problem then reduces to the optimization problem of tuning the parameters $\theta$ appropriately until we reach the optimal set $\theta^*$ and the corresponding source model ${\psi}_{\theta^*}$. 
 \par The loss function to be minimized is therefore given by:
\begin{equation}\label{loss-inverse}
    \mathcal{L}(\theta) = \| {u}(x) - \hat{u}(x_i; \psi_{\theta}) \|^2_{{L}^2(\Omega)} + \mathcal{R}(\theta),
\end{equation}
where $\mathcal{R}(\theta) = \lambda_{reg}\|\psi_{\theta} \|^2_{L^2(\Omega)}$, with scalar $\lambda_{reg} > 0$, is a Tikhonov regularization term, used due to the ill-posedness of the problem. For the numerical implementation of the method, we will use the discrete version of \eqref{loss-inverse}, given by: 
\begin{equation}\label{inverse-loss}
     \hat{\mathcal{L}}(\theta) = \frac{1}{N} \sum_{i=1}^{N} \Big({u}(x_i) - \hat{u}(x_i; {\psi}_\theta) \Big)^2 + \hat{\mathcal{R}}(\theta),
\end{equation}
where $\hat{\mathcal{R}}(\theta)$ is calculated using the discrete $L^2$ norm of the learned source term at the specific grid points we observe the data.
The approach is detailed in Algorithm \ref{alg:inverse} and shown schematically in Fig. \ref{fig:inverse}.

\begin{algorithm}[hbt!]\caption{PFNN for the inverse source problem}\label{alg:inverse}
\begin{algorithmic}
\Ensure {Coarse grid $\{x_i\}_{i=1, \dots N}$ and data ${u}(x_i)$, for $x_i \in \Omega$.}
\State {Create grid $\{x^b_i\}_{i=1, \dots N'}$ and set the PFNN with nodes corresponding to $\{x^b_i\}$  and depth $M$.}
\State {Set number of training epochs $N$, initialize iterations, $n=0$.}
\State {Approximate the unknown source $\psi(x)$ with model ${\psi}_{\theta^{(n)}}(x)$, where $\theta^{(n)}$ denotes the set of unknown  parameters of the NN (weights and biases) at iteration $n$ of training.}
\While{$n \leq N$}
    \State {Step 1. Construct a PFNN, $PFNN_M(x;{\psi}_{\theta^{(n)}})$.}
    \State {Step 2. Use the PFNN output, $\hat{u}(x_i; {\psi}_{\theta^{(n)}}) \equiv \hat{u}(x_i;\theta^{(n)})$ to calculate the model loss given by \eqref{inverse-loss}. }
    \State {Step 3. Train model using standard gradient descent methods to minimize $\mathcal{L}(\theta)$.}
    \State {Step 4. Set $n \leftarrow n+1$.}
\EndWhile \\
\Return {${\psi}_{\theta^{*}}(x)$ and $\hat{u}(x_i;\theta^{*})$} 
\end{algorithmic}
\end{algorithm}
\begin{figure}[!ht]
\begin{center}
  \includegraphics[width = 0.9\textwidth]{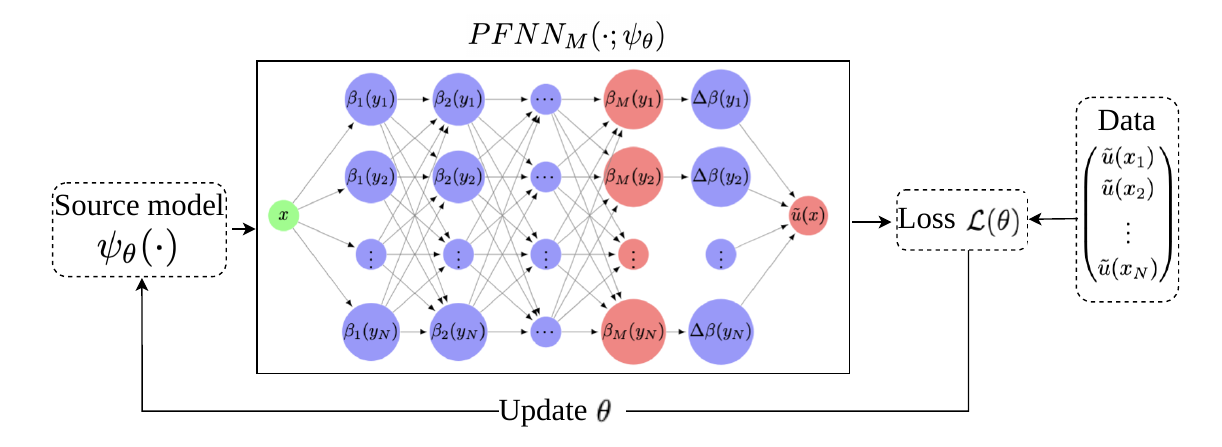} 
\end{center}
\caption{Schematic representation of Algorithm \ref{alg:inverse} to solve the inverse problem using the Potential FNN framework. The dashed line is applicable for the semi-linear PDE, in which case the data from the solution $u(x)$ are used as input to the unknown source model $\psi_{\theta}(x,u).$}\label{fig:inverse}
\end{figure}


\par We would like to emphasize that the structure of the PFNN itself ensures that we are simultaneously accounting for the data loss, as well as forcing the model to adhere to the physics of the problem. This is an important novelty of the proposed method: rather than considering an arbitrary neural network structure and calculating the residuals of the PDE, a neural network with a pre-determined, numerical-analysis based architecture is used to ensure that the functions passed through it adhere to the given physics of the problem (in our case the PDE). Furthermore, with the proposed approach, notice that we learn the source term and simultaneously the full solution $u(x), x \in \Omega$, by having constructed the the Potential Fredholm NN, $PFNN_M(\cdot; \psi_{\theta^*})$, after training.

\subsection{Universal approximation of the unknown source terms and error analysis}\label{error-section}
In this section, we show that the proposed scheme produces source term models which are universal approximators of the unknown source terms for the corresponding elliptic PDEs and we deduce explicit approximation error bounds. We build the proofs by first constructing error bounds for the corresponding forward problems. Then, from these results, we will show that we can approximate the source terms such that the approximated and true solutions to the PDEs are arbitrarily close. Throughout this section we use the sup-norm, denoted $\| f(x) \|_\mathcal{A} := sup_{x\in \mathcal{A}} |f(x)|$. For ease of notation, when there is no doubt concerning the domain under consideration, we will simply denote this norm by $\| \cdot \|$, only specifying the domain if there are more than one domains involved, i.e., ${\cal A}$ can either be $\Omega$ or $\partial \Omega$. 

\par We recall two important points: firstly, we note that the Fredholm NN replicates the process of successive approximations, which inherently introduces an approximation error $||\beta_M - \beta||_{\partial \Omega}$. Secondly, the solution constructed by the double layer potential satisfies identically the boundary condition when the BIE is satisfied. Therefore, based on these observations, it is expected that a stable scheme ensures that approximation errors at the boundary decrease progressively as the number of iterations increases. 
\par Here, we  provide bounds for these errors. Indeed, the Fredholm NN and subsequently the PFNN  construction provide a framework where this is possible. This is an important step, since DNN-based models often lack consistent error estimates, and often exhibit relatively large errors on the boundaries. 
\par For the completeness of the presentation, we extend the following abstract proposition (see \cite{georgiou2024fredholm}). 
\begin{proposition} \label{error-proposition}
Consider the integral equation $f(x) = g(x) + (\mathcal{T}f)(x)$, with the integral operator defined by:
\begin{eqnarray}
    (\mathcal{T}f)(x) =  \int_{\mathcal{A}}K(x,y) f(y)dy,
\end{eqnarray}
with $\mathcal{A}\subset \mathbb{R}^p$, and let $\mathcal{T}_N$ be the corresponding discretization on a grid of size $N$ and denote the fixed-point solution of the integral equation, $f^*$. Then:
\begin{itemize}
    \item[$(i)$] if $\mathcal{T}$ is a contractive operator, the error in the M-layer FNN approximation, $f_M$, satisfies the following bound:
 \begin{eqnarray}\label{fnn-error}
    \left\|f^*-f_M\right\|_{\mathcal{A}} \leq \frac{\mathrm{e}^{1-q}}{1-q}\big( \left\|\mathcal{T} g-g\right\|_{\mathcal{A}} + \varepsilon_N(g) \big)\mathrm{e}^{-(1-q) v_M},   
\end{eqnarray}

where $\varepsilon_N(g)$ is the error incurred by the discretization of the integral (i.e., $\varepsilon_N(g):= \|\mathcal{T}g - \mathcal{T}_Ng\|_{\mathcal{A}} $), $v_0=0, v_M=M\kappa$ (since in the Fredholm NN we consider constant values for the parameter in the Krasnosel'skii-Mann iteration, $\kappa_{\nu} = \kappa$). 
    \item[$(ii)$] if $\mathcal{T}$ is non-expansive, and assuming the operator $(I - \mathcal{T})$ is metrically subregular, i.e.:
    \begin{equation}\label{metric-sub}
        \|f^* - f \|_{\mathcal{A}} \leq \mu \| f - \mathcal{T}f \|_{\mathcal{A}},
    \end{equation}
for a constant $\mu >0$. Then,
\begin{equation}\label{error-bound-new}
    \|f_M - f^* \|_{\mathcal{A}} \leq \left(\sqrt{1 - \frac{\kappa(1-\kappa)}{\mu^2}}\right)^M \|g-f^*\|_{\mathcal{A}}.
\end{equation}

\end{itemize}
\end{proposition}

\begin{proof}
    Proof of part $(i)$ can be found in \cite{georgiou2024fredholm} for the 1D case. The extension to higher dimensions follows by simply generalizing the integral discretization error.
    
    For part $(ii)$, we start by recalling that, for non-expansive operators, the KM algorithm guarantees Fejer monotonicity, i.e.,:
    \begin{equation}\label{fejer}
        \|f_{M} - f^*\|^2 = \|f_{M-1} - f^*\|^2 - \kappa(1-\kappa)\|f_M - f_{M-1}\|^2.
    \end{equation}
Note that we have dropped the subscript of the norm, for brevity. Then, notice that \eqref{metric-sub} can written as $\| f^* - f_{M-1}\| \leq \mu \|f_M - f_{M-1}\|$. We now substitute into \eqref{fejer} to obtain:
\begin{equation}
     \|f_M - f^*\| \leq \sqrt{1 - \frac{\kappa(1-\kappa)}{\mu^2}} \|f_{M-1}-f^*\|.
\end{equation}
Simply iterating over $M$, we obtain \eqref{error-bound-new}, as required.
\end{proof}

\par Given that the Fredholm NN replicates a discretization of the BIE, as well as the corresponding integrals in the proposed formula for the potential, it will be useful to use a notation for the discretization that encapsulates various schemes that may be considered. Throughout the proof, we will use the notation $\tilde{\int_{\mathcal{A}}}$ to represent the discretized version of $\int_{\mathcal{A}}$, for any arbitrary domain $\mathcal{A} \subseteq \mathbb{R}^p$.

We emphasize that this notation is useful so that the results that follow can be applied to any chosen discretization scheme. \par 
We now demonstrate the result for the error bounds in the domain and on the boundary of the PDEs, when solving the forward problem using the PFNN. (Note that we present the proof for the Helmholtz PDE. The proof for the Poisson case is similar and requiring only straightforward modifications and is hence omitted. The non-linear case is presented separately below.) 

\begin{proposition}[Error bounds for PFNN]\label{error-bounds-thrm}
 Consider the PDE of the form (\ref{helmholtz-pde}), with a potential  given in form (\ref{potential-limit-expression}) and corresponding BIE, (\ref{BIE}), whose Fredholm NN approximation is given by $\beta_M(x): \partial \Omega \rightarrow \mathbb{R}$. Furthermore, define the following terms, depending on the discretization scheme selected for the integrals in (\ref{potential-limit-expression}):
 \begin{flalign}
     &D_1(\mathcal{A}) := \left\|\ \int_{\Omega}  \delta \Phi(x, y) dy -  \tilde{\int_{\Omega}} \delta \Phi(x, y) dy \right\|_{\mathcal{A}}, \\
     &D_2(\mathcal{A}) := \left\| \int_{\Omega} \Phi(x, y) \psi(y)  dy - \tilde{\int_{\Omega}}\Phi(x, y) \psi(y) dy \right\|_{\mathcal{A}}, \\
    &D_3(\mathcal{A}) := \left\|\int_{\partial \Omega} \Big(\beta_M(y) - \beta_M(x^*)\Big) \mathcal{D} \Phi(x,y) d\sigma_y - \tilde{\int_{\partial \Omega}} \big(\beta_M(y) - \beta_M(x^*)\big) \mathcal{D} \Phi(x,y) d\sigma_y\right\|_{\mathcal{A}}, \\
     & D_4(\mathcal{A}) := \left\| \int_{\partial \Omega} \beta_{M}(y)  \frac{\partial \Phi(x, y)}{\partial n_y}  d\sigma_y - \tilde{\int_{\partial \Omega}} \beta_M(y) \, \frac{\partial \Phi(x, y)}{\partial n_y} d\sigma_y \right\|_{\mathcal{A}},
 \end{flalign}
 
 when $x \in \mathcal{A}$, and recall that $\mathcal{D} \Phi({x}, {y}):= \Big(\frac{\partial \Phi}{\partial n_{y}}({x}, {y})- \frac{\partial \Phi}{\partial n_{y}}({x}^{\star}, {y})\Big)$.  Then, the PFNN construction providing $\tilde{u}$, as approximation of the solution of the PDE, satisfies the following error bound:
\begin{flalign}\label{error-domain}
    \left\|u - \hat{u} \right\|_{\Omega} \leq \left\|\beta - \beta_M \right\|_{\partial \Omega}\left( \frac{1}{2} + 2\left\|\int_{\partial \Omega}  \mathcal{D} \Phi(x,y)d\sigma_y\right\|_{\Omega} + \left\| \int_{\Omega} \lambda \delta \Phi(x, y) dy \right\|_{\Omega} + \left\| \int_{\partial \Omega}\frac{\partial \Phi(x, y)}{\partial n_y} d\sigma_y \right\|_{\Omega} \right) \notag \\ + \left\|\beta_M\right\| \lambda D_1(\Omega) + D_2(\Omega) + D_3(\Omega) + D_4(\Omega).
\end{flalign}

Furthermore, the following inequality holds for errors on the boundary, $x\in \partial \Omega$: 
\begin{flalign}\label{error-boundary}
  \left\|f - \hat{f} \right\|_{\partial \Omega} \leq \left\|\beta - \beta_M \right\|_{\partial \Omega} \left( \frac{1}{2} + \left\| \int_{\partial \Omega}\frac{\partial \Phi(x^*, y)}{\partial n_y} d\sigma_y \right\|_{\partial \Omega} \right)  + D_2(\partial \Omega) + D_4(\partial \Omega),  
\end{flalign}
where $\hat{f}(x) = \hat{u}(x)$, for $x\in\partial\Omega$, is the approximated boundary condition and $\left\|\beta - \beta_M \right\|_{\partial \Omega}$ is the error in the  FNN approximation of the boundary function $\beta$, as given in (\ref{fnn-error}).
\end{proposition}
\begin{proof}
The proof entails decomposing the error into the components arising from the Fredholm NN estimator and the subsequent PFNN construction. Both of these are directly associated to well-known numerical methods with well-defined error bounds, which we can now take advantage of in the context of Fredholm NNs. 
\par Considering the potential formulation (\ref{potential-limit-expression}) and keeping the notation $\mathcal{D} \Phi({x}, {y}):=\frac{\partial \Phi}{\partial n_{y}}({x}, {y})- \frac{\partial \Phi}{\partial n_{y}}(x^*, {y})$, we can write: 
\begin{flalign}
\left\|u -  \hat{u}\right\|_{\Omega} &\leq \left\|\int_{\partial \Omega} \big(\beta(y) - \beta(x^*) \big) \mathcal{D} \Phi(x,y) d\sigma_y -  \tilde{\int_{\partial \Omega}} \big(\beta_M(y) - \beta_M(x^*)\big) \mathcal{D} \Phi(x,y) d\sigma_y \right\|_{\Omega} \notag \\ 
&+ \left\| \beta(x^*) \Big(\frac{1}{2} + \int_{\Omega} \lambda \Delta \Phi(x, y) dy \Big) - \beta_M(x^*) \Big(\frac{1}{2} +  \tilde{\int_{\Omega}} \Delta \Phi(x, y) dy \Big) \right\|_{\Omega} \nonumber \\
&+ \left\| \int_{\Omega} \Phi(x, y) \psi(y) dy - \tilde{\int_{\Omega}} \Phi(x, y) \psi(y) dy \right\|_{\Omega} \notag \\ 
&+ \left\|\int_{\partial \Omega} \beta(y) \frac{\partial \Phi(x, y)}{\partial n_y}d\sigma_y - \tilde{\int_{\Omega}}\beta_M(y) \frac{\partial \Phi(x, y)}{\partial n_y} d\sigma_y\right\|_{\Omega} \\
&=: I_1 + I_2 +I_3 +I_4.
\end{flalign}
We now consider each of the terms in the above expression separately. For clarity, in the steps below we gather the terms corresponding to the integral discretization errors and subsequently handle these using the well-known results for discretized integrals:
\begin{itemize}
    \item[{\it 1st term:}]
\begin{flalign}
   I_1:= &\left\| \int_{\partial \Omega} \big(\beta(y) - \beta(x^*)\big) \mathcal{D} \Phi(x,y) d\sigma_y - \tilde{\int_{\partial \Omega}}\big(\beta_{M}(y) - \beta_{M}(x^*)\big) \mathcal{D} \Phi(x,y) d\sigma_y \right\|_{\Omega} \notag \\
     \leq &\left\| \int_{\partial \Omega}\Big(\big(\beta(y) -  \beta_M(y)\big) + \big(\beta_M(x^*) -  \beta(x^*)\big)\Big) \mathcal{D} \Phi(x,y) d\sigma_y \right\|_{\Omega} \notag \\ 
      + & \left\|\int_{\partial \Omega} \Big(\beta_M(y) - \beta_M(x^*)\Big) \mathcal{D} \Phi(x,y) d\sigma_y - \tilde{\int_{\partial \Omega}} \big(\beta_M(y) - \beta_M(x^*)\big) \mathcal{D} \Phi(x,y) d\sigma_y\right\|_{\Omega} \notag \\
     \leq & \int_{\partial \Omega} 2\left\|\beta_M - \beta\right\|_{\partial \Omega} \mathcal{D} \Phi(x,y)d\sigma_y + D_3(\Omega) \leq 2\left\|\beta_M - \beta\right\|_{\partial \Omega} \left\|\int_{\partial \Omega}  \mathcal{D} \Phi(x,y)d\sigma_y\right\|_{\Omega} + D_3(\Omega),
\end{flalign}
 where for the last steps we recall that the normal derivative of the fundamental solution is non-negative (see calculations for the potential above).
    \item[{\it 2nd term:}] 
    \begin{flalign}
   I_2:= & \left\| \beta(x^*) \Big(\frac{1}{2} + \int_{\Omega} \lambda \delta \Phi(x, y) dy \Big)  - \beta_{M}(x^*) \Big(\frac{1}{2} +  \tilde{\int_{\partial \Omega}} \lambda \delta \Phi(x, y) dy \Big)  \right\|_{\Omega} \notag \\
    \leq& \frac{1}{2} \left\| \beta - \beta_{M} \right\|_{\partial \Omega} + \left\| \beta(x^*) \int_{\Omega} \lambda \delta \Phi(x, y)dy  - \beta_{M}(x^*)  \int_{\Omega} \lambda \delta \Phi(x, y)dy \right\|_{\Omega} \notag \\
     +& \left\|\beta_M(x^*) \int_{\Omega} \lambda \delta \Phi(x, y) dy   - \beta_{M}(x^*) \tilde{\int_{\Omega}} \delta \Phi(x, y) dy \right\|_{\Omega} \notag \\
    \leq & \left\|\beta - \beta_M\right\|_{\partial \Omega} \left( \frac{1}{2} + \left\| \int_{\Omega} \lambda \delta \Phi(x, y) dy \right\|_{\Omega} \right)  + \left\|\beta_M \right\|\lambda D_1(\Omega).
    \end{flalign}
    
    \item[{\it 3rd term:}]
    \begin{eqnarray}
  I_3:= \left\| \int_{\Omega} \Phi(x, y) \psi(y)  dy - \tilde{\int_{\Omega}}\Phi(x, y) \psi(y) dy \right\|_{\Omega} =: D_2(\Omega).
    \end{eqnarray}

    \item[{\it 4th term:}]
    \begin{flalign}
   I_4:= & \left\| \int_{\partial \Omega} \beta(y) \frac{\partial \Phi(x, y)}{\partial n_y} d\sigma_y - \tilde{\int_{\partial \Omega}} \beta_M(y)  \frac{\partial \Phi(x, y)}{\partial n_y} d\sigma_y \right\|_{\Omega} \notag \\
    \leq & \left\| \int_{\partial \Omega} \beta(y) \, \frac{\partial \Phi(x, y)}{\partial n_y} d\sigma_y - \int_{\partial \Omega} \beta_{M}(y)  \frac{\partial \Phi(x, y)}{\partial n_y} d\sigma_y \right\|_{\Omega} \notag \\
    + & \left\| \int_{\partial \Omega} \beta_{M}(y)  \frac{\partial \Phi(x, y)}{\partial n_y}  d\sigma_y - \tilde{\int_{\partial \Omega}} \beta_M(y) \frac{\partial \Phi(x, y)}{\partial n_y} d\sigma_y \right\|_{\Omega} \notag \\
    \leq & \int_{\partial \Omega} \left\| \beta - \beta_{M} \right\|_{\partial \Omega} \frac{\partial \Phi(x, y)}{\partial n_y}  d\sigma_y + D_4(\Omega) \leq \left\| \beta - \beta_{M} \right\|_{\partial \Omega} \left\| \int_{\partial \Omega}\frac{\partial \Phi(x, y)}{\partial n_y} d\sigma_y \right\|_{\Omega} + D_4(\Omega).
    \end{flalign}
       
\end{itemize}
Hence, we obtain:
\begin{flalign}\label{error-bound}
    \left\|u - \hat{u} \right\|_{\Omega} \leq \left\|\beta - \beta_M \right\|\left( \frac{1}{2} + 2\left\|\int_{\partial \Omega}  \mathcal{D} \Phi(x,y)d\sigma_y\right\|_{\Omega} + \left\| \int_{\Omega} \lambda \delta \Phi(x, y) dy \right\|_{\Omega} + \left\| \int_{\partial \Omega}\frac{\partial \Phi(x, y)}{\partial n_y} d\sigma_y \right\|_{\Omega} \right) \notag \\ + \left\|\beta_M\right\|_{\partial \Omega} D_1(\Omega) + D_2(\Omega) + D_3(\Omega) + D_4(\Omega). 
\end{flalign}
The error bound for the boundary condition follows directly from the above, by noticing the vanishing terms as $x\rightarrow x^*$.
\end{proof}

\begin{remark}\label{error-remark}
The above result in combination with \eqref{fnn-error} provide an explicit expression for the error bounds in the PFNN construction, which depend on the boundary conditions of the PDE (via the BIE), the fundamental solution of the differential operator and the discretization of the integrals in the FNN representation of the potential. Through this composition, we are able to define the approximation errors in detail and quantify how or when larger errors may arise. 

\par Therefore, within the proposed framework, we see that the error in the model approximation can be attributed to: a) the number of hidden layers in the Fredholm NN, which is equivalent to the number of iterations required to converge to the fixed point within a certain accuracy, and b) the number of nodes used to construct the FNN on which we approximate the solution of the PDE, and which are then used for the numerical integration. However, we also see that the error bound also depends explicitly on the boundary function $\left\|\beta_M \right\|$. This occurs from two components of expression \eqref{error-domain} due to the error in the integral approximation. In particular, this dependency occurs through the term $\|\beta_M\|D_1(\Omega)$, as well as $D_4(\mathcal{A})$, where $\mathcal{A} = \Omega$ or $\partial \Omega$, since:
\begin{equation}
    D_4(\mathcal{A}) \leq C \max_{\substack{x \in \mathcal{A}}}\frac{d}{dy}\Big( \beta_M(y)\frac{\partial \Phi(x,y)}{\partial n_y} \Big) = C_1 \|\beta_M\|_{\mathcal{A}} + C_2,
\end{equation}
by a simple application of the product rule, where $C>0$ is a constant depending on the size of the discretization grid and $C_1, C_2 > 0$ are constants depending on the derivatives of $\frac{\partial \Phi(x,y)}{\partial n_y}$ and $\beta_M$, respectively. 
This decomposition of the errors therefore provides insight into how the form of the PDE can actually affect the error of the numerical approximation. 
\end{remark}

With the calculations above, we can now prove the main result for the inverse problem. We continue with the Helhmoltz PDE. The Poisson and semi-linear results follow easily (additional details are given below). 

\begin{theorem}[Universal Approximation Theorem via PFNNs]\label{UAT-inverse}
    Consider the PDE of the form \eqref{helmholtz-pde} with given boundary function $f: \partial \Omega \rightarrow \mathbb{R}$ and unknown source term $\psi:\Omega \rightarrow \mathbb{R}$, and true solution $u:\Omega \rightarrow \mathbb{R}$. Then, there exists a surrogate model $\psi_{\theta}: \Omega \rightarrow \mathbb{R}$, for some parameters $\theta \in \Theta$, where $\Theta \subset \mathbb{R}^P, P \in \mathbb{N}$ is a finite dimensional parameter set, such that:
    \begin{equation}
        \| u - \hat{u}_{\theta} \|_{\Omega} < \epsilon,
    \end{equation}
    for any $\epsilon >0$, where $\hat{u}_{\theta}$ is the output of the $PFNN_M(\cdot; \psi_{\theta})$. Furthermore, the error bound for the corresponding boundary condition approximation, $\hat{f}_{\theta}$, satisfies:
    \begin{equation}
        \| f - \hat{f}_{\theta} \|_{\partial \Omega} < \| u - \hat{u}_{\theta} \|_{\partial \Omega} < \epsilon.
    \end{equation}
\end{theorem}
\begin{proof}
    As mentioned, the source term $\psi$ satisfying the PDE is not unique. Hence, it suffices to show that we can approximate any such $\psi$ in the family of source terms that satisfy the PDE for the given solution $u$. Then, by the Universal Approximation Theorem for functions, there exists a neural network model with parameters $\theta \in \Theta$, $\psi_{\theta}$, such that, for any $\epsilon_1 > 0$: 
    \begin{equation}\label{uat-psi}
        \| \psi_{\theta} - \psi \|_{\Omega} < \epsilon_1.
    \end{equation}

    We can now show that the output of $PFNN_M(\cdot; \psi)$ can be approximated by $PFNN_M(\cdot; \psi_{\theta})$. To this end, we construct the error bound in the PFNNs as follows. Define $
g(x) = 2\!\left( f(x) - \int_{\Omega} \Phi(x,y)\,\psi(y)\,dy \right)
$ and
$g_{\theta}(x) = 2\!\left( f(x) - \int_{\Omega} \Phi(x,y)\,\psi_{\theta}(y)\,dy \right).$ Then:
\begin{equation}
\|g - g_{\theta}\|_{\partial \Omega} = 2\left\| \int_{\Omega} \Phi(x,y)\bigl(\psi(y) - \psi_{\theta}(y)\bigr)\,dy \right\|_{\partial \Omega}
\leq 2\|\psi - \psi_{\theta}\|_{\partial \Omega} \left\|\int_{\Omega} \Phi(x,y)\,dy\right\|_{\partial \Omega}
= C_1(R,N)\,\varepsilon_1
\end{equation}
where $C_1(R,N)$ is a constant that depends on the radius of the bounded domain $\Omega$ and on the discretization grid size, $N$, used to approximate the integral.

We now use this to obtain the error bound for the boundary function. Then, for the corresponding boundary function $\beta(x), x \in \partial \Omega$, we obtain:
\begin{flalign}
    \| \beta - \beta_{\theta} \|_{\partial \Omega} \leq \left\|\beta - \mathcal{F}_M\left(\cdot; g, \frac{\partial \Phi}{\partial n_y}\right) \right\|_{\partial \Omega} &+ \left\|\beta_{\theta} - \mathcal{F}_M\left(\cdot; g_{\theta}, \frac{\partial \Phi}{\partial n_y}\right) \right\| +  \left\| \mathcal{F}_M\left(\cdot; g, \frac{\partial \Phi}{\partial n_y}\right) - \mathcal{F}_M\left(\cdot; g_{\theta}, \frac{\partial \Phi}{\partial n_y}\right) \right\|_{\partial \Omega} \notag \\&= E(M,N)+ C_2(N)\|g-g_{\theta}\|_{\partial \Omega} = E(M,N) + C_1(R,N)C_2(N)\epsilon_1,
\end{flalign}
where $\mathcal{F}_M\left(\cdot;g, \frac{\partial \Phi}{\partial n_y}\right)$ is the Fredholm Neural Network $M$ layers, with first layer activation $g(\cdot)$ and weights given by the kernel $\frac{\partial \Phi}{\partial n_y}$ (in accordance to Theorem \ref{FUA}), $E(M,N)$ is the error in the approximation of $\beta$ and $\beta_{\theta}$ by the corresponding Fredholm Neural Networks (depending on the discretization grid size $N$ and number of layers $M$ - see \cite{georgiou2024fredholm}) and $C_2(N)$ is a constant depending on the norm of the Kransosel'skii-Mann algorithm applied at each layer of the Fredholm Neural Network.


Finally, we must account for the final layer of the PFNN corresponding to the double layer potential. We can therefore write:
\begin{align}\label{error-1}
&\|PFNN_M(\cdot;\psi) - PFNN_M(\cdot;\psi_{\theta})\|_{\Omega} \notag\\
&= \left\| \tilde{\int}_{\partial\Omega} \beta(y)\,\frac{\partial\Phi(x,y)}{\partial n_y}\,d\sigma_y
   +  \tilde{\int_{\Omega}} \Phi(x,y)\,\psi(y)\,dy
   -  \tilde{\int_{\partial\Omega}} \beta_{\theta}(y)\,\frac{\partial\Phi(x,y)}{\partial n_y}\,d\sigma_y
   -  \tilde{\int_{\Omega}} \Phi(x,y)\,\psi_{\theta}(y)\,dy \right\|_{\Omega} \notag\\
&\leq \left\| \tilde{\int_{\partial\Omega}} \frac{\partial\Phi(x,y)}{\partial n_y}\,
   \bigl(\beta(y) - \beta_{\theta}(y)\bigr)\,dy \right\|
   + \left\|  \tilde{\int_{\Omega}} \Phi(x,y)\,\bigl(\psi(y) - \psi_{\theta}(y)\bigr)\,dy \right\|_{\Omega} \notag\\
&\leq \|\beta - \beta_{\theta}\|_{\partial \Omega} \underbrace{\left\|\tilde{\int_{\partial\Omega}} \frac{\partial\Phi(x,y)}{\partial n_y} dy\right\|_{\Omega}}_{:=C_3}
   + \|\psi - \psi_{\theta}\|_{\Omega}  \underbrace{\left\|\tilde{\int}_{\Omega} \Phi(x,y)dy \right\|_{\Omega }}_{:= C_4} \notag \\&= C_3E(M,N) + \big(C_1(R,N)C_2(N)C_3 + C_4\big)\epsilon_1,
\end{align}

 Hence, from \eqref{error-1}, we can deduce that $\psi_{\theta}$ can be selected so that $PFNN_M(\cdot;\psi_{\theta})$ approximates the Potential Fredholm NN, $PFNN_M(\cdot;g)$ arbitrarily close.

Finally, the desired bound is deduced as follows: let $\hat{u}_{\theta}$ be the output of ${PFNN}_M(\cdot;\psi_{\theta})$ and $\hat{u}$ of ${PFNN}_M(\cdot;\psi)$. We then have:
\begin{equation}
\|u - \hat{u}_{\theta}\|_{\Omega} \leq \|u - \hat{u}\|_{\Omega} + \|\hat{u} - \hat{u}_{\theta}\|_{\Omega}.
\end{equation}
The first term is bounded by Proposition \ref{error-bounds-thrm}, from which it follows that we can select the Fredholm Neural Network depth $M$ and the discretization grid size $N$ large enough so that $\| u - \hat{u} \| < \epsilon_2$ and the second term satisfies \eqref{error-1}. Hence, for any closeness level $\epsilon$, we can select $\epsilon_1$, $\epsilon_2$ and the corresponding depth and discretization size, $M$ and $N$, respectively, such that $C_3E(M,N) + \big(C_1(R,N)C_2(N)C_3 + C_4\big)\epsilon_1 + \epsilon_2 < \epsilon$, to obtain:
\begin{equation}
    \| u - \hat{u}_{\theta} \|_{\Omega} < \epsilon,
\end{equation}
as required.

Finally, we can deduce that the error on the boundary is strictly less that $\epsilon$, since by Proposition \ref{error-bounds-thrm}, $\| f - \hat{f}_{\theta} \|_{\partial \Omega} < \| u - \hat{u}_{\theta}\|_{\Omega}$.
\end{proof}

An analogous result holds for the Poisson PDE, as the above proof can easily be replicated using the PFNN construction as given in section \ref{section-PFNN-Poisson}. For the semi-linear case the proof of Theorem \ref{UAT-inverse} holds, but the model for the source term becomes $\psi_{\theta}(x,u)$, i.e., we add another input dimension, corresponding to the solution input (for which we assume we have training data in this setting).

At this point it is worth highlighting that Theorem \ref{UAT-inverse} shows that we are able to obtain smaller errors on the boundary, by the construction of the PFNN and the training algorithm. This is an important theoretical guarantee, since black-box neural networks often show increased errors on the boundaries (when they are trained using soft loss functions for the boundary data), compared to the interior of the domain. With the PFNN, the network is designed to ensure the boundary conditions are satisfied. As we will see in the numerical examples, indeed the proposed architecture produces solutions which satisfy the boundary conditions within machine precision.

\par Finally, we note that the  error analysis can be useful when considering the computational requirements for the construction of DNNs under the proposed framework. In particular, for a given computational cost, there exists a trade-off between the number of layers in the Fredholm NN for the approximation of the boundary function $\beta_M$, and the discretization of the integral terms; increasing the number of hidden layers in the Fredholm NN will ensure a smaller error $\left\| \beta_M - \beta \right\|_{\partial \Omega}$, thus decreasing the first four components in (\ref{error-domain}), but selecting a more computationally demanding numerical integration will decrease the final four terms. The selection of these hyperparameters can be done in a robust way since we know the components and how the underlying form of the PDE can affect the errors.

\section{Numerical Implementation}\label{numerics-section}
In this section, we detail the numerical implementation of Potential Fredhol NNs. To this end, it is important to note that the fundamental solutions both for the Poisson as well as the Helmholtz operators exhibit weak singularities which, although integrable, can cause issues during the numerical estimation of the integrals which depend on them. We describe how these are handled to ensure high accuracy is maintained throughout the domain and boundary. 
\par Our main focus is on the solution of the inverse problems for the sources. However, for completeness, in Appendix \ref{appendixx}, we also show the numerical implementation for forward problems using PFNNs, considering various number of layers and the errors in the domain and boundary. With these examples for both the inverse and forward problems, we aim to show in detail how the theoretical framework presented above can be used to construct the PFNNs and can create accurate and explainable neural network models, whose properties and performance can be understood through the lens of classical numerical analysis for the solution of both forward and inverse problems of elliptic PDEs.
\par We now proceed with the implementation of the methodology. The calculations in the examples that follow were performed in \emph{Matlab}. Although the theoretical framework holds for arbitrary dimensions and boundary geometries, for this first introduction to Potential Fredholm Neural Networks, we will consider examples of PDEs in 2D with a simple circular disc domain, as well as a PDE in 3D with a spheroid boundary. Higher dimensional cases with alternative boundary geometries will be considered in future work, where we will focus solely on details regarding numerical implementation.  

\subsection{2D Poisson inverse source problem}\label{inv-ex-poi-sec}
We consider the Poisson PDE on the 2D disc, $\Omega:= \{(x_1, x_2):x_1^2 + x_2^2 \leq 1\}$: 
\begin{eqnarray}\label{inv-ex}
    \Delta u(x) = \psi(x), \quad x=(x_1,x_2) \in\Omega,
\end{eqnarray}
with boundary condition $f(x_1, x_2) = 2x_2, x \equiv (x_1, x_2) \in \partial \Omega$, and unknown source term $\psi: {\Omega} \rightarrow \mathbb{R}$. Working in polar coordinates,
we suppose we have data of the solution over a uniform $(r,\phi)$-grid of size $20\times 20$ and our goal is to estimate the unknown source term $\psi(x)$.
For this example, rather than using the exact true solution, we generate training data $\tilde{u}(r,\phi)$ using the forward PFNN with $M = 100$ layers and $100$ nodes per layer. (The exact calculations for the forward problem are given in Appendix \ref{appendix-poisson}.)
\par To solve the inverse problem, we consider a shallow neural network as the model for the unknown source term $\psi_{\theta}(x)$, with a single hidden layer, 20 nodes and a $\tanh$ activation function. Then, training is performed using Algorithm \ref{alg:inverse} where for the underlying PFNN we consider $M=100$ layers to ensure convergence of the fixed-point scheme solving the BIE, with a $\phi-$grid of size 100 corresponding to the interior nodes (we note that, of course, other hyperparameters of the Fredholm NN used for training can be chosen; we select these similar to our data generation for simplicity in this illustrative example). We perform $50$ runs of this model training process using a Levenberg-Marquardt (LM) optimizer with 600 iterations, and a Tikhonov regularization constant $\lambda_{reg} = 1.0$E$-12$ (note that a small regularization is required here to ensure appropriate trade-off between the data loss and the regularizing term).
\par Across the 50 models runs, the average  training MSE is $8.48$E$-07$. The 10-90th percentile interval is $[8.65$E$-08, 2.93$E$-06]$ and median $2.39$E$-07$. Furthermore, we test the model on an unseen grid $\Omega_{test}:= \{(r_i,\phi_j)\}_{i,j=1,\dots,50}$, over which the PDE is solved using the learned source function $\psi_{\theta}(x)$ to pass through the PFNN, to $\hat{u}(\cdot ; \psi_{\theta^*} )$ and compare the result to the true solution of the PDE, which is $u(x_1,x_2) = x_2(x_1^2+x_2^2)(1+x_1^2+x_2^2)$. We show a breakdown of the mean absolute error (MAE) and $L^{\infty}$ errors within the interior of the domain and across the boundary in Table \ref{tab:error_metrics_inverse} (for brevity, these results are shown for the test grid; the results for the training grid are of the same order of magnitude). We note that of particular importance are the errors on the boundary, where we know that $u(x_1, x_2) = f(x_1, x_2) = 2x_2$. Specifically, we see that the proposed modelling framework achieves machine precision on the boundary, with absolute errors of the order of E$-15$. This shows that via the PFNN architecture we are able to construct the model such that the boundary conditions are satisfied identically, as required. This is achieved without imposing any explicit additional terms for the boundary conditions in the loss function, but rather simply from the construction of the PFNN which is built to satisfy the boundary conditions with extremely high accuracy (recall Theorem \eqref{UAT-inverse}). 
\par  Finally, in Fig. \ref{fig:example_inv} we display the contour of learned source term function corresponding to the model achieving the best training MSE ($6.70$E$-08$), along with the absolute error over the training and test grids. These results show not only a high accuracy in the resulting solution of the PDE, but also powerful generalization capability.


\begin{figure}[ht]
    \centering
    \begin{subfigure}[t]{0.33\textwidth}
        \centering
        \includegraphics[width=\textwidth]{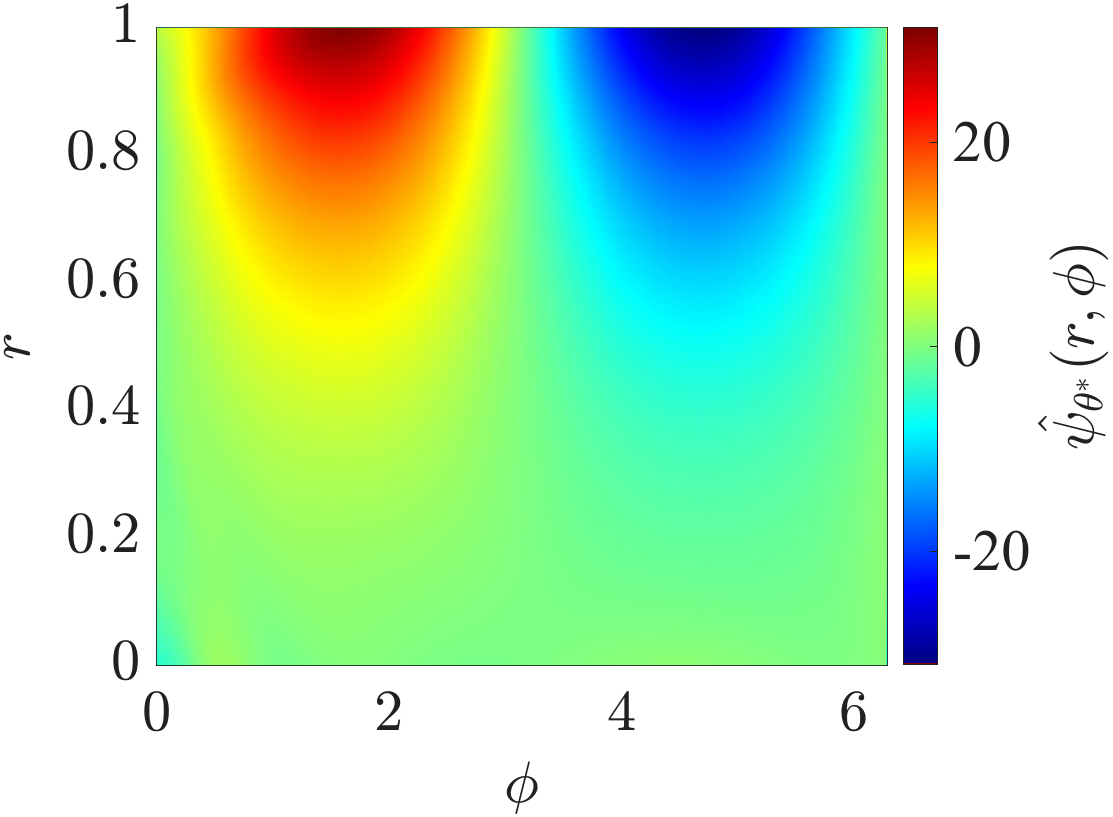}
        \caption{The learned source term $\psi_{{\theta^*}}$, where $\theta^*$ is the optimal set of model parameters.}
    \end{subfigure}
    \begin{subfigure}[t]{0.33\textwidth}
        \centering
        \includegraphics[width=\textwidth]{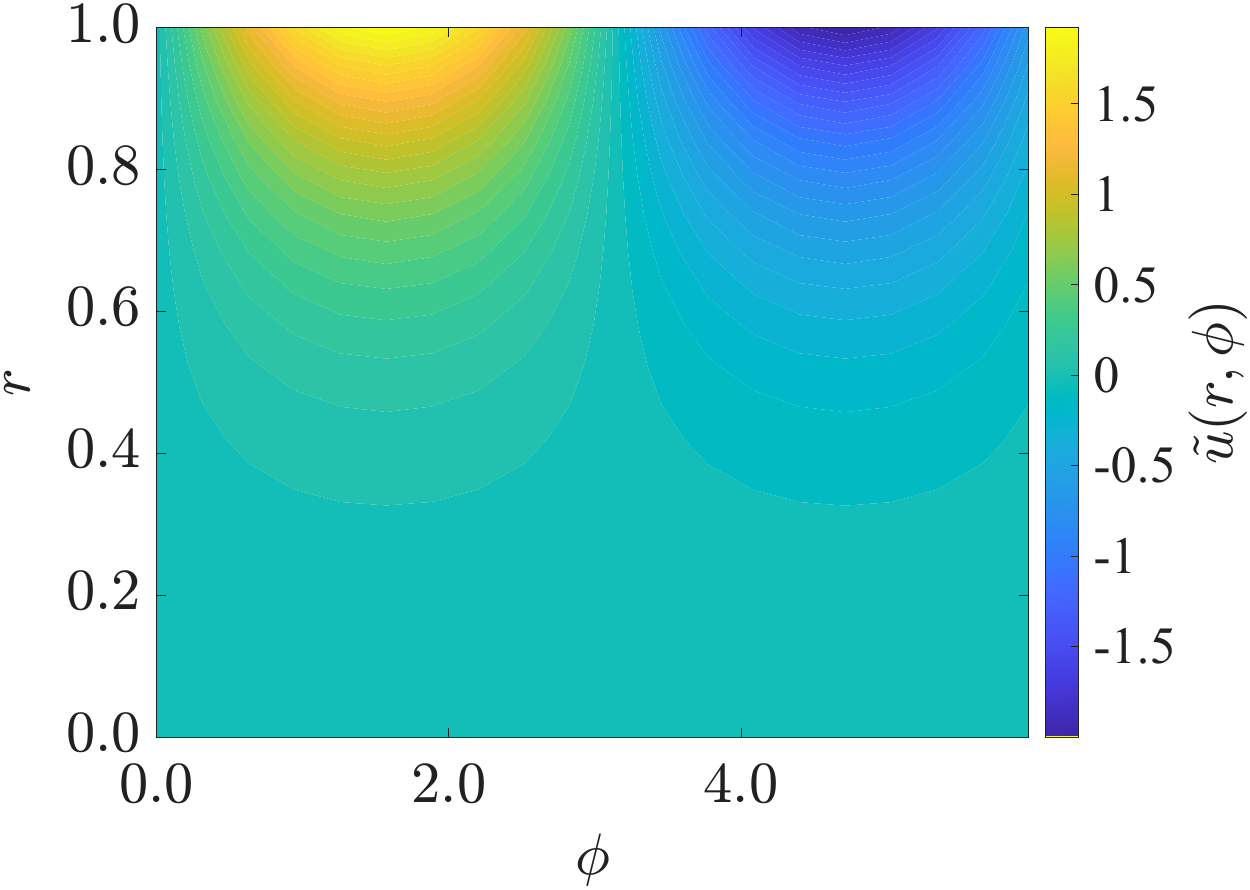}
        \caption{The absolute error between the solution of the PDE using the learned source term and the training data, $|\hat{u} - \tilde{u}|$.}
    \end{subfigure}
    \begin{subfigure}[t]{0.33\textwidth}
        \centering
        \includegraphics[width=\textwidth]{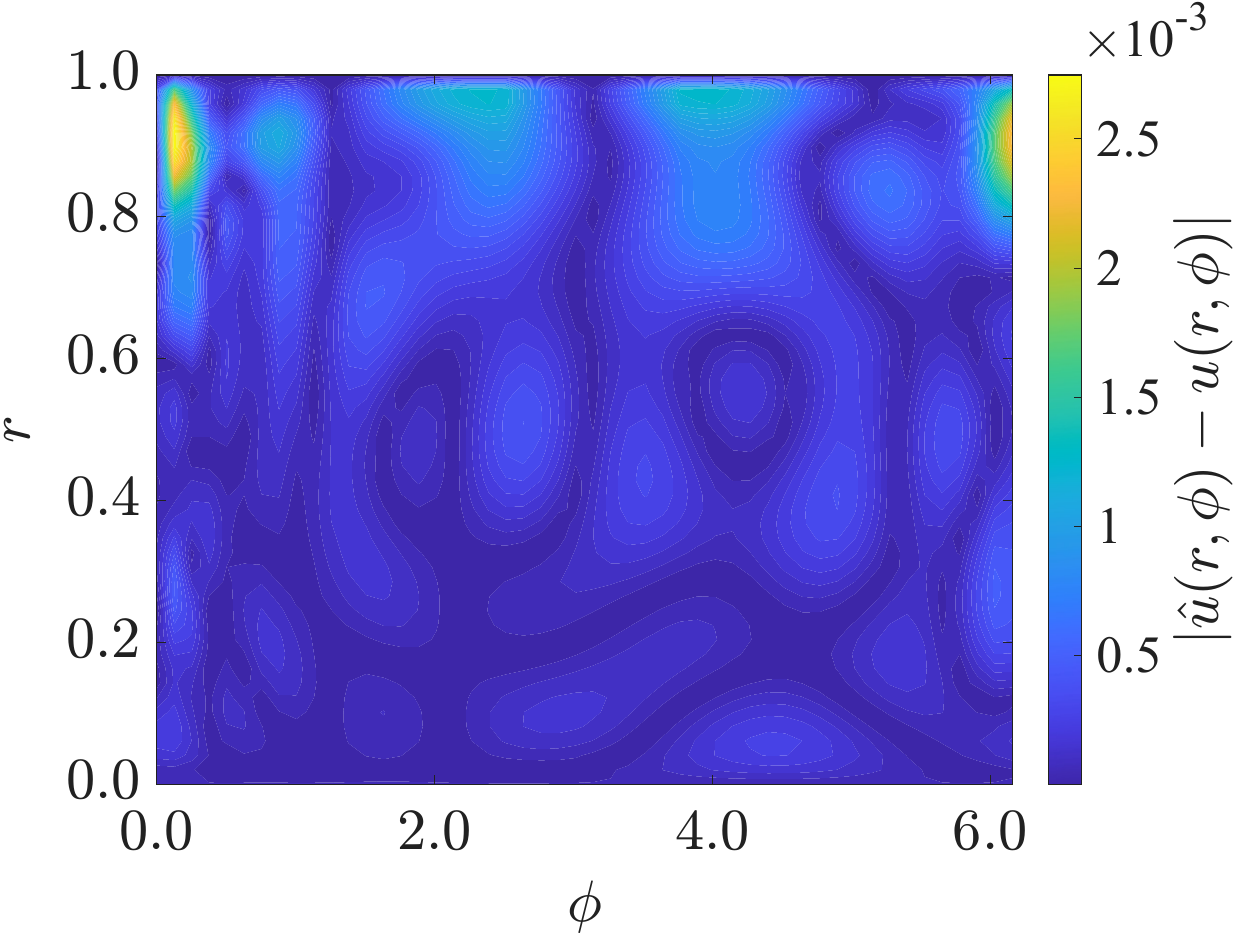}
        \caption{The absolute error between the solution of the PDE using the learned source term and the true solution on a test grid, $|\hat{u} - u|$.}
    \end{subfigure}
    \caption{Results of the Poisson inverse source problem \ref{inv-ex-poi-sec}.}
    \label{fig:example_inv}
\end{figure}

\subsection{2D Helmholtz}\label{inv-ex-helm-sec}
We now consider the PDE, again in the 2D disc, of the form: 
\begin{equation}\label{inv-ex-helm}
    \Delta u(x) - \lambda u = \psi(x),
\end{equation}
with boundary condition $f(x) \equiv f(x_1, x_2) = x_1^3 + 2x_1^2 -2$ for $x \in  \partial \Omega$.
We are interested in approximating the source term with a model $\psi_{\theta}(x)$, given data from the true solution $u(x) = x_1^3 - 2x_2^2$. (Note that the details of solving the forward version of the problem with the known source term are given in Appendix \ref{appendix-helm}.)

As in the example above, we consider data from the true solution on a grid of size $20\times 20$. For the source model $\psi_{\theta}(x)$, we use a neural network with a single hidden layer, 20 neurons and a $\tanh$ activation function. For training, we use Algorithm \ref{alg:inverse} where the forward step is done using a PFNN with $M=100$ and 100 nodes per layer correpsonding to the $\phi-$ grid in the BIE. Note that during training, both integrals in the final layer of the PFNN corresponding to the terms $\int_{\Omega}\lambda \delta \Phi(x,y)dy$ and $\int_{\Omega} \Phi(x,y)\psi(y)dy$ are evaluated via trapezoidal quadrature on the $100 \times 100$ uniform grid in $(r, \theta)$. Training is performed with the LM optimizer using 200 iterations.

We run again 50 training instances, and the mean MSE loss was $7.08$E$-06$. The median and 90-10th percentile intervals are $8.94$E$-06$ and $[1.24$E$-07$, $1.04$E$-05]$. We test the model on an unseen test grid of size $50\times 50$ and compare to the true solution and report the MAE and $L^{\infty}$ errors in Table \ref{tab:error_metrics_inverse}. We highlight again that the approximation on the boundary achieves machine precision.

\begin{figure}[ht]
    \centering
    \begin{subfigure}[t]{0.32\textwidth}
        \centering
        \includegraphics[width=\textwidth]{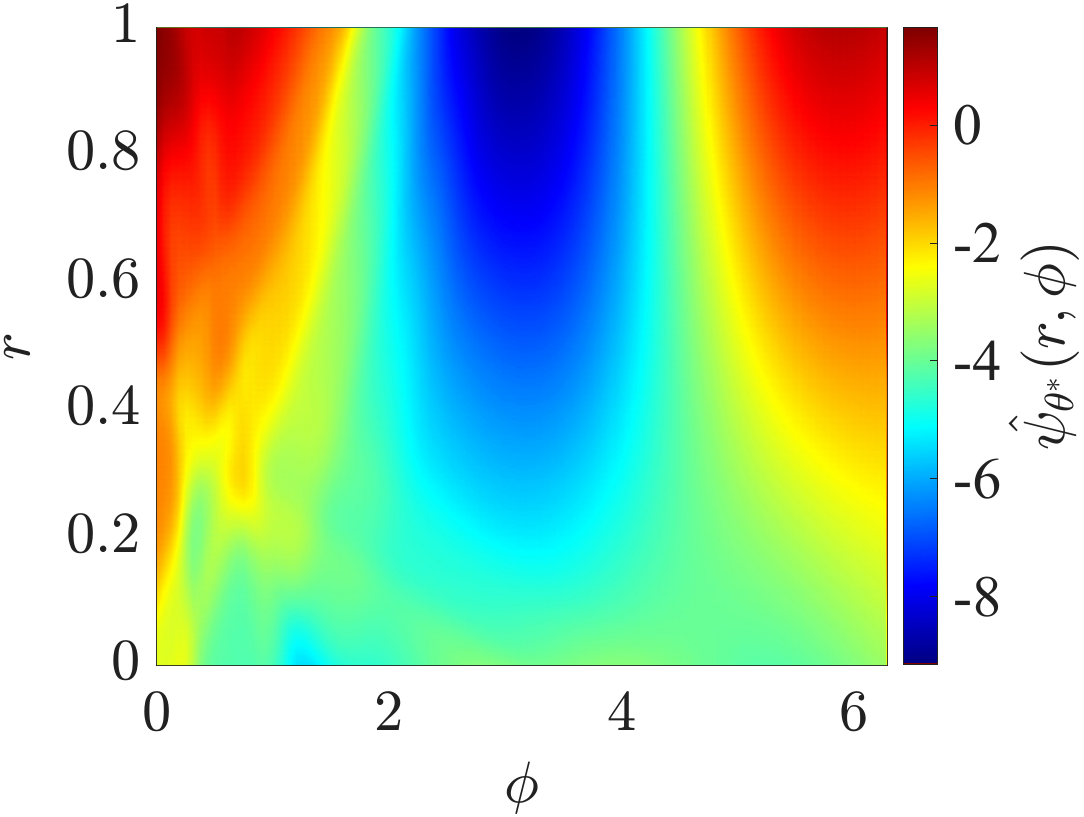}
        \caption{The learned source term $\psi_{{\theta^*}}$, where $\theta^*$ is the optimal set of model parameters.}
    \end{subfigure}
    \begin{subfigure}[t]{0.33\textwidth}
        \centering
        \includegraphics[width=\textwidth]{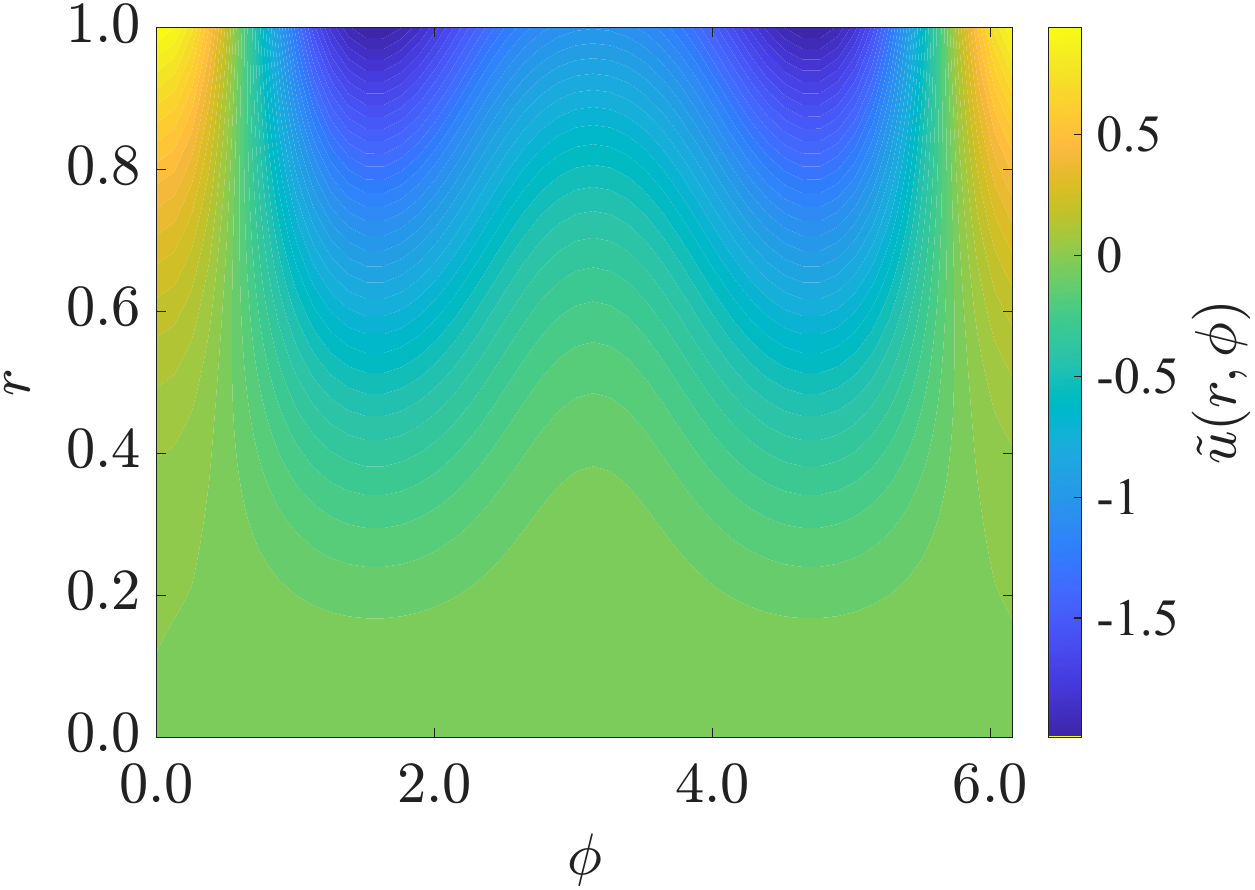}
        \caption{The absolute error between the solution of the PDE using the learned source term and the training data, $|\hat{u} - \tilde{u}|$.}
    \end{subfigure}
    \begin{subfigure}[t]{0.33\textwidth}
        \centering
        \includegraphics[width=\textwidth]{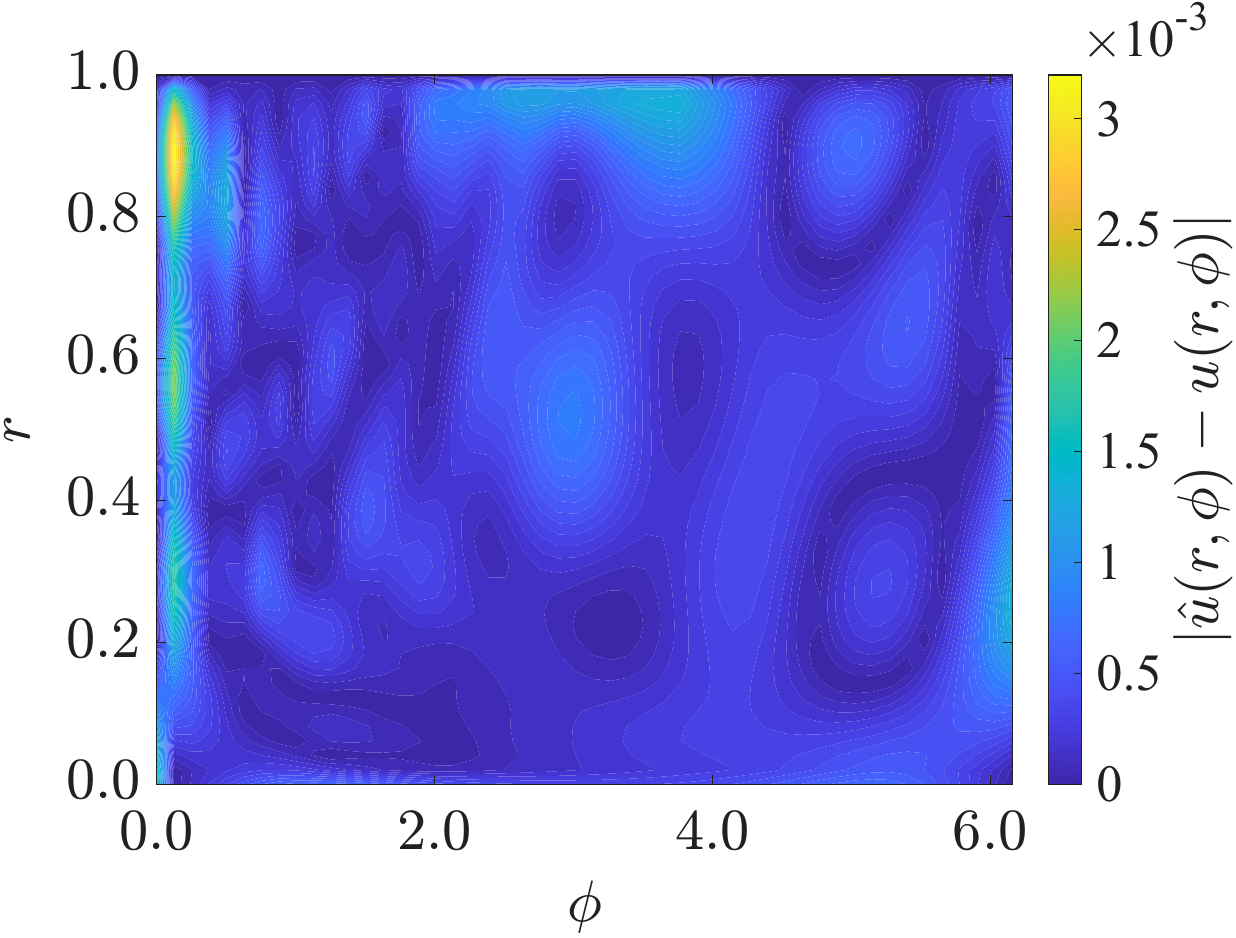}
        \caption{The absolute error between the solution of the PDE using the learned source term and the true solution on a test grid, $|\hat{u} - u|$.}
    \end{subfigure}
    \caption{Results of the Poisson inverse source problem \ref{inv-ex-helm-sec}.}
    \label{fig:example_inv}
\end{figure}

\subsection{2D semi-linear elliptic PDE}\label{inv-ex-sl-sec}
Finally, we consider the semi-linear equation of the form:
\begin{eqnarray}\label{inv-ex-sl}
    \Delta u(x) = \psi(x,u), 
\end{eqnarray}
on the 2D unit disc, with boundary condition $f(x) = 0, x  \in \partial \Omega$.

Suppose we have data from the true solution ${u}(r, \phi)$ over a $(r,\phi)$-grid of size $20\times 20$ and our goal is to estimate the unknown source term $\psi(x,u)$. For this example, we will use data from the true solution given by $u(r,\phi) = r^2(1-r^2)\sin(2\phi)$.
\par To learn the source function $\psi_{\theta}(x,u)$, we use a shallow neural network with a single hidden layer, 60 nodes and a $\tanh$ activation function. In this case, the input is three-dimensional $(x_1, x_2, u)$. For training we use Algorithm \ref{alg:inverse} with a PFNN with $M=100$ layers and a $\phi-$grid of size 100 corresponding to the interior nodes. The volume integrals involving the learned source are computed via 2D trapezoidal quadrature on a $150\times 150$ integration grid. Again, $50$ runs of this model training process using the LM optimizer with 150 iterations, and a Tikhonov regularization constant $\lambda_{reg} = 1.0$E$-12$.
\par The average training MSE is $3.05$E$-05$ and the corresponding 10-90th percentile intervals are $[4.26$E$-07, 4.49$E$-05]$ with median $2.13$E$-06$. Testing the model on the unseen grid produces the error metrics on the domain interior and boundary as given in Table \ref{tab:error_metrics_inverse}. In Fig. \ref{fig:example_inv} we display the contour of the learned source term function corresponding to the model achieving the best training MSE, the predicted solution $\hat{u}(x;\theta^*)$ and the absolute error when comparing to the true solution on the test grid.

\begin{figure}[ht]
    \centering
    \begin{subfigure}[t]{0.33\textwidth}
        \centering
        \includegraphics[width=\textwidth]{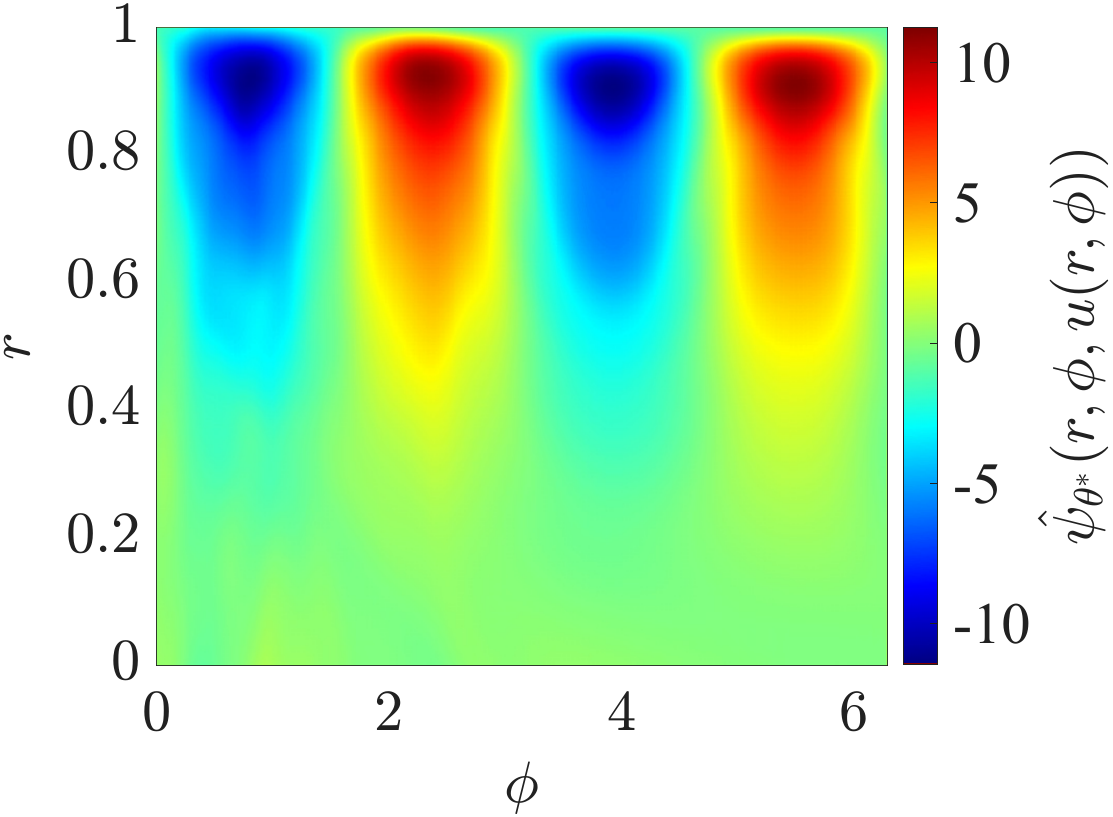}
        \caption{The learned source term $\psi_{{\theta^*}}$, where $\theta^*$ is the optimal set of model parameters.}
    \end{subfigure}
    \begin{subfigure}[t]{0.33\textwidth}
        \centering
        \includegraphics[width=\textwidth]{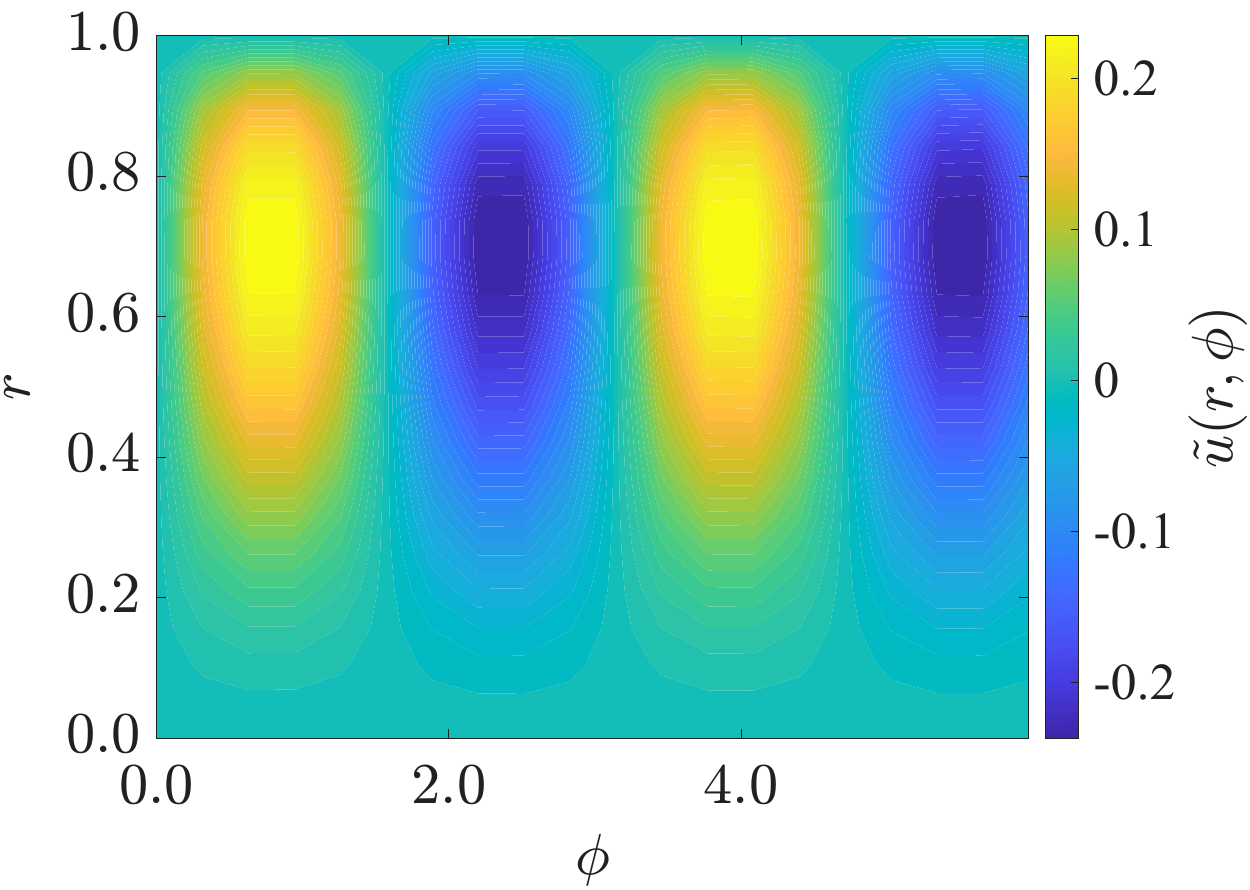}
        \caption{The absolute error between the solution of the PDE using the learned source term and the training data, $|\hat{u} - \tilde{u}|$.}
    \end{subfigure}
    \begin{subfigure}[t]{0.33\textwidth}
        \centering
        \includegraphics[width=\textwidth]{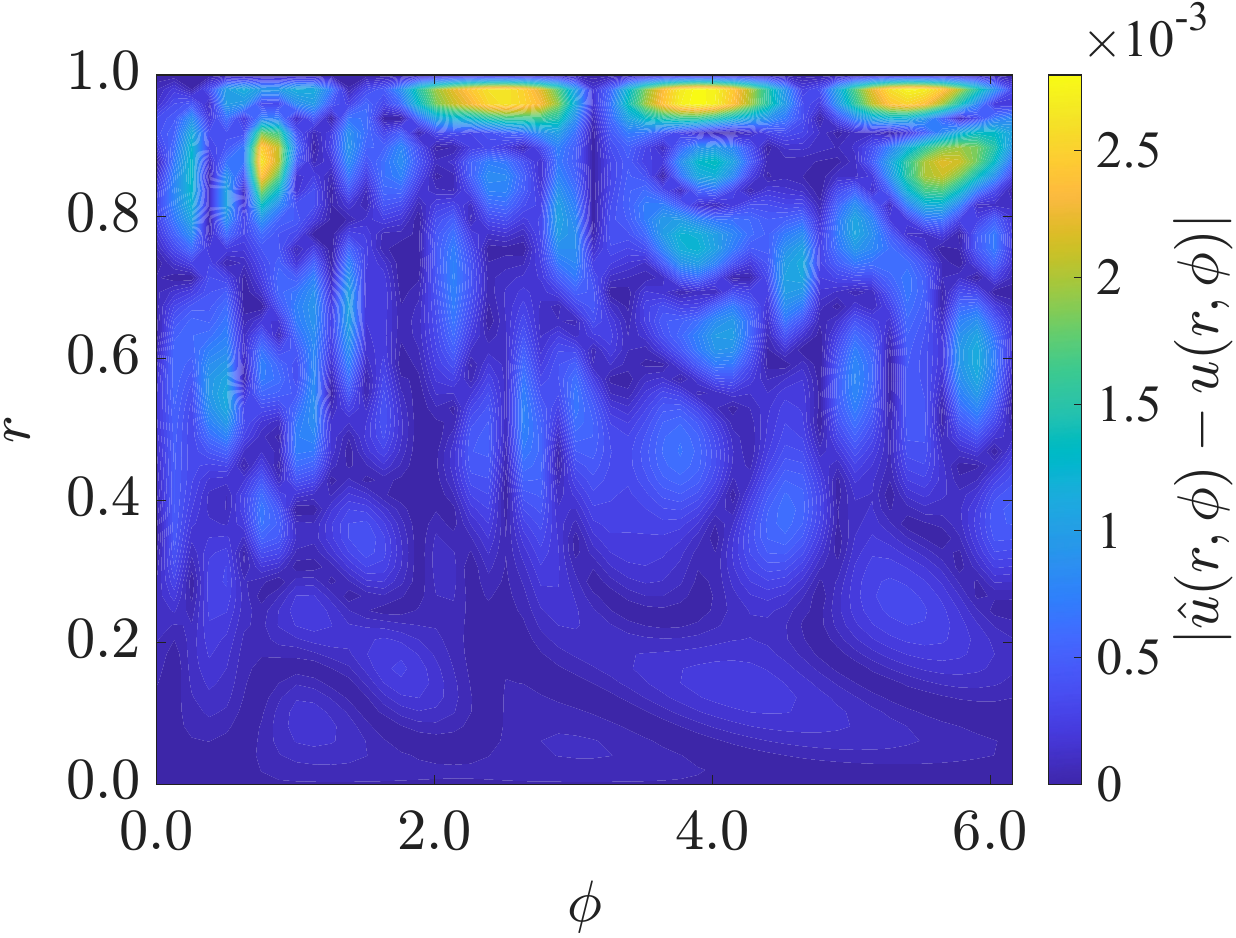}
        \caption{The absolute error between the solution of the PDE using the learned source term and the true solution on a test grid, $|\hat{u} - u|$.}
    \end{subfigure}
    \caption{Results of the inverse semi-linear inverse source problem \eqref{inv-ex-sl}.}
    \label{fig:example_inv_semi_linear}
\end{figure}

\subsection{3D Poisson PDE with spheroid boundary}\label{sec:example-3d-poisson-inverse}
We now extend the inverse source framework to three dimensions and to a 
geometry beyond the disk and sphere. We consider the 3D Poisson PDE:
\begin{equation}\label{inv-ex-3d}
    \Delta u(x) = \psi(x), \qquad x \in \Omega \subset \mathbb{R}^3,
\end{equation}
with boundary condition $f(x) = u(x)|_{\partial \Omega}$ given exactly, 
where $u$ is the manufactured solution defined below. The domain is 
the prolate spheroid
\begin{equation}
    \Omega = \left\{ x \in \mathbb{R}^3 \,:\,
    \frac{x_1^2}{a^2} + \frac{x_2^2}{a^2} + \frac{x_3^2}{c^2} < 1 \right\},
    \qquad a = 1,\; c = 1.5,
\end{equation}
and the unknown is the source term $\psi: \Omega \rightarrow \mathbb{R}$.

\par To manufacture a PDE with a closed-form solution suitable for straightforward error analysis, we consider a smooth isotropic Gaussian source term, given by:
\begin{equation}\label{eq:psi-true-3d}
    \psi(x) = \alpha \exp(-\beta |x|^2),
    \qquad \alpha = 1.0, \quad \beta = 5.0,
\end{equation}
and use the fact that the Newton potential $\mathcal{N}_{\psi}:= \int_{\Omega} \Phi(x,y)\psi(y)dy$ is a solution of the PDE $\Delta \mathcal{N}_{\psi} = \psi$, by definition. Specifically, with the selected source term given by \eqref{eq:psi-true-3d}, we get a closed form solution:
\begin{equation}\label{eq:newton-pot-gaussian}
    \mathcal{N}_{\psi}(r) = -\alpha\!\left[
    \frac{\sqrt{\pi}}{4\beta^{3/2}\, r}\,\mathrm{erf}\!\left(\sqrt{\beta}\, r\right)
    - \frac{e^{-\beta r^2}}{2\beta}
    \right]
    - \frac{\alpha}{2\beta}\, e^{-\beta r^2},
    \qquad r > 0,
\end{equation}
for $r = |x|$, and with the limit $\mathcal{N}_{\psi}(0) = -\alpha/(2\beta)$. Hence, we can choose the manufactured solution given by:
\begin{equation}
    u(x) = \mathcal{N}_{\psi}(|x|) + x_1 x_2 x_3.
\end{equation}
(Note that although $\psi$ does not have compact support, it decays 
to $\exp(-\beta c^2) \approx 1.3$E$-05$ at the spheroid's polar 
tip, so the global Newton potential \eqref{eq:newton-pot-gaussian} differs from the $\Omega$-restricted one  by a negligible amount relative to the inverse-problem error scale, and so we use the global formula as the reference for the closed form solution.)

\par Now, the 3D fundamental solution is $\Phi(x,y) = -\frac{1}{4\pi |x-y|}$, and 
unlike the 2D disk case, $\partial \Phi(x^*,y)/\partial n_y$ is no 
longer a constant function of $y$ on $\partial\Omega$, and furthermore exhibits a weak singularity. This requires a correction to ensure that the numerical value of the integral matches the theory (see Appendix \ref{appendix-poisson-3d} for the exact calculations). We note that such a correction is only one method of handling the weak singularity and many other approaches can be considered, which will be the subject of future work. 

\par For training, we again assume the boundary function $f$ is known. The training data consists of $200$ interior observations of the true $u$, and in this example, we prefer to sample uniformly in $\Omega$ (rather than consider a predefined grid of observations). To solve the inverse problem, we model the unknown source term as a  shallow neural network $\psi_\theta(x)$ with three inputs, a single hidden layer of $20$ nodes, and $\tanh$  activation. Training follows Algorithm \ref{alg:inverse}, using an
underlying PFNN uses $M = 200$ layers, a boundary discretization $N_\theta = N_\varphi = 60$, volume 
discretization of size $25\times2 0 \times 30)$, Krasnosel'skii-Mann constant $\kappa = 0.5$, and Tikhonov regularization $\lambda_{reg} = 1.0\text{E}{-10}$. Again, we perform $50$ independent runs of the training procedure with the Levenberg-Marquardt optimizer with $100$ iterations. 

\par Across the 50 model runs, the average training MSE is $1.79$E$-07$. The 10–90th percentile interval is $[8.02$E$-08$, $2.75$E$-07]$ with median $1.76$E$-07$. We report the mean absolute error (MAE) and $L^\infty$ errors of the recovered solution $\hat u(\,\cdot\,;\psi_{\theta^*})$ against the analytic reference $u$ on a 3D test grid $\Omega_\text{test}$ in the interior, and against $f$ on the boundary $\partial \Omega$, in Table \ref{tab:error_metrics_inverse}.

As in our 2D examples, we observe that the PFNN architecture satisfies the boundary condition to machine precision without any explicit boundary residual term in the loss function. We highlight that this holds even on the non-spherical 3D geometry, demonstrating that the row-sum-corrected discrete BIE preserves the exact boundary-data agreement that we proved in Theorem \eqref{UAT-inverse}. 

In Fig. \ref{fig:example_inv_3d} we display, for the model achieving the best training MSE, contour plots on a representative 2D slice through the spheroid (in spheroidal-polar coordinates $(\varphi, s) \in [0, 2\pi] \times [0, 1]$, with $s$ the radial fraction within the slice and the boundary $\partial \Omega$ at $s = 1$): the recovered source $\psi_{\theta^*}$, the predicted solution $\hat{u}$ and the absolute error and the absolute error $|\hat u - u|$ on the test grid.

\begin{figure}[ht]
    \centering
    \begin{subfigure}[t]{0.33\textwidth}
        \centering
        \includegraphics[width=\textwidth]{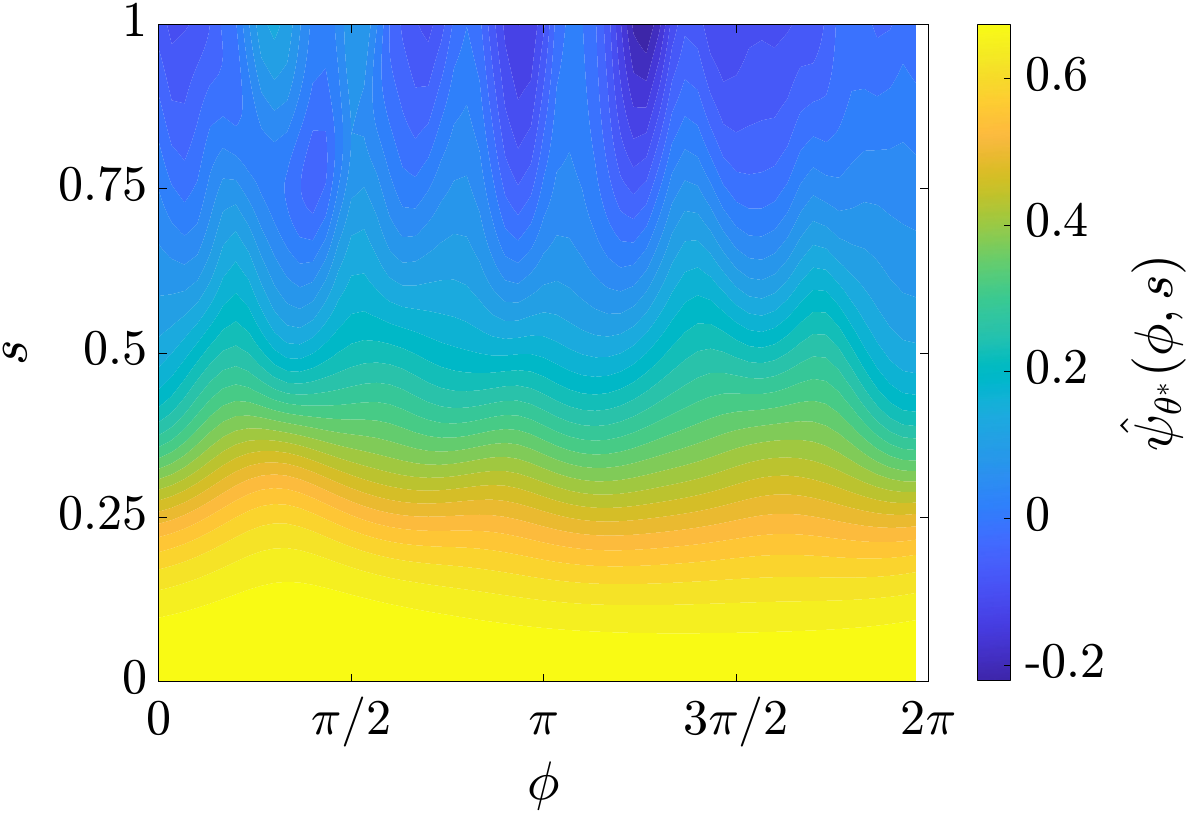}
        \caption{The learned source term $\hat{\psi}_{\theta^*}(\phi, s)$, where $\theta^*$ is the optimal set of model parameters.}
    \end{subfigure}
    \begin{subfigure}[t]{0.33\textwidth}
        \centering
        \includegraphics[width=\textwidth]{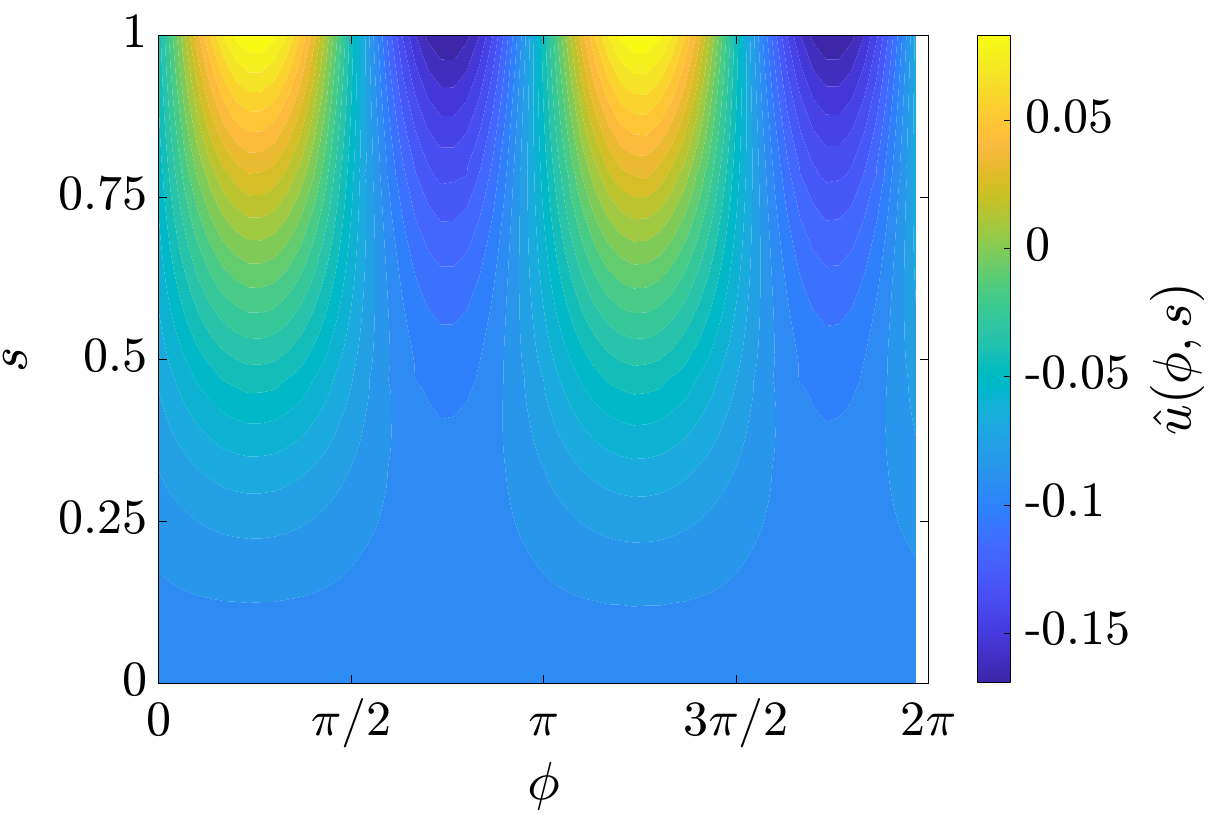}
        \caption{The solution of the PDE on the slice obtained using the learned source, $\hat{u}(\phi, s)$.}
    \end{subfigure}
    \begin{subfigure}[t]{0.33\textwidth}
        \centering
        \includegraphics[width=\textwidth]{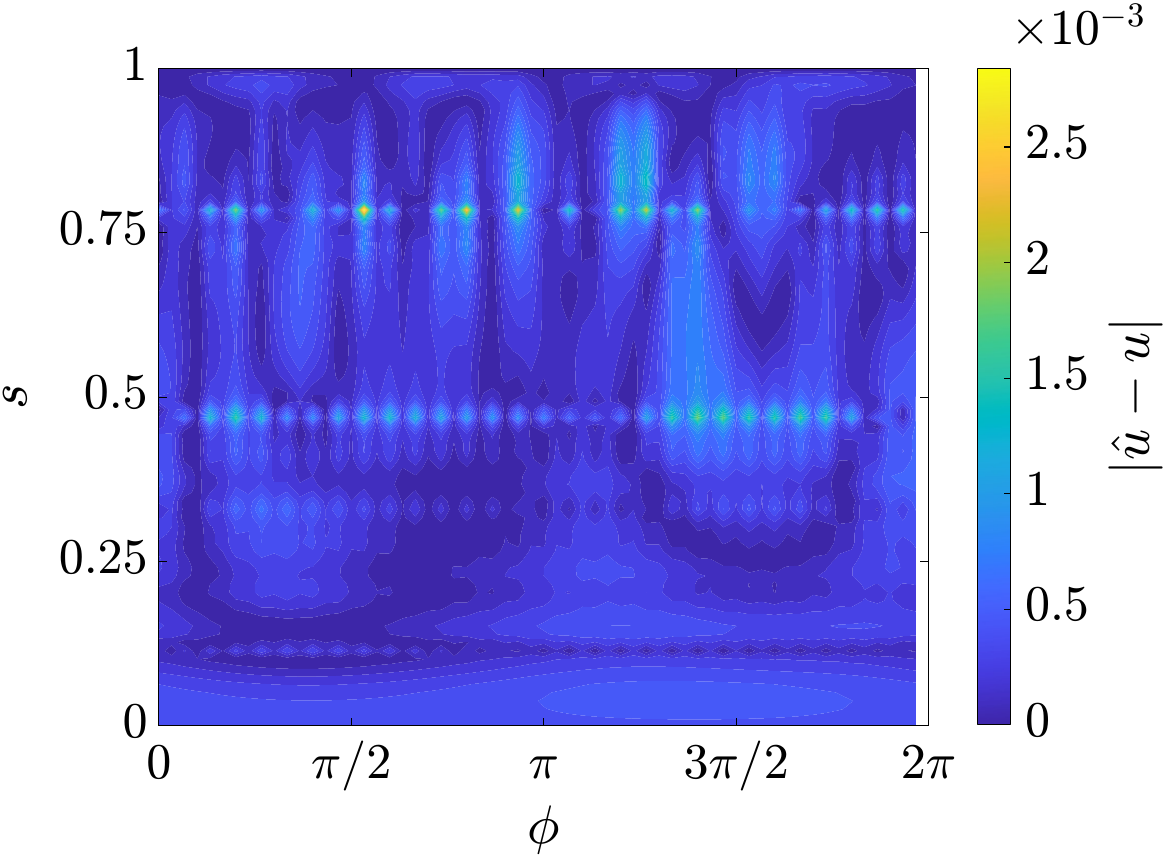}
        \caption{Contour of the absolute error between the solution of the PDE using the learned source term and the true solution on a test grid, $|\hat{u} - u|$.}
    \end{subfigure}
    \caption{Results of the 3D inverse source problem \eqref{inv-ex-3d} on 
the prolate spheroid, for the model achieving the best training MSE across the 50 runs. All four panels are evaluated on a 2D slice through the spheroid interior at $z \approx 0.32$, parametrized in polar coordinates $(\phi, s)$ with $\phi \in [0, 2\pi]$ and $s \in [0, 1]$. The slice is aligned with the boundary quadrature grid so that the $s = 1$ row of panel (c) reflects the BIE residual at quadrature nodes, where the boundary condition is satisfied to machine precision by construction.}
    \label{fig:example_inv_3d}
\end{figure}


\begin{table}[ht]
\centering
\begin{tabular}{@{}ll|ccc|ccc@{}}
\toprule
Example & Region ($\mathcal{A})$
  & \multicolumn{3}{c|}{MAE($\mathcal{A}$) $= \textbf{mean}_{x \in \mathcal{A}} |\hat{u} - u|$}
  & \multicolumn{3}{c}{$L^{\infty}(\mathcal{A}) = \textbf{max}_{x \in \mathcal{A}}|\hat{u} - u|$} \\
\cmidrule(lr){3-5}\cmidrule(lr){6-8}
& & Mean & Median & [10th\%,\,90th\%]
  & Mean & Median & [10th\%,\,90th\%] \\
\midrule
\multirow{2}{*}{Ex. \ref{inv-ex-poi-sec}}
  & $\Omega_{\mathrm{int}}$
  & 7.13E-04 & 5.14E-04 & [3.50E-04,\,1.50E-03]
  & 7.79E-03 & 6.5E-03 & [3.86E-03,\,1.48E-02] \\
  & $\partial\Omega$
  & 2.52E-15 & 2.48E-15 & [1.46E-15,\,4.28E-15]
  & 8.12E-15 & 7.99E-15 & [5.55E-15,\,1.31E-14] \\
\midrule
\multirow{2}{*}{Ex. \ref{inv-ex-helm-sec}}
  & $\Omega_{\mathrm{int}}$
  & 1.78E-03 & 2.13E-03 & [4.27E-04,\,2.45E-03]
  & 8.78E-03 & 1.04E-02 & [3.04E-03,\,1.11E-02] \\
  & $\partial\Omega$
  & 1.31E-15 & 1.33E-15 & [9.55E-16,\, 1.51E-16]
  & 5.85E-15 & 5.77E-15 & [3.55E-15,\,7.99E-15] \\
\midrule
\multirow{2}{*}{Ex. \ref{inv-ex-sl-sec}}
  & $\Omega_{\mathrm{int}}$
  & 2.24E-03 & 1.08E-03 & [5.26E-04,\,4.36E-03]
  & 1.48E-02 & 8.58E-03 & [4.46E-03,\,2.81E-02] \\
  & $\partial\Omega$
  & 2.56E-16 & 2.47E-16 & [1.72E-16,\,3.67E-16] 
  & 9.58E-16 & 9.16E-16 & [5.83E-16,\,1.39E-15] \\
  \midrule
  \multirow{2}{*}{Ex. \ref{sec:example-3d-poisson-inverse}}
  & $\Omega_{\mathrm{int}}$
  & 5.72E-04 & 5.61E-04 & [4.25E-04,\,7.61E-04]
  & 1.12E-02 & 1.11E-02 & [1.09E-02,\,1.15E-02] \\
  & $\partial\Omega$
  & 1.17E-16 & 1.17E-16 & [1.15E-16,\,1.12E-16] 
  & 1.40E-15 & 1.39E-16 & [1.21E-16,\,1.62E-15] \\
\bottomrule
\end{tabular}
\caption{Error metrics across the 50 training runs in the interior and on the boundary of the test grid.}
\label{tab:error_metrics_inverse}
\end{table}

\section{Conclusions and Discussion}\label{conclusion}
The concept of the representation of iterative numerical analysis algorithms with tailor-made deep neural network architectures dates back to 1990s' as seen in \cite{rico1992discrete}, where a feedforward multilayer NN was used to implement the fourth-order Runge-Kutta method. This area of research remains active and is evolving towards numerical analysis-informed/ explainable deep learning models \cite{guo2022personalized,zappala2022neural,lohner2023neural,doncevic2024recursively}. The idea of the framework is to construct DNNs to directly model the iterative process, with its parameters reflecting the specific steps and coefficients of the numerical method. The weights and biases as well as other hyperparameters of the DNNs, are precomputed based on the iterative method, so no training on a loss function is necessary. This tailor-made \textit{numerical analysis-informed} approach removes the curse of dimensionality arising in iterative optimization.

 Within this line of research, we build upon the method that we introduced in \cite{georgiou2024fredholm} for solving integral equations and extend the framework by developing the Potential Fredholm Neural Network and the Recurrent Potential Fredholm NN for solving linear and semi-linear elliptic PDEs, both in the case of the forward and inverse problem. This method relies on the reformulation of classical numerical methods, such as fixed point iterations in the Banach space, and numerical integration schemes based on the boundary integral method (BIE), into the neural network framework. We show that the underlying mathematical theory can be used to construct a DNN for which we can also obtain a-priori, explainable error bounds that allow us to analyze the contributing factors that create these errors. In total, the proposed framework describes a method of constructing DNNs whose architecture is directly related to numerical methods and therefore significantly enhances model explainability, as well as accuracy. In our examples, we show that this approach achieves high accuracy throughout the domain and near-machine precision the boundary of the PDEs, thereby satisfying identically the boundary conditions, a strict requirement in the field of numerical methods for PDEs. By giving rigorous constructions and proofs of the convergence and error bounds, we showcase how custom neural network architectures can be used to solve issues such as the accuracy on the boundary without the need of hard-encoding the conditions, as well as interpretable (in the ''classical'' numerical methods sense) error analysis.

\par Concluding, we believe the proposed framework creates a path of interesting future research. Such directions include the high-dimensional PDEs, highly nonlinear systems and time-dependent PDEs, as well as applications in the field Scientific Machine Learning, such as Uncertainty Quantification (UQ). 

\section*{Acknowledgements}
K.G. acknowledges support from the PNRR MUR Italy, project PE0000013-Future Artificial Intelligence Research-FAIR. 
C.S. acknowledges partial support from the PNRR MUR Italy, projects PE0000013-Future Artificial Intelligence Research-FAIR \& CN0000013 CN HPC - National Centre for HPC, Big Data and Quantum Computing, and from the GNCS group of INdAM. A.N.Y. acknowledges the use of resources from the Stochastic Modelling and 
Applications Laboratory, AUEB.

\appendix
\section{Appendix: Solution of the Forward Problem}\label{appendixx}
\subsection{Poisson PDE}\label{appendix-poisson}
Recall the form of the BIE (\ref{BIE}). To construct the integral equation, we have $\frac{\partial \Phi(x, y)}{\partial n_y} = n_y \cdot \nabla_y\Phi(x,y) $, with $n_y = \frac{y}{|y|}$. Therefore:
\begin{equation}\notag 
    \nabla_y \Phi(x, y) = \left( \frac{y_1 - x_1}{2 \pi |x - y|^2}, \frac{y_2 - x_2}{2 \pi |x - y|^2} \right),
\end{equation}
and writing in terms of the arc-length parameterization we get:
\begin{equation}
    \frac{\partial \Phi(x, y(\theta_2))}{\partial n_y} = \frac{r_2^2 - r_1 r_2 \cos (\theta_1 - \theta_2)}{2 \pi |x - y|^2 |y|}.
\end{equation}
For the BIE in particular we have $x, y \in \partial \Omega$, and the kernel in the integral operator will become:
\begin{equation}\notag
    \frac{\partial \Phi(x, y(s))}{\partial n_y} = \frac{1 - \cos(\theta_1 - \theta_2)}{2 \pi (2 - 2\cos(\theta_1 - \theta_2))} = \frac{1}{4 \pi},
\end{equation}
and, finally, the integral term is given by: 
\begin{flalign}\label{poisson-fund}
     \int_{\Omega} \Phi(x, y) \psi(y) dy &= \int_0^{2\pi} \int_0^1 \frac{1}{2\pi} \ln |x - y| \psi(y) dy =  \notag \\ &=\int_0^{2\pi} \int_0^1 \frac{1}{2\pi} \ln \left( (r_1 \cos \theta_1 - r_2 \cos \theta_2)^2 + (r_1 \sin \theta_1 - r_2 \sin \theta_2)^2 \right)^{\frac{1}{2}}
    \psi(r_2, \theta_2) r_2 dr_2 d\theta_2.
\end{flalign}
Written out explicitly, the BIE becomes:
\begin{flalign}\label{poisson-ex}
    \beta(x) = 2 \Big( f(x) - \int_0^{2\pi} \int_0^1 \frac{1}{2\pi} \ln \big((\cos \theta_1 - r_2 \cos \theta_2)^2 + (\sin \theta_1 - r_2 \sin \theta_2)^2 \big)^{\frac{1}{2}} 
    &\psi(r_2, \theta_2) r_2 dr_2 d\theta_2 \Big)
    \notag \\ &- 2 \int_{\partial \Omega} \beta(\theta) \frac{1}{4\pi} d\theta,
\end{flalign}
and the final layer of the PFNN is as described in Proposition \ref{prop-poisson}. As we have noted, the double integral in \eqref{poisson-fund} is weakly singular at points within the integration domain. Hence, using naive numerical integration schemes may cause large errors. To account for this, we use an adaptive quadrature approach explicitly specifying the points at which singularities arise, i.e., at $r_2 = r_1$ and $\theta_2 = \theta_1$ (this specification can easily be implemented both in \emph{Matlab} as well as \emph{Python}). In this way, we compute the integral term required both for the BIE and for the final layer of the neural network in accordance to \eqref{potential-form-1}.  Finally, the term $\int_{\partial \Omega} \frac{1}{4\pi} \beta(y(s)) ds$ in \eqref{potential-form-1} is approximated with a simple Riemann sum on the selected grid that we use to estimate the boundary function $\beta(x)$.

Here, for our illustrations, we consider the following Poisson PDE:
 \begin{equation}
 \Delta u(x) = 2x_1,
     \label{ex1}
 \end{equation}
on the unit disc, ${\Omega} = \{x \in \mathbb{R}^2: x_1^2 + x_2^2 \leq 1 \}$ with boundary condition $f(x)=0$, for $x \in  \partial {\Omega}$, 
for which, an analytical solution exists, and it is given by: 
\begin{equation}\label{example-poisson}
u(x) = \frac{1}{4}x_1(x_1^2 + x_2^2 -1).
\end{equation}
We consider a $\phi-$grid for the BIE of size $N = 1000$ to ensure sufficiently accurate approximation of the boundary function $\beta(x_i),$ for $x_i = (1,\phi_i)$, and an $r-$ grid of size $100$. To show the effect of the number of hidden layers in the FNN on the overall error in the solution of the PDE, we use various layers selections $M$ in the estimation of the boundary function $\beta(x)$. The solution and error metrics $L_{\infty}$ and Mean Absolute Error (MAE) are given in Fig. \ref{fig:poisson_1}. 

\begin{figure}[ht]
    \centering
    \begin{subfigure}[t]{0.43\textwidth}
        \centering
        \includegraphics[width=\textwidth]{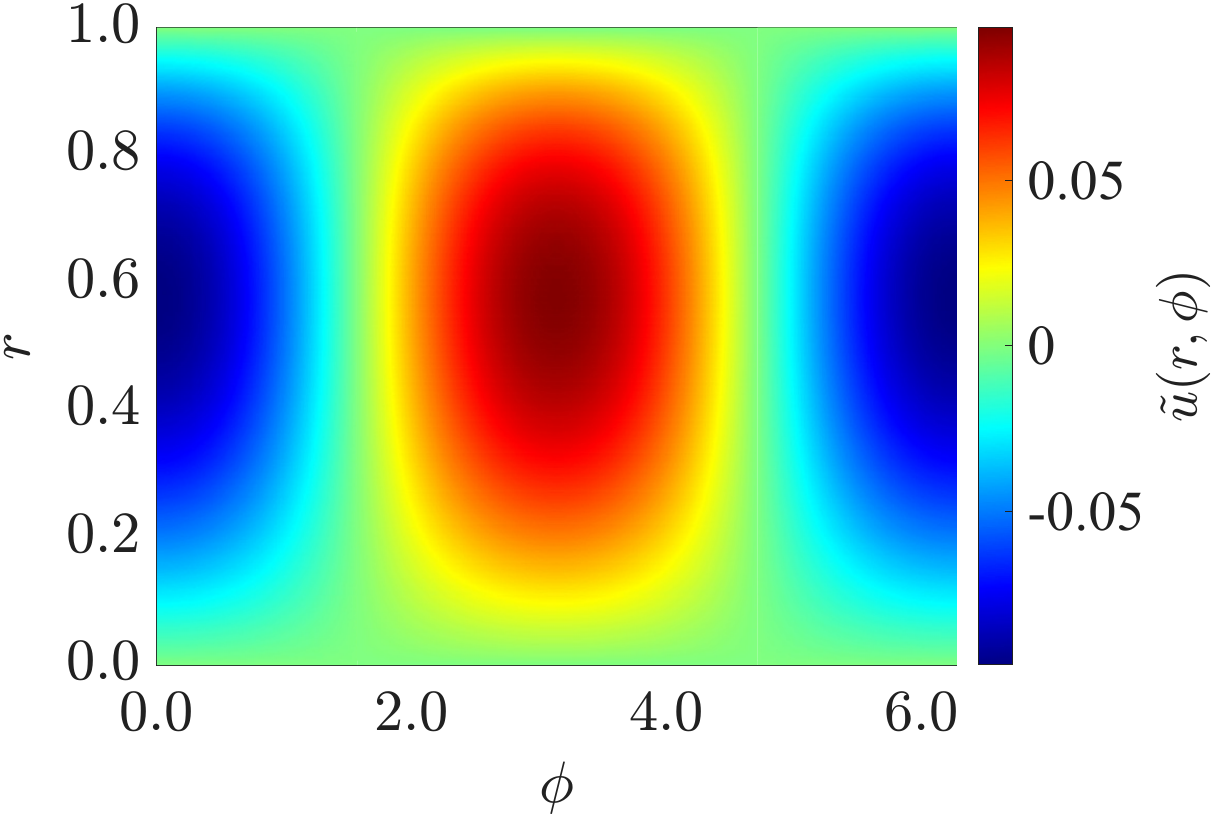}
        \caption{PFNN approximation $\tilde{u}$.}
        \label{subfig:poisson_1_u}
    \end{subfigure}
    \begin{subfigure}[t]{0.37\textwidth}
        \centering
        \includegraphics[width=\textwidth]{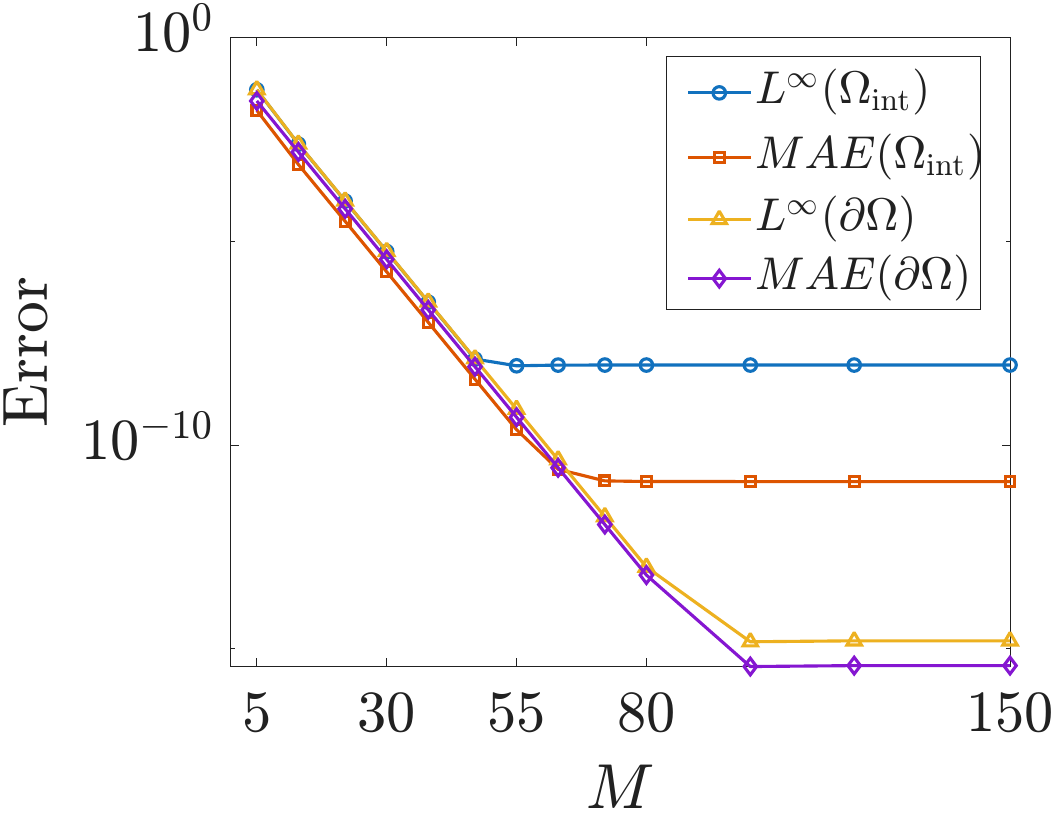}
        \caption{$L^{\infty}$ and MAE of the PFNN approximation in the interior $(\Omega_{\text{int}})$ and on the boundary ($\partial \Omega$) as a function of the hidden layers (log-scale).}
        \label{subfig:poisson_1_error}
        \end{subfigure}
    
    \caption{Results of the PFNN for the solution of the linear Poisson PDE example in section \eqref{inv-ex-poi-sec}.}
    \label{fig:poisson_1}
\end{figure}

\subsection{Helmholtz PDE}\label{appendix-helm}
For the PDE (\ref{helmholtz-pde}), the fundamental solution is given by $\Phi(x,y) = -\frac{1}{2\pi} K_0\left( \sqrt{\lambda} |x - y| \right)$. Hence, for $x = (x_1, x_2)$, $y = (y_1, y_2)$, we have that:
\begin{equation}
\frac{\partial \Phi(x,y)}{\partial y_1} = -\frac{\sqrt{\lambda} (y_1 - x_1)}{2\pi |x - y|} K_0' \left( \sqrt{\lambda} |x - y| \right)
\end{equation}
and similarly for the derivative with respect to $y_2$. Therefore, we obtain:
\begin{equation}
\frac{\partial \Phi(x,y)}{\partial n_y} = -\frac{\sqrt{\lambda} K_0' \left( \sqrt{\lambda} |x - y| \right)}{2\pi |x - y|} \big( y_1 (y_1 - x_1) + y_2 (y_2 - x_2) \big).
\end{equation}
From the properties of the Bessel function, we know that $K_0'(x) = -K_1(x)$. Furthermore, we know that the Bessel function exhibits a singularity at $x=y$. We consider this separately. Firstly, for $x,y \in \partial \Omega, x\neq y$ we have: 
\begin{equation}
\frac{\partial \Phi(x,y)}{\partial n_y} = \frac{\sqrt{\lambda} K_1 \left( \sqrt{\lambda} |x - y| \right)}{2\pi |x - y|^2} \left( y_1 (y_1 - x_1) + y_2 (y_2 - x_2) \right) = \frac{\sqrt{\lambda} K_1 \left( \sqrt{\lambda} |x - y| \right)}{4\pi |x - y|}.
\end{equation}
In this case, as $x \to y$, for $x,y \in \partial \Omega$ the Bessel function exhibits a singularity and we therefore have a weakly singular kernel in the integral operator for the BIE (in contrast to the Poisson PDE, as seen above). As this  affects the construction of the Fredholm NN, we examine the behavior of the kernel as $x \to y$. Specifically, we have that $K_0(z) \sim - \ln(z)$ (recall \cite{kropinski2011fast}). Then:
\begin{equation}
\Phi(x,y) \sim \frac{1}{2\pi} \ln \left( \sqrt{\lambda} |x - y| \right),
\end{equation}
and so:
\begin{equation}
\frac{\partial \Phi(x,y)}{\partial n_y} = \frac{y_1 (y_1 - x_1) + y_2 (y_2 - x_2)}{2\pi |x - y|^2} = \frac{1}{4\pi}.
\end{equation}
We therefore have a complete approach to numerically construct the kernel for the BIE, which takes into account the underlying mathematical theory of the fundamental solution. 
For our illustrations for this case, we consider the following  linear PDE:
\begin{equation}\label{ex-helmholtz}
\Delta u(x)- \lambda u(x) = -x_1^3 + 2x_2^2 + 6x_1 -4,
\end{equation}
with $\lambda = 1$ on the unit disc, $\Omega = \{x_1, x_2 \in \mathbb{R}: x_1^2 + x_2^2 \leq 1 \}$ with boundary condition $f(x):= x_1^3 + 2x_1^2 -2$ for $(x_1,x_2) \in  \partial \Omega$.
For the above PDE, the analytical solution exists and reads:
 \begin{equation*}
 u(x) = x_1^3 - 2x_2^2.
  \end{equation*}
We consider a $(r,\phi)$ grid of size $100\times 1000$. For the numerical integration of $\int_{\Omega} \lambda \delta \Phi(x, y) dy$ as well as the additional term $\int_{\Omega} \Phi(x, y) \psi(y) dy$ the adaptive quadrature scheme as detailed above is used in order to handle the kernel singularities with high accuracy. The resulting solution and errors are shown in Fig. \ref{fig:semi-linear}.
\begin{figure}[ht]
    \centering
        \begin{subfigure}[t]{0.45\textwidth}
        \centering
        \includegraphics[width=\textwidth ]{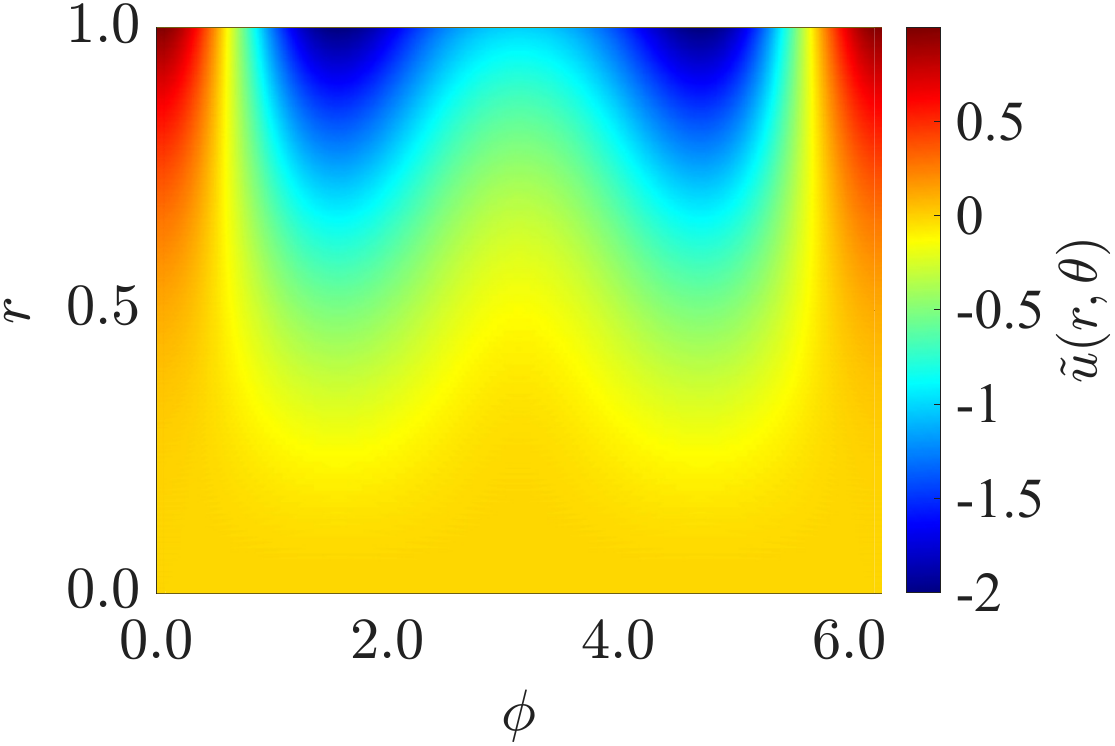}
        \caption{PFNN approximation $\tilde{u}$.}
    \end{subfigure}
    \begin{subfigure}[t]{0.40\textwidth}
        \centering
        \includegraphics[ width=\textwidth]{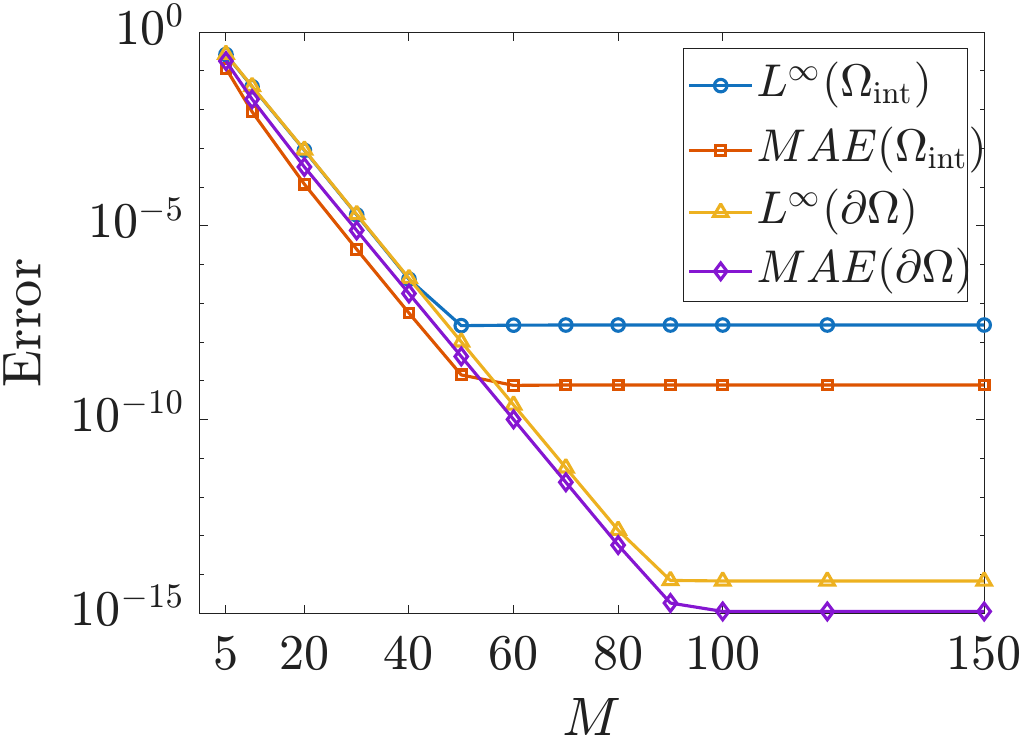}
        \caption{$L^{\infty}$ and MAE of the PFNN approximation in the interior of the domain, denoted $\Omega_{\text{int}}$, and on $\partial \Omega$ as a function of the hidden layers (log-scale).}\label{subfig:helmholtz-error}
    \end{subfigure}
\caption{Results of the PFNN for the Helmholtz PDE example in section \ref{inv-ex-helm-sec}.}
\label{fig:semi-linear}
\end{figure}

\subsection{Semi-linear elliptic PDE} 
\subsubsection{Picard iterations for the forward problem}

For solutions to the PDE \eqref{nl-pde}, fixed point schemes are often employed, which linearize the PDE at each step of the iteration. In line with this approach, we consider the monotone iteration scheme (see e.g., \cite{deng1996boundary}), as below: 

\begin{enumerate}
    \item Choose a $\lambda > 0$ "sufficiently large" (see below)
    \item Take an initial guess $u_0(x)$
    \item Solve, for $n = 0, 1, 2, \ldots$ the PDE:
    \begin{equation} 
    \begin{cases}
    \Delta u_{n+1}(x) - \lambda u_{n+1}(x) = -\lambda u_n(x) + \psi(x, u_n(x)), & x \in \Omega, \label{iteration}\\
    u_{n+1}(x) = f(x), & x \in \partial \Omega.
    \end{cases}
    \end{equation}
\end{enumerate}
    
Regarding the appropriate choice for $\lambda$, the value must be large enough to ensure the mapping:
\begin{eqnarray}\label{mapping}
T(u) = (\Delta - \lambda)^{-1}\big( - \lambda u + \psi(x,u) \big),
\end{eqnarray}
is a contraction, with the resulting iteration given by $u_{n+1} = T(u_n)$.
By Banach's fixed point theorem, this ensures that the scheme converges to the solution of the PDE (see \cite{deng1996boundary, sakakihara1987iterative} for such conditions and further details, and \cite{brebbia2016boundary} for other iterative schemes).\par This iterative scheme allows us to apply the proposed framework at each step of the scheme, whereby the source term will change at each iteration, in accordance to \eqref{iteration}.
Given that the source terms will depend on the solution $u_n(x)$, the numerical approximation of the integrals $\int_{\Omega} \Phi(x,y)\psi(y)dy$ must be considered. Here, we select a straightforward quadrature rule, across the points at which $u_n(x)$ has been calculated in order to avoid any interpolation-like approximations. For clarity, a full description of the fixed point algorithm solved via the FNN is given in Algorithm \ref{alg:nl-algo} and a schematic representation is given in Fig. \ref{fig:nl-algo}.

\begin{algorithm}[hbt!]
\caption{ Picard iteration for the PFNN solution to semi-linear PDEs}
\label{alg:nl-algo}
\small
\begin{algorithmic}
\State {Set number of iterations $N'$.}
\State {Discretize the domain $\Omega$ with interior points $\mathcal{X} = \{x_1, \ldots, x_{N}\}$ and the boundary $\partial \Omega$ with points $\mathcal{Y} = \{x^*_1, \ldots, x^*_N\}$, by the orthogonal projection $x_i^* := \text{argmin}_{y \in \partial \Omega} \|x_i - y\|$.}
\State {Select parameter $\lambda$ and initialize $u_0(x)$ (e.g., $u_0(x) = g(x)$ for all $x \in \mathcal{X}$).}
\Ensure The approximated fixed-point solution $u^*(x)$
\While{$n \leq N'$}
    \State {Step 1. Construct the PDE (\ref{iteration}) with source term $\psi_n(x) := -\lambda u_n(x) + \psi(x,u_n(x))$.}
    \State {Step 2. Approximate the particular solution via generalized quadrature:
    \begin{equation}
        \int_{\Omega} \Phi(x, y) \psi_n(y) dy \approx \sum_{j=1}^{M_\Omega} \Phi(x, x_j) \psi_n(x_j) \Delta y_j.
    \end{equation}
    }
    \State {Step 3. Solve the BIE (\ref{bie-sl}) via the FNN to find the density $\beta_n$ on $\mathcal{Y}$.}
    \State {Step 4. Solve the PDE using the PFNN (incorporating the projection $x_i^*$) to estimate $u_n(x_i)$ for all $i \in \{1, \ldots, M_\Omega\}$.}
    \State{Step 5. Set $n \rightarrow n+1$.}
\EndWhile \\
\Return {$u_{N'}(x)\approx u^*(x).$}
\end{algorithmic}
\end{algorithm}

\begin{figure}[ht]
    \centering
    \includegraphics[width=0.85\textwidth]{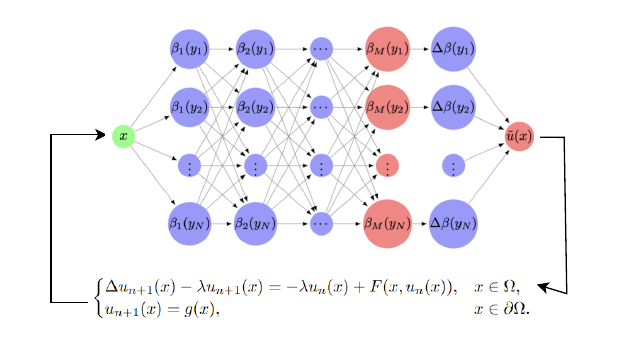}
    \caption{Schematic of the iterative Potential Fredholm NN approach to solve semi-linear elliptic PDEs. } \label{fig:nl-algo}
\end{figure}

We now turn to the Picard iterative approach to solve semi-linear problems, as described above. Firstly, we can also prove the corresponding error bound for the forward problem. 

\begin{lemma}
    Consider the non-linear PDE (\ref{nl-pde}), where $F(x,u)$ is $C^1$ with respect to $(x,u) \in \Omega \times \mathbb{R}$, and $\lambda >0$ such that the mapping $T(u)$, as given in \eqref{mapping}, is a $q$-contraction, with $q<1$.  Furthermore, let the approximation $\tilde{u}_n(x)$ be given by the Recurrent PFNN construction, as described in Algorithm \ref{alg:nl-algo}, and with known $u_0(x)$. Then, the following error bounds hold:
    \begin{flalign}\label{error-recurrent}
    \begin{cases}
        \left\|\tilde{u}_n - u \right\|_{\Omega} \leq  \epsilon^{PFNN}_{\Omega} + q^n\left\| u_0 - u  \right\|_{\Omega}, \\
        \left\|\tilde{u}_n - f \right\|_{\partial \Omega} \leq  \epsilon^{PFNN}_{\partial \Omega} + q^n\left\| u_0 - u  \right\|_{\partial \Omega},
    \end{cases}
    \end{flalign}
 where $\epsilon_{\Omega}^{PFNN}, \epsilon_{\partial \Omega}^{PFNN}$ denote the errors associated with the PFNN construction in the domain and on the boundary, as given by (\ref{error-domain}) and (\ref{error-boundary}) respectively.   
\end{lemma}
\begin{proof}
Notice that we can decompose the error into a component arising from the fixed point iterations and another due to the neural network approximation by writing $\left\|\tilde{u}_n - u \right\|_{\Omega} \leq \left\|\tilde{u}_n - u_n \right\|_{\Omega} + \left\|u_n - u \right\|_{\Omega}$. Substituting the error bound (\ref{error-domain}) and using the known error bound for the fixed point iterations results in (\ref{error-recurrent}). The error on the boundary follows in exactly the same fashion.
\end{proof}

For our illustrations, we consider the following 2D PDE:

\begin{eqnarray}\label{semi-linear-pde-example}
\begin{cases}
\Delta u(x) - e^{u} =  - e^{1-x_1^2 -x_2^2} - 4, \quad x \in \Omega \\
u(x) = 0, \quad x \in \partial \Omega,
\end{cases}
    \label{ex_nl}
\end{eqnarray}
which is in the form of the well-known and studied Liouville–Bratu–Gelfand  equation (see e.g., (\cite{gel1963some,wan2004thermo,jacobsen2002liouville} for examples and further analysis).
The above form is selected such that the analytical solution exists given by: 
\begin{equation}
u(x_1,x_2) = 1-(x_1^2 + x_2^2).
\end{equation}
We construct the Recurrent PFNN, where in each iteration a PFNN with $100$ hidden layers is used, and perform 12 iterations of Algorithm \ref{alg:nl-algo}, with $\lambda =1.0$. In Fig. \ref{fig:example_nl} we display the results when using a $(r,\phi)$ grid of size $120\times 120$. 


\begin{figure}[ht]
    \centering
    \begin{subfigure}[t]{0.45\textwidth}
        \centering
        \includegraphics[width=\textwidth]{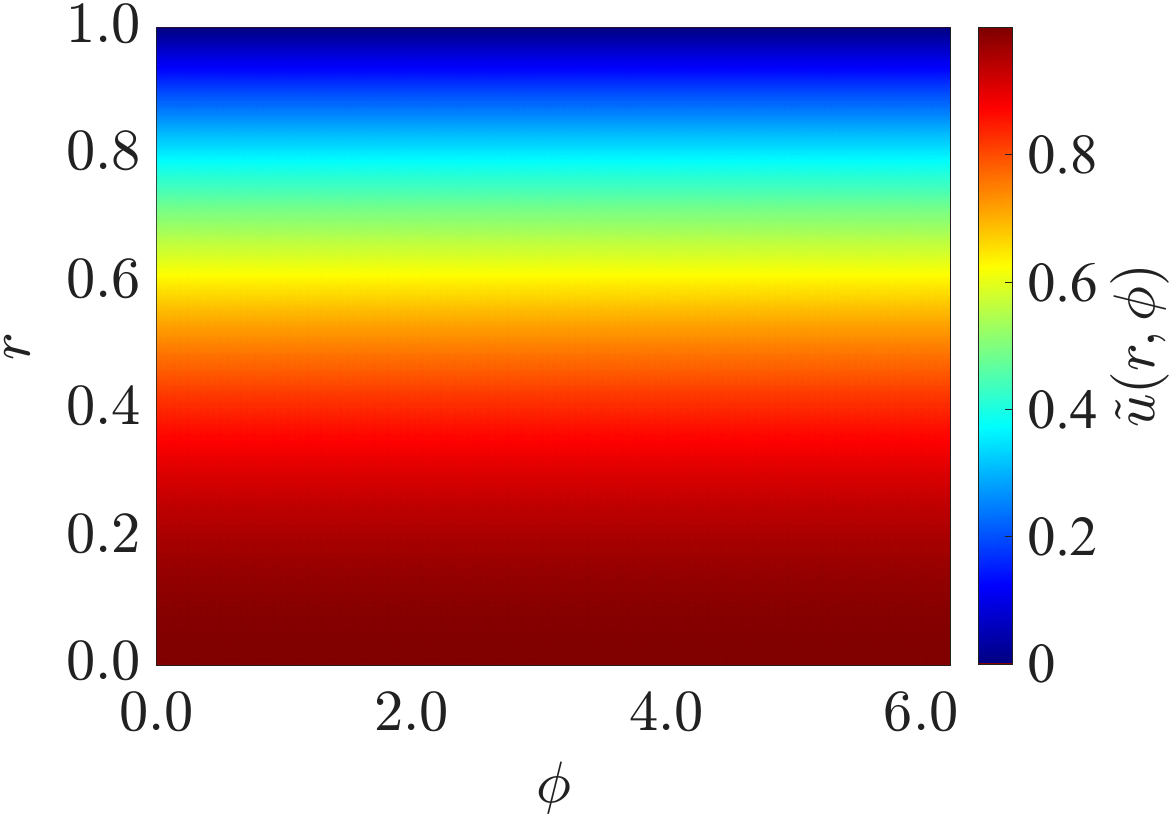}
        \caption{Recurrent PFNN solution $\tilde{u}$.}
        \label{subfig:non_linear_50_u}
    \end{subfigure}
    \begin{subfigure}[t]{0.43\textwidth}
        \centering
        \includegraphics[width=\textwidth]{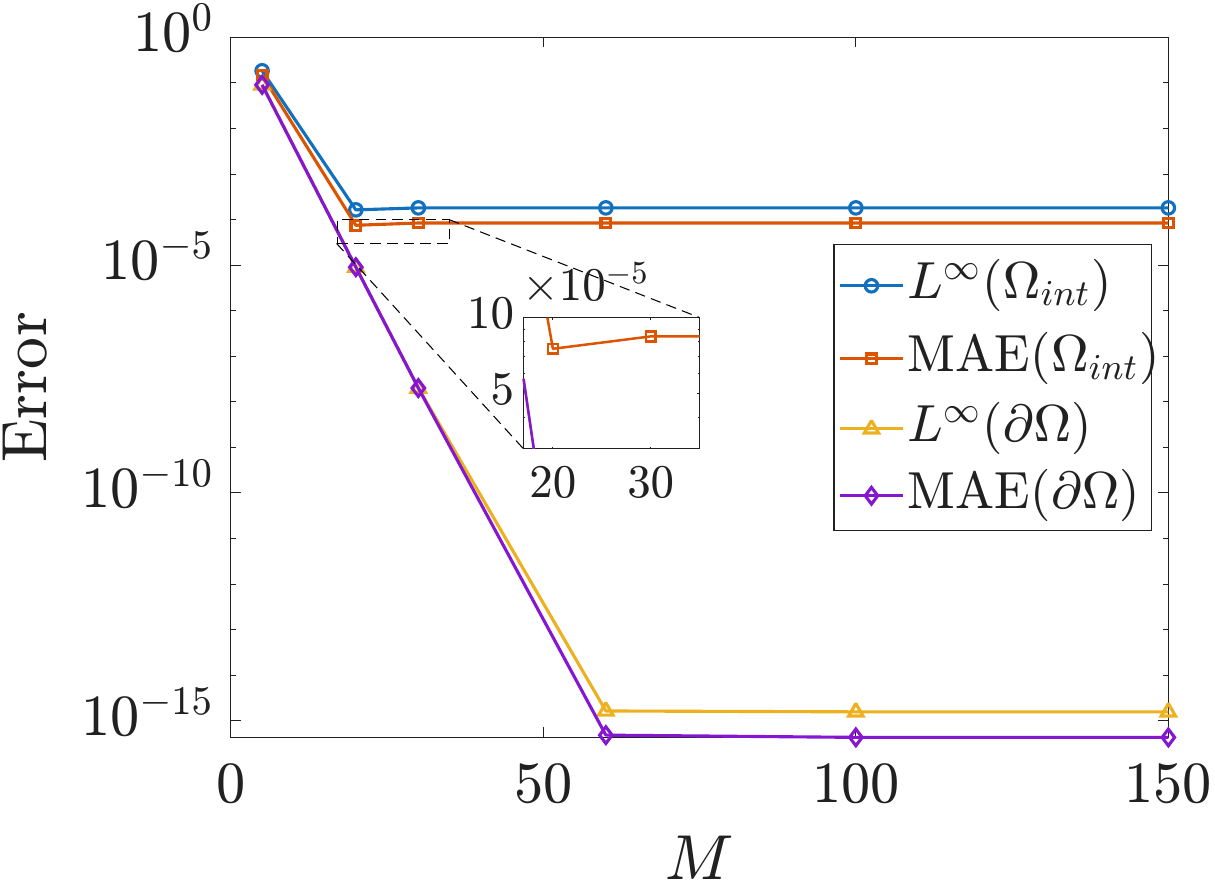}
        \caption{$L^{\infty}$ and MAE of the iterative PFNN approximation in $\Omega_{\text{int}}$ and on $\partial \Omega$ as a function of the hidden layers (log-scale). The inset plot is explained in section \ref{explainability}.}
        \label{subfig:non_linear_50_error}
    \end{subfigure}
    \caption{Results of the iterative Potential Fredholm Neural Network construction for the semi-linear PDE example in section \eqref{inv-ex-sl-sec}.}
    \label{fig:example_nl}
\end{figure}
The above examples show that the proposed framework is able to achieve highly accurate performance both in the interior and on the boundary of the domain. Of course, we note that alternative numerical integration schemes can also be considered and implemented within the PFNN construction.

\subsubsection{An interesting note on explainability using PFNNs}\label{explainability}
The benefits of the proposed approach lie primarily in its potential to better understand and, therefore, explain the source of errors 
as for example, those in Fig. \ref{subfig:poisson_1_error} and \ref{subfig:helmholtz-error}. In these two cases, we see that the errors decrease monotonically (though not necessarily strictly) with the number of hidden layers. This result is explainable given the results of Theorem \ref{error-bounds-thrm}, where we have shown that the error bound depends on the approximation $\| \beta - \beta_M \|$, as $M\to \infty$ (note that such a behaviour is not observed in PINNs in which case, the performance varies significantly when increasing the number of hidden layers).

However, it is also shown that the error bound \eqref{error-domain} depends on $\|\beta_M \|_{\partial\Omega}$, i.e., the Fredholm NN approximation of the boundary function itself. Hence, since this value varies for different $M$, it is possible to observe larger errors as $M$ increases, depending on the comparison between the terms depending on each of the two factors mentioned above. 
\par Taking PDE \eqref{semi-linear-pde-example} as an example, these error bounds allow us to interpret and explain the errors seen in Fig. \ref{subfig:non_linear_50_error}, where a very slight increase can be observed in the interior error metrics when considering the Recurrent PFNN approximations with $M=20$ and $M= 30$. Specifically, let us consider the first iteration in the implementation described for \eqref{ex_nl}. We have that $\|\beta_{M=20} \|_{\partial \Omega} = 1.0442995$, $\|\beta_{M=30}\|_{\partial \Omega} = 1.0443096$. This difference is translated into an increased error between the iterative PFNN approximation in each case, with $\| \tilde{u}_1 - u\|_{ \Omega} \approx 4.0794324$E$-01$ and $ 4.0794929$E$-01$ for $M=20$ and $M=30$, respectively. This difference in the error is then propagated throughout the remaining iterations of the iterative PFNN, as expected and in accordance with \eqref{error-recurrent} (since the second term in the error bound is constant for both schemes). For larger $M$ the error then remains constant, as the limit of both $\| \beta - \beta_M \|_{\partial \Omega}$ and $\|\beta_M \|_{\partial \Omega}$ has been reached. 
\par These results indicate that, considering alternative mathematical formulations for the solution of PDEs that can be connected to a DNN model architecture, we can construct DNN models that adhere to explainable and interpretable behavior as hyperparameters, such as the number of hidden layers vary. We note that similar results hold for the number of nodes in the layers of the neural network (recall that in the iterative PFNN construction, the nodes correspond to the discretization of the integral in the BIE).

\subsection{3D Poisson PDE}\label{appendix-poisson-3d}

The construction of the BIE \eqref{BIE} extends naturally to three 
dimensions, with two notable structural differences from the 2D case 
that we make explicit here. Firstly, the 3D fundamental solution is:
\begin{equation}
    \Phi(x,y) = \frac{-1}{4\pi |x-y|},
\end{equation}
and secondly, the simplification that holds for $x^*, y$ 
on the unit disk's boundary in 2D does extend to 3D, nor to any non-spherical boundary. The kernel must therefore be retained explicitly in the BIE. Computing the kernel in 3D, we have 
$\partial \Phi(x,y)/\partial n_y = n_y \cdot \nabla_y \Phi(x,y)$ with
\begin{equation}
    \nabla_y \Phi(x,y) = \frac{y - x}{4\pi |x-y|^3}, 
    \qquad
    \frac{\partial \Phi(x,y)}{\partial n_y} = \frac{(y - x) \cdot n_y}{4\pi |x-y|^3}.
\end{equation}
For an arbitrary smooth bounded domain $\Omega \subset \mathbb{R}^3$ 
with outward unit normal $n_y$ at $y \in \partial\Omega$, this expression 
is a non-constant function of $y$. It is integrable but 
weakly singular at $y = x$.

For the BIE, taking $x, y \in \partial\Omega$ and applying the standard 
double-layer jump relation gives the Fredholm equation of the second kind 
for the boundary density $\beta$:
\begin{equation}\label{poisson-3d-bie}
    \beta(x) = 2\,\Bigg( f(x) - \int_{\Omega} \Phi(x,y)\,\psi(y)\,dy \Bigg) 
              - 2\int_{\partial\Omega} \beta(y)\,\frac{(y-x)\cdot n_y}{4\pi |x-y|^3}\,d\sigma_y, 
              \qquad x \in \partial\Omega.
\end{equation}
Unlike the 2D disk case, where the boundary kernel reduces to the 
constant $1/(4\pi)$ and the second integral simplifies to a multiple 
of the boundary integral of $\beta$, in 3D the kernel must be evaluated 
pointwise on $\partial\Omega$. The volume integral is given explicitly 
by
\begin{equation}\label{poisson-3d-vol}
    \int_{\Omega} \Phi(x,y)\,\psi(y)\,dy 
    = -\int_{\Omega} \frac{\psi(y)}{4\pi |x-y|}\,dy,
\end{equation}
which is again weakly singular at $y = x$ but integrable in 3D.

Two numerical adaptations are required relative to the 2D case. The 
first is the surface quadrature on $\partial\Omega$, where 
parameterize the boundary by spherical-style angles 
$(\theta, \varphi) \in [0,\pi] \times [0,2\pi)$ and substitute $\mu = \cos\theta$ to absorb the polar Jacobian factor and map the polar variable to $[-1,1]$. This is done in order to apply-Legendre quadrature. For example, for the spheroid boundary
$\Omega = \{x_1^2/a^2 + x_2^2/a^2 + x_3^2/c^2 < 1\}$, 
parameterized by $y(\theta,\varphi) = (a\sin\theta\cos\varphi,\, 
a\sin\theta\sin\varphi,\, c\cos\theta)$, the surface element after the 
substitution is:
\begin{equation}\label{eq:spheroid-area-element}
    d\sigma = J(\mu)\,d\mu\,d\varphi, 
    \qquad
    J(\mu) := a\sqrt{c^2(1-\mu^2) + a^2 \mu^2}.
\end{equation}
We can now discretize this double integral by a tensor-product rule combining 
$N_\theta$-point Gauss-Legendre quadrature in $\mu$ with weights 
$\{a_{\mu,i}\}$ (spectrally accurate on smooth non-periodic integrands), 
and $N_\varphi$ uniform nodes in $\varphi$ with weight 
$w_\varphi = 2\pi/N_\varphi$ (spectrally accurate on smooth periodic 
integrands, allowing us to use significantly less points for the numerical integration). The discretized surface integral is then
\begin{equation}\label{eq:surface-quadrature}
    \int_{\partial\Omega} g(y)\,d\sigma_y 
    \approx 
    \sum_{i=1}^{N_\theta}\sum_{j=1}^{N_\varphi} g(y_{ij})\,w_{ij},
    \qquad
    w_{ij} := a_{\mu,i}\,w_\varphi\,J(\mu_i),
\end{equation}
where $y_{ij} = y(\mu_i, \varphi_j)$ are the boundary quadrature nodes 
and $w_{ij}$ are the combined weights. The volume integral 
\eqref{poisson-3d-vol} is discretized analogously by a tensor-product 
rule of $N_r$ Gauss-Legendre nodes in the radial direction, $N_\mu$ 
Gauss-Legendre nodes in $\mu$, and $N_{\varphi,V}$ trapezoidal nodes in 
$\varphi$, with the singularity at $y = x$ handled by a small distance 
cap (adequate because evaluation points and quadrature nodes are 
generically distinct).

Secondly, given the weakly singular integral kernel of the discrete 
double-layer matrix, the discretization must be controlled to ensure the discrete analogue of identity \eqref{important-identity} is still satisfied. This is of paramount importance, as the Fredholm Neural Network for the BIE assumes non-expansiveness of the integral operator. To ensure this condition, we overwrite the diagonal of $K$ to enforce the identity exactly, i.e.,
\begin{equation}\label{eq:row-sum-correction-3d}
    K_{ii} = \tfrac{1}{2} - \sum_{j \neq i} K_{ij},
\end{equation}
where: 
\begin{equation}
    K_{ij} = w_{ij}\,\frac{(y_j - y_i) \cdot n_j}{4\pi |y_i - y_j|^3}, 
\end{equation}
is the kernel discretization.

For our example, we consider the 3D Poisson PDE
\begin{equation}
    \Delta u(x) = \psi(x), \qquad x \in \Omega,
\end{equation}
on the spheroid 
\begin{equation}
   \Omega = \left\{x \in \mathbb{R}^3 : \frac{x_1^2}{a^2} + 
\frac{x_2^2}{a^2} + \frac{x_3^2}{c^2} < 1\right\},
\end{equation}
with $a = 1$, $c = 1.5$. We consider an
analytic reference solution on this domain, and manufacture the corresponding source term $\psi := \Delta u$. Specifically, we take
\begin{equation}\label{eq:3d-true-soln}
    u(x) = e^{x_1}\sin(x_2)\cos(x_3) + x_1 x_2 x_3,
\end{equation}
for which direct calculation gives
\begin{equation}\label{eq:3d-source}
    \psi(x) = \Delta u(x) = -e^{x_1}\sin(x_2)\cos(x_3),
\end{equation}
and the boundary condition is $f(x) = u(x)|_{\partial \Omega}$.

For the BIE, we use a boundary discretization with $N_\theta = N_\varphi = 50$, and volume discretization  $(N_r, N_\mu, N^{V}_{\varphi}) = (50, 40, 60)$ (note the difference between the two $\phi$-grids: the first use forthe surface integral discretization with $N_{\phi}$ nodes, the second for the volume integral with $N^V_{\phi}$ nodes). We use Krasnosel'skii-Mann  constant $\kappa = 0.5$. Convergence and error metrics with respect to the number of layers $M$ in the PFNN is shown in Fig.~\ref{fig:poisson_3d_forward}. Notice again that machine precision is achieved on the boundary, as expected by the PFNN construction.

\begin{figure}[ht]
    \centering
        \begin{subfigure}[t]{0.45\textwidth}
        \centering
        \includegraphics[width=\textwidth ]{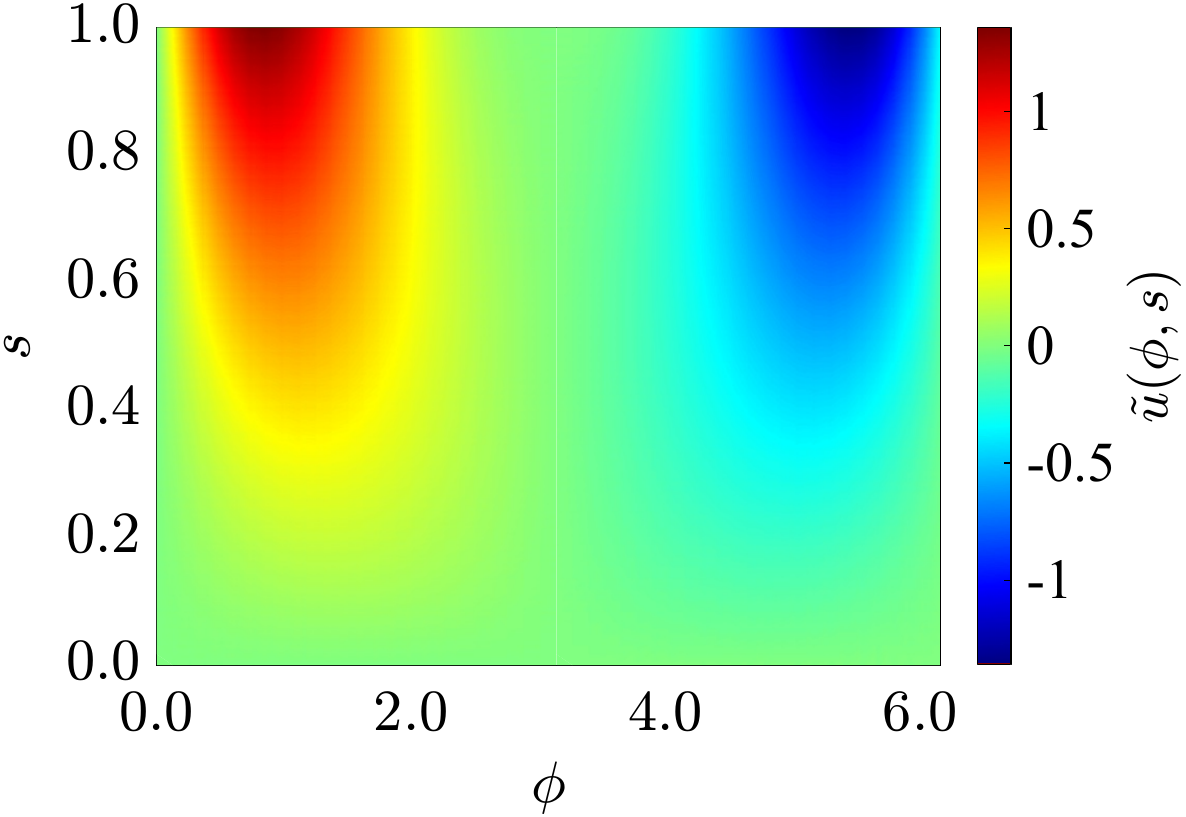}
        \caption{PFNN approximation $\tilde{u}$.}
    \end{subfigure}
    \begin{subfigure}[t]{0.40\textwidth}
        \centering
        \includegraphics[ width=\textwidth]{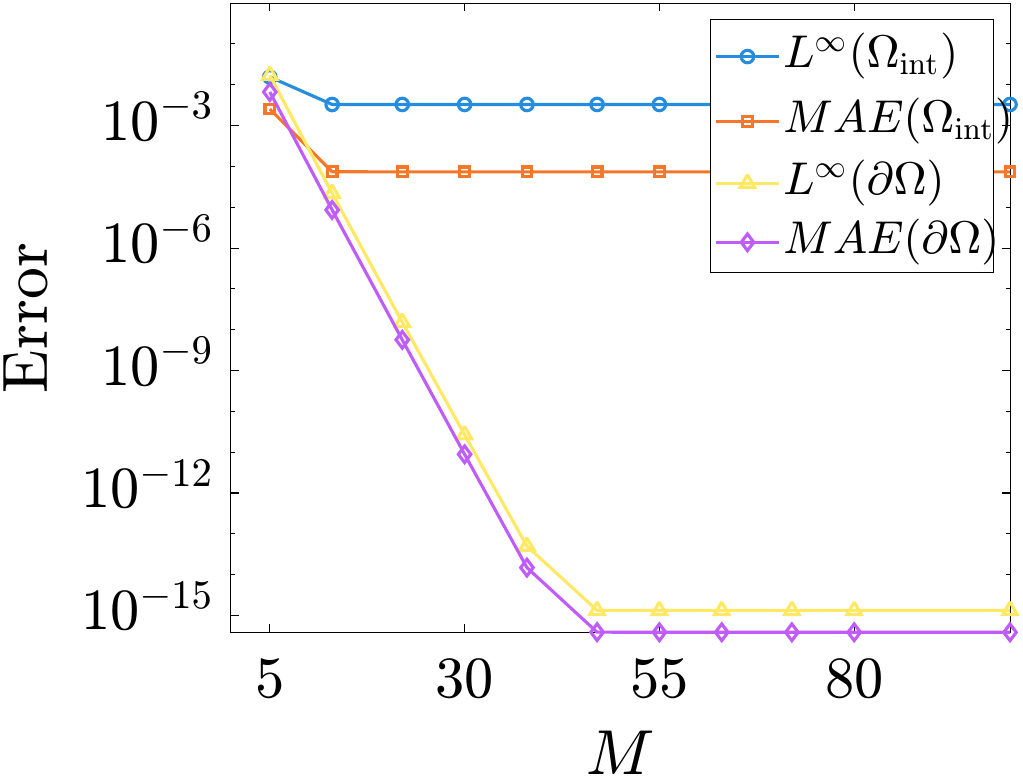}
        \caption{$L^{\infty}$ and MAE of the PFNN approximation in the interior of the domain, denoted $\Omega_{\text{int}}$, and on $\partial \Omega$ as a function of the hidden layers (log-scale).}\label{subfig:helmholtz-error}
    \end{subfigure}
\caption{Results of the PFNN for the 3D Poisson PDE example in section \ref{appendix-poisson-3d}.}
\label{fig:poisson_3d_forward}
\end{figure}


\bibliographystyle{plain}
\bibliography{arXiv_Fredholm_Neural_Networks}

\begin{thebibliography}{10}

\bibitem{akrivis2024runge}
Georgios Akrivis, Charalambos~G Makridakis, and Costas Smaragdakis.
\newblock Runge-kutta physics informed neural networks: Formulation and analysis.
\newblock {\em arXiv preprint arXiv:2412.20575}, 2024.

\bibitem{bolager2024sampling}
Erik~L Bolager, Iryna Burak, Chinmay Datar, Qing Sun, and Felix Dietrich.
\newblock Sampling weights of deep neural networks.
\newblock {\em Advances in Neural Information Processing Systems}, 36, 2024.

\bibitem{brebbia2016boundary}
Carlos~Alberto Brebbia and Stephen Walker.
\newblock {\em Boundary element techniques in engineering}.
\newblock Elsevier, 2016.

\bibitem{calabro2021extreme}
Francesco Calabr{\`o}, Gianluca Fabiani, and Constantinos Siettos.
\newblock Extreme learning machine collocation for the numerical solution of elliptic {PDEs} with sharp gradients.
\newblock {\em Computer Methods in Applied Mechanics and Engineering}, 387:114188, 2021.

\bibitem{casas2014optimal}
Eduardo Casas and Karl Kunisch.
\newblock Optimal control of semilinear elliptic equations in measure spaces.
\newblock {\em SIAM Journal on Control and Optimization}, 52(1):339--364, 2014.

\bibitem{chantada2024exact}
Augusto~T Chantada, Pavlos Protopapas, Luca~Gomez Bachar, Susana~J Landau, and Claudia~G Sc{\'o}ccola.
\newblock Exact and approximate error bounds for physics-informed neural networks.
\newblock {\em arXiv preprint arXiv:2411.13848}, 2024.

\bibitem{chen2021solving}
Yifan Chen, Bamdad Hosseini, Houman Owhadi, and Andrew~M Stuart.
\newblock Solving and learning nonlinear pdes with gaussian processes.
\newblock {\em Journal of Computational Physics}, 447:110668, 2021.

\bibitem{de2022physics}
Mario De~Florio, Enrico Schiassi, and Roberto Furfaro.
\newblock Physics-informed neural networks and functional interpolation for stiff chemical kinetics.
\newblock {\em Chaos: An Interdisciplinary Journal of Nonlinear Science}, 32(6), 2022.

\bibitem{de2022error}
Tim De~Ryck and Siddhartha Mishra.
\newblock Error analysis for physics-informed neural networks (pinns) approximating kolmogorov pdes.
\newblock {\em Advances in Computational Mathematics}, 48(6):79, 2022.

\bibitem{de2022generic}
Tim De~Ryck and Siddhartha Mishra.
\newblock Generic bounds on the approximation error for physics-informed (and) operator learning.
\newblock {\em Advances in Neural Information Processing Systems}, 35:10945--10958, 2022.

\bibitem{deng1996boundary}
Yuanhua Deng, Goong Chen, Wei-Ming Ni, and Jianxin Zhou.
\newblock Boundary element monotone iteration scheme for semilinear elliptic partial differential equations.
\newblock {\em Mathematics of computation}, 65(215):943--982, 1996.

\bibitem{doncevic2024recursively}
Danimir~T Doncevic, Alexander Mitsos, Yue Guo, Qianxiao Li, Felix Dietrich, Manuel Dahmen, and Ioannis~G Kevrekidis.
\newblock A recursively recurrent neural network (r2n2) architecture for learning iterative algorithms.
\newblock {\em SIAM Journal on Scientific Computing}, 46(2):A719--A743, 2024.

\bibitem{dong2021local}
Suchuan Dong and Zongwei Li.
\newblock Local extreme learning machines and domain decomposition for solving linear and nonlinear partial differential equations.
\newblock {\em Computer Methods in Applied Mechanics and Engineering}, 387:114129, 2021.

\bibitem{dong2021modified}
Suchuan Dong and Zongwei Li.
\newblock A modified batch intrinsic plasticity method for pre-training the random coefficients of extreme learning machines.
\newblock {\em Journal of Computational Physics}, 445:110585, 2021.

\bibitem{doumeche2023convergence}
Nathan Doum{\`e}che, G{\'e}rard Biau, and Claire Boyer.
\newblock Convergence and error analysis of pinns.
\newblock {\em arXiv preprint arXiv:2305.01240}, 2023.

\bibitem{effati2012neural}
Sohrab Effati and Reza Buzhabadi.
\newblock A neural network approach for solving fredholm integral equations of the second kind.
\newblock {\em Neural Computing and Applications}, 21:843--852, 2012.

\bibitem{evans2010partial}
Lawrence~C. Evans.
\newblock {\em Partial Differential Equations}.
\newblock American Mathematical Society, 2 edition, 2010.

\bibitem{fabiani2021numerical}
Gianluca Fabiani, Francesco Calabr{\`o}, Lucia Russo, and Constantinos Siettos.
\newblock Numerical solution and bifurcation analysis of nonlinear partial differential equations with extreme learning machines.
\newblock {\em Journal of Scientific Computing}, 89:44, 2021.

\bibitem{fabiani2024task}
Gianluca Fabiani, Nikolaos Evangelou, Tianqi Cui, Juan~M Bello-Rivas, Cristina~P Martin-Linares, Constantinos Siettos, and Ioannis~G Kevrekidis.
\newblock Task-oriented machine learning surrogates for tipping points of agent-based models.
\newblock {\em Nature communications}, 15(1):4117, 2024.

\bibitem{fabiani2023parsimonious}
Gianluca Fabiani, Evangelos Galaris, Lucia Russo, and Constantinos Siettos.
\newblock Parsimonious physics-informed random projection neural networks for initial value problems of odes and index-1 daes.
\newblock {\em Chaos: An Interdisciplinary Journal of Nonlinear Science}, 33(4), 2023.

\bibitem{fabiani2024randonet}
Gianluca Fabiani, Ioannis~G Kevrekidis, Constantinos Siettos, and Athanasios~N Yannacopoulos.
\newblock Randonet: Shallow-networks with random projections for learning linear and nonlinear operators.
\newblock {\em arXiv preprint arXiv:2406.05470}, 2024.

\bibitem{folland2020introduction}
Gerald~B Folland.
\newblock {\em Introduction to partial differential equations}.
\newblock Princeton university press, 2020.

\bibitem{frey2022deep}
R{\"u}diger Frey and Verena K{\"o}ck.
\newblock Deep neural network algorithms for parabolic pides and applications in insurance and finance.
\newblock {\em Computation}, 10(11):201, 2022.

\bibitem{galaris2022numerical}
Evangelos Galaris, Gianluca Fabiani, Ioannis Gallos, Ioannis Kevrekidis, and Constantinos Siettos.
\newblock Numerical bifurcation analysis of pdes from lattice boltzmann model simulations: a parsimonious machine learning approach.
\newblock {\em Journal of Scientific Computing}, 92(2):1--30, 2022.

\bibitem{gazoulis2023stability}
Dimitrios Gazoulis, Ioannis Gkanis, and Charalambos~G Makridakis.
\newblock On the stability and convergence of physics informed neural networks.
\newblock {\em arXiv preprint arXiv:2308.05423}, 2023.

\bibitem{gel1963some}
IM~Gel et~al.
\newblock Some problems in the theory of quasilinear equations.
\newblock {\em American Mathematical Society Translations: Series 2}, 29:295--381, 1963.

\bibitem{georgiou2024fredholm}
Kyriakos Georgiou, Constantinos Siettos, and Athanasios~N Yannacopoulos.
\newblock Fredholm neural networks.
\newblock {\em SIAM Journal on Scientific Computing (to appear)}, 2025.

\bibitem{georgiou2024deep}
Kyriakos Georgiou and Athanasios~N Yannacopoulos.
\newblock Deep neural networks for probability of default modelling.
\newblock {\em Journal of Industrial and Management Optimization}, pages 0--0, 2024.

\bibitem{gilbarg1977elliptic}
David Gilbarg, Neil~S Trudinger, David Gilbarg, and NS~Trudinger.
\newblock {\em Elliptic partial differential equations of second order}, volume 224.
\newblock Springer, 1977.

\bibitem{grekas2025deep}
Georgios Grekas and Charalambos~G Makridakis.
\newblock Deep ritz-finite element methods: Neural network methods trained with finite elements.
\newblock {\em Computer Methods in Applied Mechanics and Engineering}, 437:117798, 2025.

\bibitem{guan2022solving}
Yu~Guan, Tingting Fang, Diankun Zhang, and Congming Jin.
\newblock Solving fredholm integral equations using deep learning.
\newblock {\em International Journal of Applied and Computational Mathematics}, 8(2):87, 2022.

\bibitem{guo2022personalized}
Yue Guo, Felix Dietrich, Tom Bertalan, Danimir~T Doncevic, Manuel Dahmen, Ioannis~G Kevrekidis, and Qianxiao Li.
\newblock Personalized algorithm generation: A case study in learning ode integrators.
\newblock {\em SIAM Journal on Scientific Computing}, 44(4):A1911--A1933, 2022.

\bibitem{han2018solving}
Jiequn Han, Arnulf Jentzen, and E~Weinan.
\newblock Solving high-dimensional partial differential equations using deep learning.
\newblock {\em Proceedings of the National Academy of Sciences}, 115(34):8505--8510, 2018.

\bibitem{hillebrecht2022certified}
Birgit Hillebrecht and Benjamin Unger.
\newblock Certified machine learning: A posteriori error estimation for physics-informed neural networks.
\newblock In {\em 2022 International Joint Conference on Neural Networks (IJCNN)}, pages 1--8. IEEE, 2022.

\bibitem{hillebrecht2023rigorous}
Birgit Hillebrecht and Benjamin Unger.
\newblock Rigorous a posteriori error bounds for pde-defined pinns.
\newblock {\em IEEE Transactions on Neural Networks and Learning Systems}, 2023.

\bibitem{hsiao2008boundary}
George~C Hsiao and Wolfgang~L Wendland.
\newblock {\em Boundary integral equations}.
\newblock Springer, 2008.

\bibitem{jacob2024spikans}
Bruno Jacob, Amanda~A Howard, and Panos Stinis.
\newblock Spikans: Separable physics-informed kolmogorov-arnold networks.
\newblock {\em arXiv preprint arXiv:2411.06286}, 2024.

\bibitem{jacobsen2002liouville}
Jon Jacobsen and Klaus Schmitt.
\newblock The liouville--bratu--gelfand problem for radial operators.
\newblock {\em Journal of Differential Equations}, 184(1):283--298, 2002.

\bibitem{jafarian2013utilizing}
A~Jafarian and S~Measoomy Nia.
\newblock Utilizing feed-back neural network approach for solving linear fredholm integral equations system.
\newblock {\em Applied Mathematical Modelling}, 37(7):5027--5038, 2013.

\bibitem{ji2021stiff}
Weiqi Ji, Weilun Qiu, Zhiyu Shi, Shaowu Pan, and Sili Deng.
\newblock Stiff-pinn: Physics-informed neural network for stiff chemical kinetics.
\newblock {\em The Journal of Physical Chemistry A}, 125(36):8098--8106, 2021.

\bibitem{kahana2023geometry}
Adar Kahana, Enrui Zhang, Somdatta Goswami, George Karniadakis, Rishikesh Ranade, and Jay Pathak.
\newblock On the geometry transferability of the hybrid iterative numerical solver for differential equations.
\newblock {\em Computational Mechanics}, 72(3):471--484, 2023.

\bibitem{karniadakis2021physics}
George~Em Karniadakis, Ioannis~G Kevrekidis, Lu~Lu, Paris Perdikaris, Sifan Wang, and Liu Yang.
\newblock Physics-informed machine learning.
\newblock {\em Nature Reviews Physics}, 3(6):422--440, 2021.

\bibitem{kellogg2012foundations}
Oliver~Dimon Kellogg.
\newblock {\em Foundations of potential theory}, volume~31.
\newblock Springer Science \& Business Media, 2012.

\bibitem{kim2021stiff}
Suyong Kim, Weiqi Ji, Sili Deng, Yingbo Ma, and Christopher Rackauckas.
\newblock Stiff neural ordinary differential equations.
\newblock {\em Chaos: An Interdisciplinary Journal of Nonlinear Science}, 31(9), 2021.

\bibitem{koenig2024kan}
Benjamin~C Koenig, Suyong Kim, and Sili Deng.
\newblock Kan-odes: Kolmogorov--arnold network ordinary differential equations for learning dynamical systems and hidden physics.
\newblock {\em Computer Methods in Applied Mechanics and Engineering}, 432:117397, 2024.

\bibitem{kopanivcakova2024deeponet}
Alena Kopani{\v{c}}{\'a}kov{\'a} and George~Em Karniadakis.
\newblock Deeponet based preconditioning strategies for solving parametric linear systems of equations.
\newblock {\em arXiv preprint arXiv:2401.02016}, 2024.

\bibitem{kovachki2023neural}
Nikola~B Kovachki, Zongyi Li, Burigede Liu, Kamyar Azizzadenesheli, Kaushik Bhattacharya, Andrew~M Stuart, and Anima Anandkumar.
\newblock Neural operator: Learning maps between function spaces with applications to pdes.
\newblock {\em J. Mach. Learn. Res.}, 24(89):1--97, 2023.

\bibitem{kress2014linear}
Rainer Kress.
\newblock {\em Linear Integral Equations}.
\newblock Springer, 3 edition, 2014.

\bibitem{krishnapriyan2021characterizing}
Aditi Krishnapriyan, Amir Gholami, Shandian Zhe, Robert Kirby, and Michael~W Mahoney.
\newblock Characterizing possible failure modes in physics-informed neural networks.
\newblock {\em Advances in Neural Information Processing Systems}, 34:26548--26560, 2021.

\bibitem{kropinski2011fast}
Mary Catherine~A Kropinski and Bryan~D Quaife.
\newblock Fast integral equation methods for the modified helmholtz equation.
\newblock {\em Journal of Computational Physics}, 230(2):425--434, 2011.

\bibitem{laghi2023physics}
Laura Laghi, Enrico Schiassi, Mario De~Florio, Roberto Furfaro, and Domiziano Mostacci.
\newblock Physics-informed neural networks for 1-d steady-state diffusion-advection-reaction equations.
\newblock {\em Nuclear Science and Engineering}, 197(9):2373--2403, 2023.

\bibitem{lee2020coarse}
Seungjoon Lee, Mahdi Kooshkbaghi, Konstantinos Spiliotis, Constantinos~I Siettos, and Ioannis~G. Kevrekidis.
\newblock Coarse-scale pdes from fine-scale observations via machine learning.
\newblock {\em Chaos: An Interdisciplinary Journal of Nonlinear Science}, 30(1):013141, 2020.

\bibitem{lee2023learning}
Seungjoon Lee, Yorgos~M Psarellis, Constantinos~I Siettos, and Ioannis~G Kevrekidis.
\newblock Learning black-and gray-box chemotactic pdes/closures from agent based monte carlo simulation data.
\newblock {\em Journal of Mathematical Biology}, 87(1):15, 2023.

\bibitem{li2020Fourier}
Zongyi Li, Nikola Kovachki, Kamyar Azizzadenesheli, Burigede Liu, Kaushik Bhattacharya, Andrew Stuart, and Anima Anandkumar.
\newblock Fourier neural operator for parametric partial differential equations.
\newblock {\em arXiv preprint arXiv:2010.08895}, 2020.

\bibitem{li2020multipole}
Zongyi Li, Nikola Kovachki, Kamyar Azizzadenesheli, Burigede Liu, Andrew Stuart, Kaushik Bhattacharya, and Anima Anandkumar.
\newblock Multipole graph neural operator for parametric partial differential equations.
\newblock {\em Advances in Neural Information Processing Systems}, 33:6755--6766, 2020.

\bibitem{lin2021binet}
Guochang Lin, Pipi Hu, Fukai Chen, Xiang Chen, Junqing Chen, Jun Wang, and Zuoqiang Shi.
\newblock Binet: learning to solve partial differential equations with boundary integral networks.
\newblock {\em arXiv preprint arXiv:2110.00352}, 2021.

\bibitem{liu2024kan}
Ziming Liu, Yixuan Wang, Sachin Vaidya, Fabian Ruehle, James Halverson, Marin Solja{\v{c}}i{\'c}, Thomas~Y Hou, and Max Tegmark.
\newblock Kan: Kolmogorov-arnold networks.
\newblock {\em arXiv preprint arXiv:2404.19756}, 2024.

\bibitem{lohner2023neural}
Rainald L{\"o}hner and Harbir Antil.
\newblock Neural network representation of time integrators.
\newblock {\em International Journal for Numerical Methods in Engineering}, 124(18):4192--4198, 2023.

\bibitem{lu2021learning}
Lu~Lu, Pengzhan Jin, Guofei Pang, Zhongqiang Zhang, and George~Em Karniadakis.
\newblock Learning nonlinear operators via deeponet based on the universal approximation theorem of operators.
\newblock {\em Nature Machine Intelligence}, 3(3):218--229, 2021.

\bibitem{lu2021deepxde}
Lu~Lu, Xuhui Meng, Zhiping Mao, and George~Em Karniadakis.
\newblock {DeepXDE:} a deep learning library for solving differential equations.
\newblock {\em SIAM Review}, 63(1):208--228, 2021.

\bibitem{mall2016application}
Susmita Mall and Snehashish Chakraverty.
\newblock Application of legendre neural network for solving ordinary differential equations.
\newblock {\em Applied Soft Computing}, 43:347--356, 2016.

\bibitem{mclean2000strongly}
William McLean.
\newblock {\em Strongly Elliptic Systems and Boundary Integral Equations}.
\newblock Cambridge University Press, 2000.

\bibitem{meng2024solving}
Bin Meng, Yutong Lu, and Ying Jiang.
\newblock Solving partial differential equations in different domains by operator learning method based on boundary integral equations.
\newblock {\em arXiv preprint arXiv:2406.02298}, 2024.

\bibitem{mikhlin2014integral}
Solomon~Grigorevich Mikhlin.
\newblock {\em Integral equations: and their applications to certain problems in mechanics, mathematical physics and technology}.
\newblock Elsevier, 2014.

\bibitem{nelsen2021random}
Nicholas~H Nelsen and Andrew~M Stuart.
\newblock The random feature model for input-output maps between banach spaces.
\newblock {\em SIAM Journal on Scientific Computing}, 43(5):A3212--A3243, 2021.

\bibitem{pang2019neural}
Guofei Pang, Liu Yang, and George~Em Karniadakis.
\newblock Neural-net-induced gaussian process regression for function approximation and pde solution.
\newblock {\em Journal of Computational Physics}, 384:270--288, 2019.

\bibitem{poppenberg2002existence}
Markus Poppenberg, Klaus Schmitt, and Zhi-Qiang Wang.
\newblock On the existence of soliton solutions to quasilinear schr{\"o}dinger equations.
\newblock {\em Calculus of Variations and Partial Differential Equations}, 14(3):329--344, 2002.

\bibitem{qu2023boundary}
Wenzhen Qu, Yan Gu, Shengdong Zhao, et~al.
\newblock Boundary integrated neural networks (binns) for acoustic radiation and scattering.
\newblock {\em arXiv preprint arXiv:2307.10521}, 2023.

\bibitem{raissi2019physics}
Maziar Raissi, Paris Perdikaris, and George~E Karniadakis.
\newblock Physics-informed neural networks: A deep learning framework for solving forward and inverse problems involving nonlinear partial differential equations.
\newblock {\em Journal of Computational physics}, 378:686--707, 2019.

\bibitem{rico1992discrete}
Ramiro Rico-Martinez, K~Krischer, IG~Kevrekidis, MC~Kube, and JL~Hudson.
\newblock Discrete-vs. continuous-time nonlinear signal processing of cu electrodissolution data.
\newblock {\em Chemical Engineering Communications}, 118(1):25--48, 1992.

\bibitem{sakakihara1987iterative}
J~Sakakihara.
\newblock An iterative boundary integral equation method for mildly nonlinear elliptic partial differential equations.
\newblock {\em Brebbia, CA (Springer-Verlag, NY)}, pages 13--49, 1987.

\bibitem{solodskikh2023integral}
Kirill Solodskikh, Azim Kurbanov, Ruslan Aydarkhanov, Irina Zhelavskaya, Yury Parfenov, Dehua Song, and Stamatios Lefkimmiatis.
\newblock Integral neural networks.
\newblock In {\em Proceedings of the IEEE/CVF Conference on Computer Vision and Pattern Recognition}, pages 16113--16122, 2023.

\bibitem{stiasny2024error}
Jochen Stiasny and Spyros Chatzivasileiadis.
\newblock Error estimation for physics-informed neural networks with implicit runge-kutta methods.
\newblock {\em arXiv preprint arXiv:2401.05211}, 2024.

\bibitem{toscano2025kkans}
Juan~Diego Toscano, Li-Lian Wang, and George~Em Karniadakis.
\newblock Kkans: Kurkova-kolmogorov-arnold networks and their learning dynamics.
\newblock {\em Neural Networks}, 191:107831, 2025.

\bibitem{wan2004thermo}
Yu-Qin Wan, Qian Guo, and Ning Pan.
\newblock Thermo-electro-hydrodynamic model for electrospinning process.
\newblock {\em International Journal of Nonlinear Sciences and Numerical Simulation}, 5(1):5--8, 2004.

\bibitem{wang2022and}
Sifan Wang, Xinling Yu, and Paris Perdikaris.
\newblock When and why pinns fail to train: A neural tangent kernel perspective.
\newblock {\em Journal of Computational Physics}, 449:110768, 2022.

\bibitem{yu2020hyper}
Tong Yu and Hong Zhu.
\newblock Hyper-parameter optimization: A review of algorithms and applications.
\newblock {\em arXiv preprint arXiv:2003.05689}, 2020.

\bibitem{yuan2022pinn}
Lei Yuan, Yi-Qing Ni, Xiang-Yun Deng, and Shuo Hao.
\newblock A-pinn: Auxiliary physics informed neural networks for forward and inverse problems of nonlinear integro-differential equations.
\newblock {\em Journal of Computational Physics}, 462:111260, 2022.

\bibitem{zappala2022neural}
Emanuele Zappala, Antonio Henrique de~Oliveira Fonseca, Josue~Ortega Caro, Andrew~Henry Moberly, Michael~James Higley, Jessica Cardin, and David van Dijk.
\newblock Neural integral equations.
\newblock {\em arXiv preprint arXiv:2209.15190}, 2022.

\bibitem{zhang2024blending}
Enrui Zhang, Adar Kahana, Alena Kopani{\v{c}}{\'a}kov{\'a}, Eli Turkel, Rishikesh Ranade, Jay Pathak, and George~Em Karniadakis.
\newblock Blending neural operators and relaxation methods in pde numerical solvers.
\newblock {\em Nature Machine Intelligence}, pages 1--11, 2024.

\end{thebibliography}

\end{document}